%% file: soergelN.tex
\documentclass[11pt]{amsart}
\usepackage{
  amsmath,
  amsfonts,
  amssymb,
  graphicx, 
  times,
  color,
  hyphenat,
  pinlabel,
  mvsymb
}
\input xy
\xyoption{all}
\input defs.tex

\title{The diagrammatic Soergel category and $sl(N)$-foams, for $N\geq 4$}

\author{Marco Mackaay}
\address{Departamento de Matem\'{a}tica\\ Universidade do Algarve\\ 
Campus de Gambelas\\ 8005-139 Faro\\ Portugal and CAMGSD\\Instituto Superior T\'{e}cnico\\ 
Avenida Rovisco Pais\\ 
1049-001 Lisboa\\ Portugal}
\email{mmackaay@ualg.pt}

\author{Pedro Vaz}
\address{Institut de Math\'ematiques de Jussieu \\ Universit\'e Paris 7 \\   
 175 Rue du Chevaleret, F-75013 Paris, France and  
CAMGSD\\Instituto Superior T\'{e}cnico\\ Avenida Rovisco Pais\\ 1049-001 Lisboa\\ Portugal}
\email{pfortevaz@ualg.pt}

\begin{document}
%
%
\newdimen\captionwidth\captionwidth=\hsize
%
%
\begin{abstract} 
For each $N\geq 4$, we define a monoidal functor from Elias and Khovanov's 
diagrammatic version of Soergel's category of bimodules to the category of 
$sl(N)$ foams defined by Mackaay, Sto\v{s}i\'c and Vaz. We show that 
through these functors Soergel's category can be obtained from the $sl(N)$ 
foams.  
\end{abstract}
\maketitle
%
%
\section{Introduction}\label{sec:intro}

In \cite{Soergel} Soergel categorified the Hecke algebra using bimodules. Just 
as the Hecke algebra is important for the construction of the HOMFLY-PT link 
polynomial, so is Soergel's category for the construction of 
Khovanov and Rozansky's HOMFLY-PT link homology~\cite{KhR2}, 
as explained by Khovanov in ~\cite{Kh2}. Elias and Khovanov~\cite{EKh} 
constructed a diagrammatic version of the Soergel category with generators and 
relations, which Elias and Krasner~\cite{EKr} used for a diagrammatic 
construction of Rouquier's complexes associated to braids.  

In \cite{B-N} Bar-Natan gave a new version of Khovanov's~\cite{Kh} 
original link homology, also called the $sl(2)$ link homology, 
using 2d-cobordisms modulo certain relations, which we will call $sl(2)$ 
foams. Using 2d-cobordisms with a particular sort of singularity modulo 
certain relations, which we will call $sl(3)$ foams, 
Khovanov constructed the $sl(3)$ link homology~\cite{Kh1}. Khovanov 
and Rozansky~\cite{KhR1} then constructed the $sl(N)$ link homologies, 
for any $N\geq 1$, using matrix factorizations. These link homologies are 
closely related to the HOMFLY-PT link homology by Rasmussen's spectral 
sequences~\cite{Rasm}, with $E_1$-page isomorphic to the HOMFLY-PT 
homology and converging to the $sl(N)$ homology, for any $N\geq 1$. 
In \cite{MSV1} Mackaay, Sto\v{s}i\'c and Vaz gave an alternative construction 
of these $sl(N)$ link homologies, for $N\geq 4$, using $sl(N)$ foams, which are 
2d-cobordisms with two types of singularities satisfying relations determined 
by a formula from quantum field theory, originally obtained by 
Kapustin and Li~\cite{KL} and later adapted by Khovanov and 
Rozansky~\cite{KhR3}.

Khovanov and Rozansky in \cite{KhR1} and \cite{KhR2} and Rasmussen in 
\cite{Rasm} used matrix factorizations for their constructions. Therefore, the 
question arises whether their results can be understood in diagrammatic terms 
and what could be learned from that. In~\cite{V1} Vaz constructed functors 
from Elias and Khovanov's diagrammatic version of Soergel's category to the 
categories of $sl(2)$ and $sl(3)$ foams. 
In this paper we construct the analogous functors from the same version of 
Soergel's category to the category of $sl(N)$ foams for 
$N\geq 4$. To complete the picture, one would like to construct 
the analogues of Rasmussen's spectral sequences in this setting. But 
for that, one would first have to understand 
the Hochschild homology of bimodules in diagrammatic terms, which has not been 
accomplished yet. However for this, one would first have to understand the Hochschild 
homology of bimodules in diagrammatic terms, which has not been accomplished yet.
Hochschild homology plays an integral part of the construction.
Nevertheless, there is an interesting result which can already be shown using the 
functors in this paper. In a certain technical sense, which we will make 
precise in Proposition~\ref{prop:faithfulness}, Soergel's category can be 
obtained from the $sl(N)$ foams, and therefore from the Kapustin-Li formula, 
using our functors. This result should be compared to Rasmussen's Theorem 1 
in~\cite{Rasm}.       
 
We have tried to make the paper as self-contained as possible, but the reader 
should definitely leaf through \cite{EKh}, \cite{EKr}, \cite{MSV1} and 
\cite{V1} before reading the rest of this paper.

In Section 2 we recall Elias and Khovanov's version of Soergel's category. 
In Section 3 we review $sl(N)$ foams, as 
defined by Mackaay, Sto\v{s}i\'{c} and Vaz. Section 4 contains the new results: 
the definition of our functors, the proof that they are 
indeed monoidal, and a statement on faithfulness in Proposition~\ref{prop:faithfulness}.

\section{Elias and Khovanov's version of Soergel's category}
\label{sec:soergel}

This section is a reminder of the diagrammatics for Soergel categories introduced by 
Elias and Khovanov in~\cite{EKh}. 
Actually we give the version which they explained in Section 4.5 of~\cite{EKh} and which 
can be found in detail in ~\cite{EKr}.

Fix a positive integer $n$. The category $\mathcal{SC}_1$ is the category whose objects 
are finite length sequences of points on the real line, where each point is colored by an 
integer between $1$ and $n$. 
We read sequences of points from left to right. 
Two colors $i$ and $j$ are called adjacent if $\vert i-j\vert=1$ and distant if 
$\vert i-j\vert >1$.  
The morphisms of $\mathcal{SC}_1$ are given by generators modulo relations. 
A morphism of $\mathcal{SC}_1$ is a $\bC$-linear combination of planar diagrams 
constructed by horizontal and vertical gluings of the following generators 
(by convention no label means a generic color $j$):
\begin{itemize}
\item Generators involving only one color:
\begin{equation*}
\xymatrix@R=1.0mm{
\figins{-15}{0.5}{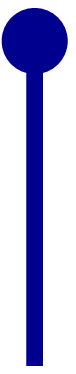}
&
\figins{-9}{0.5}{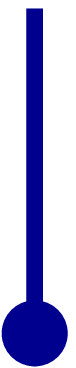}
&
\figins{-15}{0.55}{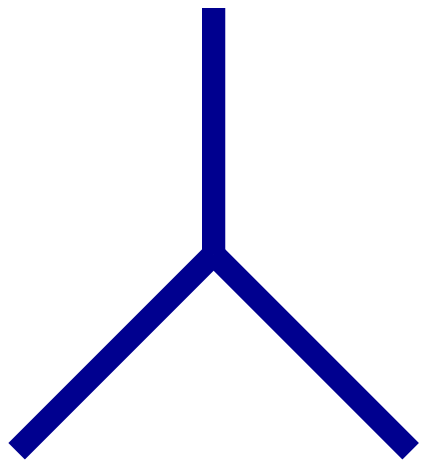}
&
\figins{-15}{0.55}{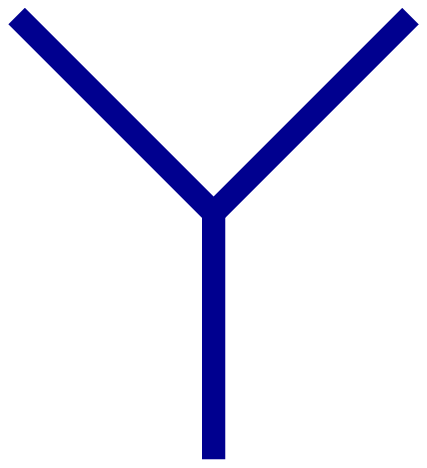}
\\
\text{EndDot} & \text{StartDot} & \text{Merge} & \text{Split}
}
\end{equation*}

\medskip

It is useful to define the cap and cup as
\begin{equation*}
\figins{-17}{0.55}{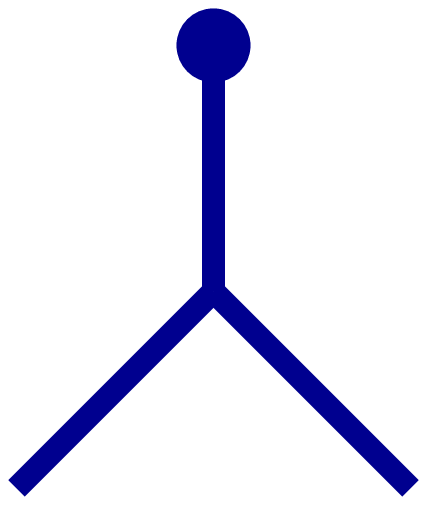}\
\equiv\
\figins{-17}{0.55}{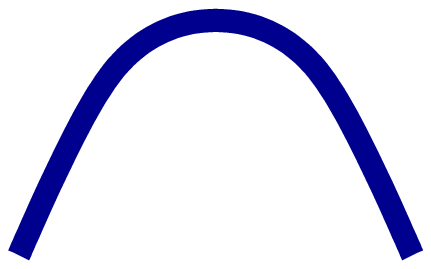}\
\mspace{50mu}
\figins{-17}{0.55}{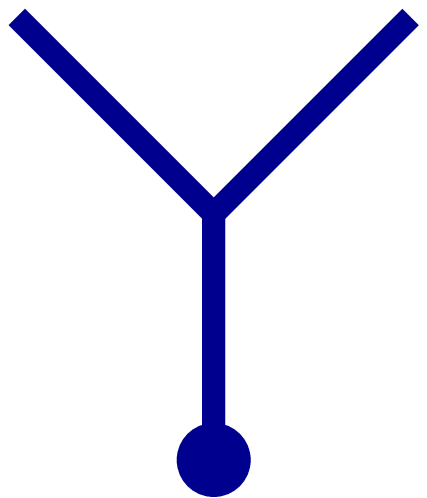}\
\equiv\
\figins{-17}{0.55}{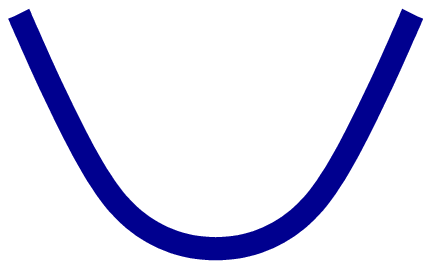}
\end{equation*}

\medskip

\item Generators involving two colors:
\begin{itemize}
\item The 4-valent vertex, with distant colors,
\begin{equation*}
\labellist
\tiny\hair 2pt
\pinlabel $i$ at  -4 -10
\pinlabel $j$ at 134 -10
\endlabellist
\figins{-15}{0.55}{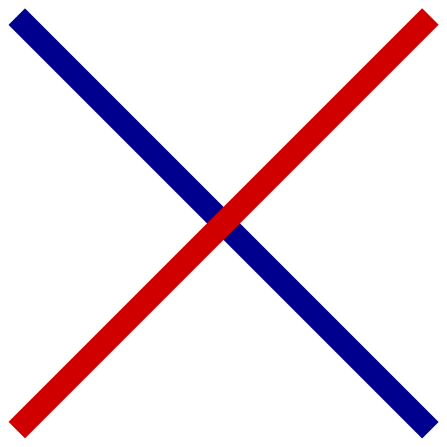}\vspace{1.5ex}
\end{equation*}
\item and the 6-valent vertex, with adjacent colors $i$ and $j$
\begin{equation*}
\labellist
\tiny\hair 2pt
\pinlabel $i$   at  -4 -10 
\pinlabel $j$ at  66 -12
\pinlabel $j$ at 232 -12
\pinlabel $i$   at 298 -10
\endlabellist
\figins{-17}{0.55}{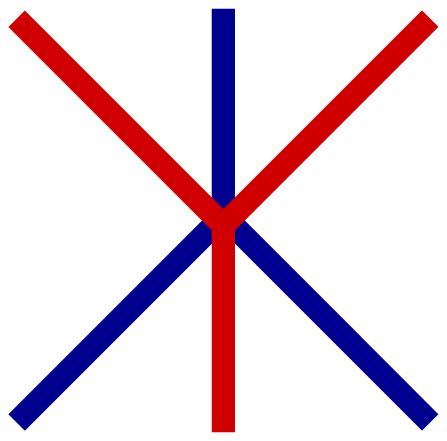}
\mspace{55mu}
\figins{-17}{0.55}{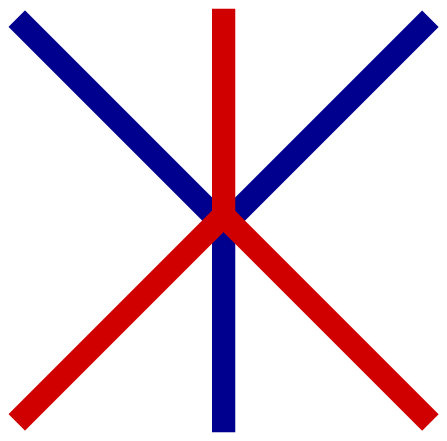}\ .
\vspace{1.5ex}
\end{equation*}
\end{itemize}
\end{itemize}

\n read from bottom to top. In this setting a diagram represents a morphism from the 
bottom bounda\-ry to the top.
We can add a new colored point to a sequence and this endows $\mathcal{SC}_1$ with a 
monoidal structure on objects, which is extended to morphisms in the obvious way. 
Composition of morphisms consists of stacking one diagram on top of the other.

We consider our diagrams modulo the following relations.

\n\emph{``Isotopy'' relations:}
\begin{equation}\label{eq:adj}
\figins{-17}{0.55}{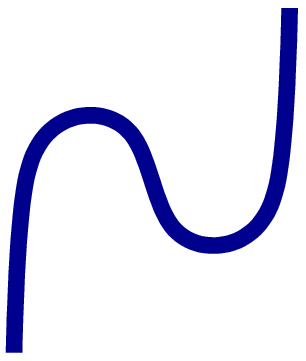}\
=\
\figins{-17}{0.55}{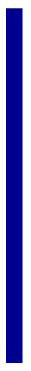}\
=\
\figins{-17}{0.55}{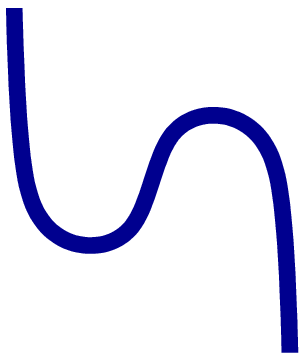}
\end{equation}

\begin{equation}\label{eq:curldot}
\figins{-17}{0.55}{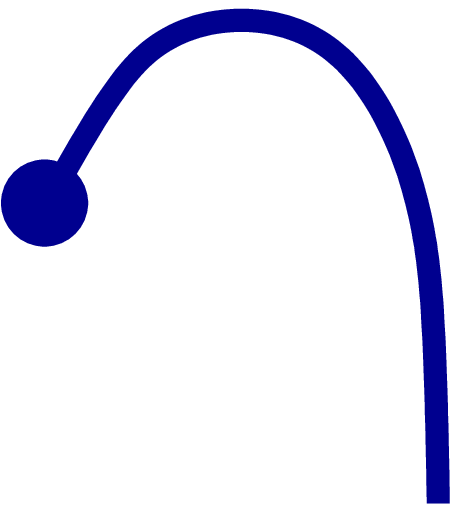}\
=\
\figins{-17}{0.55}{enddot}\
=\
\figins{-17}{0.55}{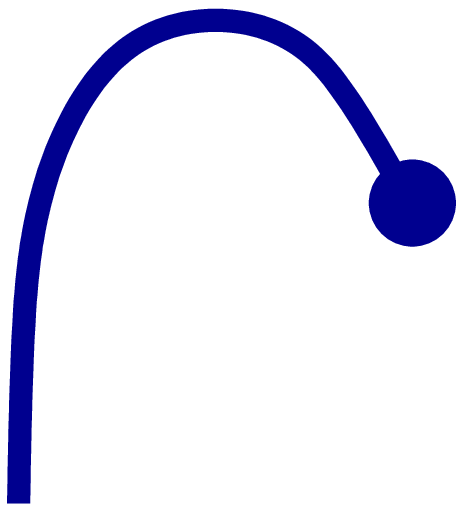}
\end{equation}

\begin{equation}\label{eq:v3rot}
\figins{-17}{0.55}{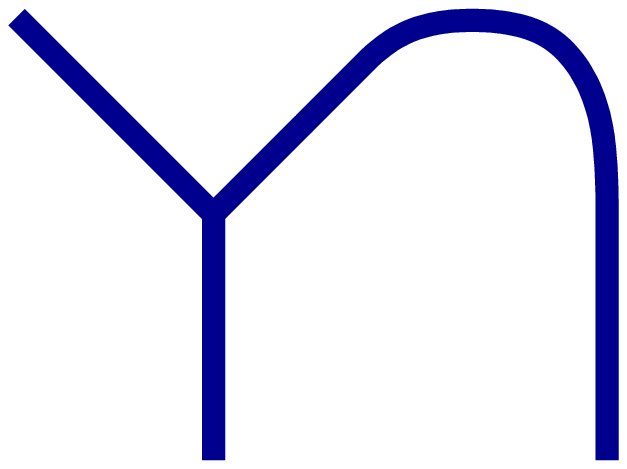}\
=\
\figins{-17}{0.55}{merge}\
=\
\figins{-17}{0.55}{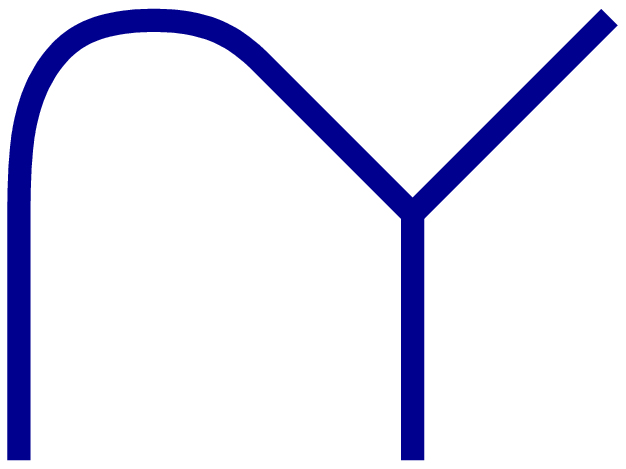}
\end{equation}

\begin{equation}\label{eq:v4rot}
\figins{-17}{0.55}{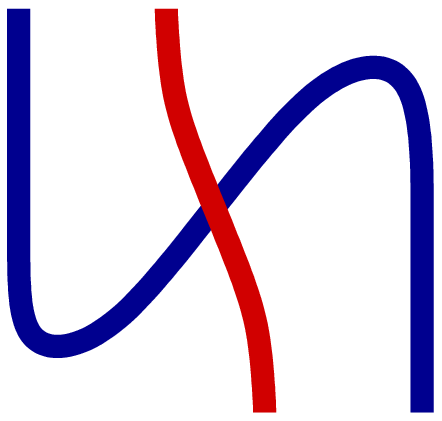}\
=\
\figins{-17}{0.55}{4vert}\
=\
\figins{-17}{0.55}{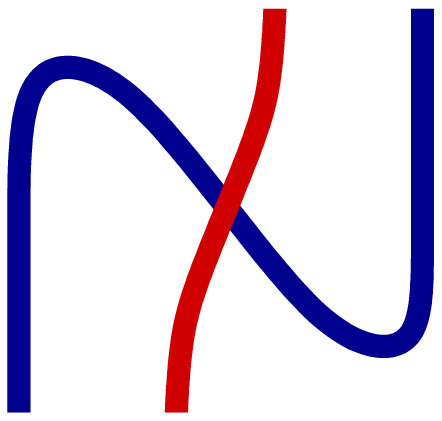}
\end{equation}

\begin{equation}\label{eq:v6rot}
\figins{-17}{0.55}{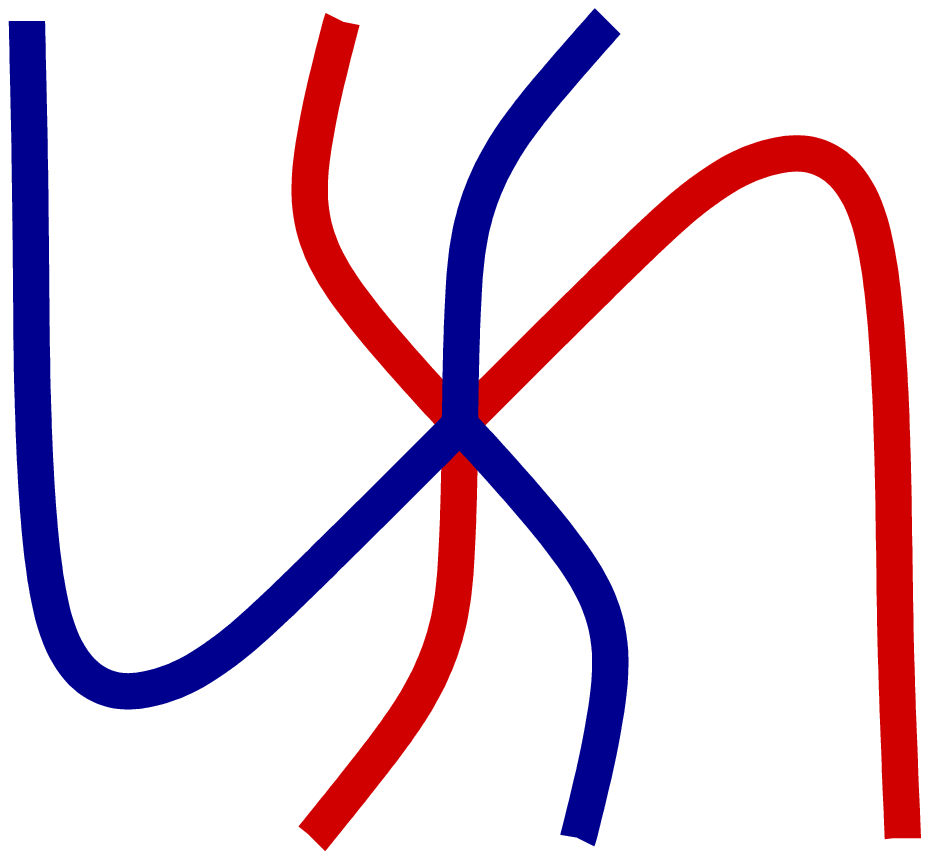}\
=\
\figins{-17}{0.55}{6vertu}\
=\
\figins{-17}{0.55}{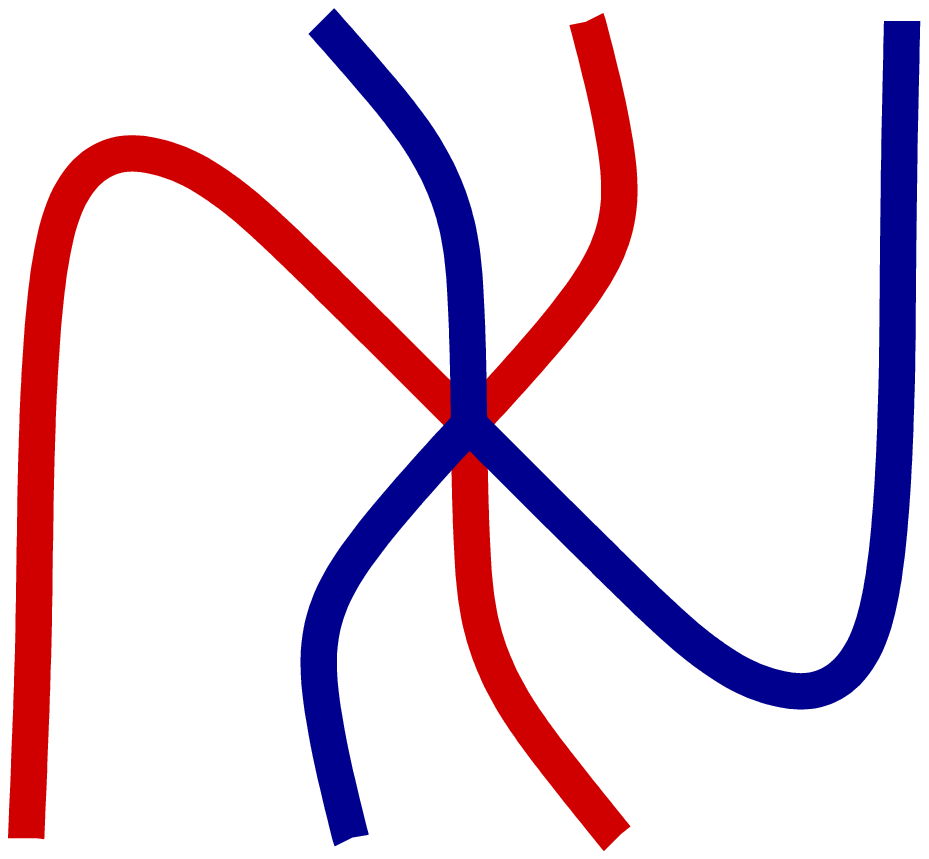}
\end{equation}

The relations are presented in terms of diagrams with generic colorings. 
Because of isotopy invariance, one may draw a diagram with a boundary on the side,
and view it as a morphism in $\mathcal{SC}_1$ by either bending the line up or down.
By the same reasoning, a horizontal line corresponds to a sequence of cups and caps.

\medskip
\n\emph{One color relations:}

\begin{equation}\label{eq:dumbrot}
\figins{-16}{0.5}{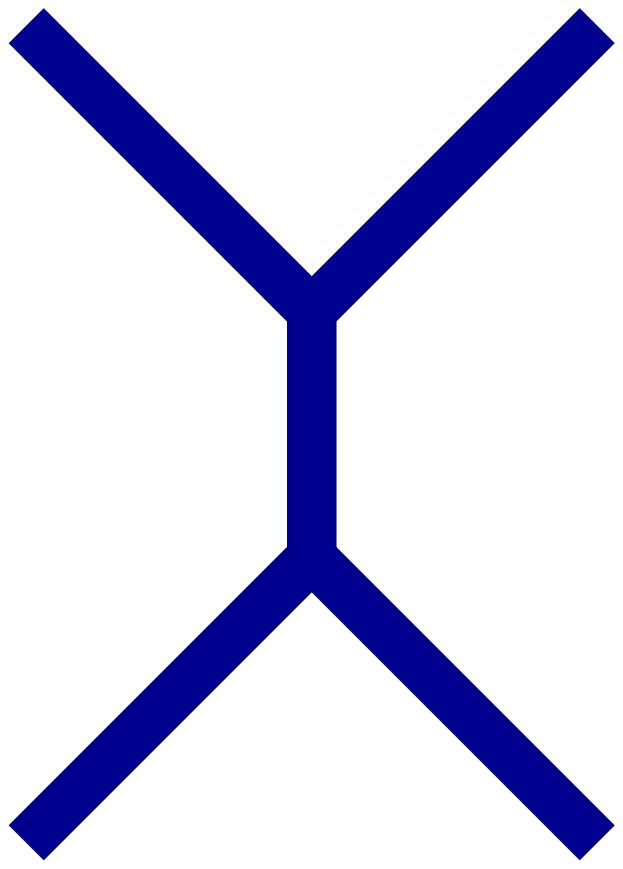}\
=\
\figins{-14}{0.45}{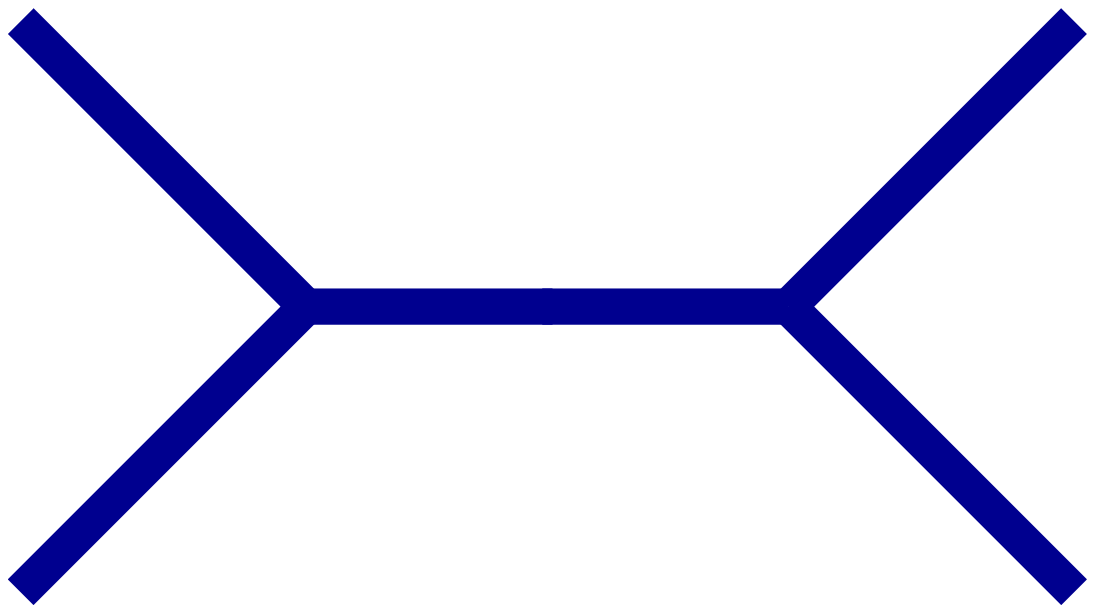}
\end{equation}

\begin{equation}\label{eq:lollipop}
\figins{-17}{0.55}{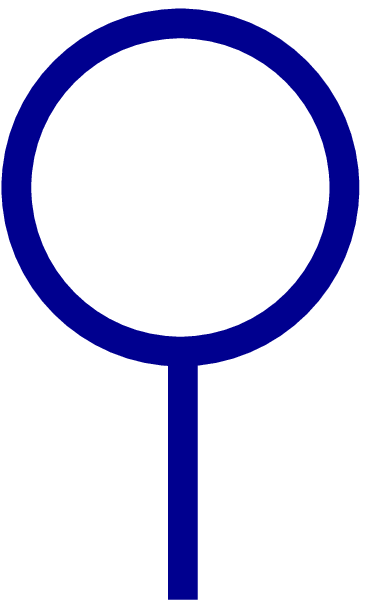}\
=\
0
\end{equation}

\begin{equation}\label{eq:deltam}
\figins{-17}{0.55}{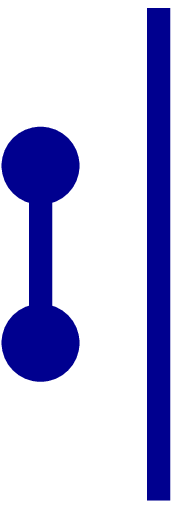}\
+\
\figins{-17}{0.55}{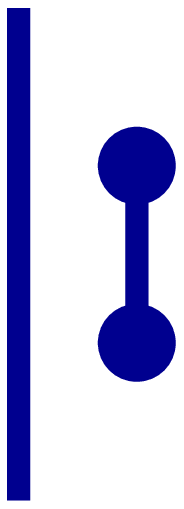}\
=\ 2\ 
\figins{-17}{0.55}{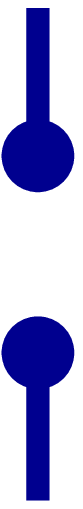}
\end{equation}

\medskip
\n\emph{Two distant colors:}
\begin{equation}\label{eq:reid2dist}
\figins{-32}{0.9}{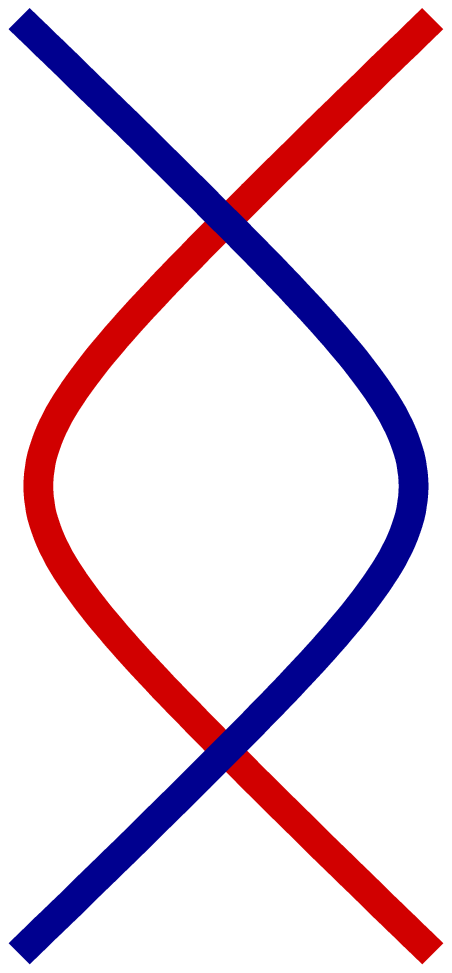}\
=\
\figins{-32}{0.9}{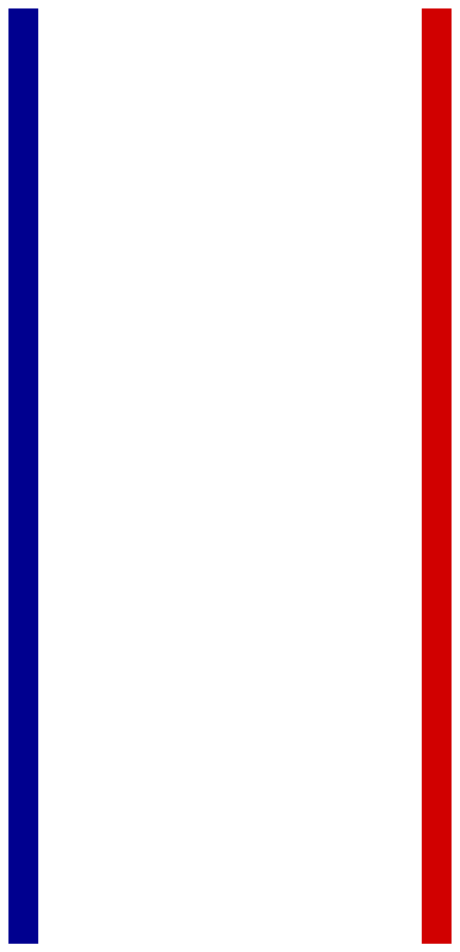}
\end{equation}

\begin{equation}\label{eq:slidedotdist}
\figins{-16}{0.5}{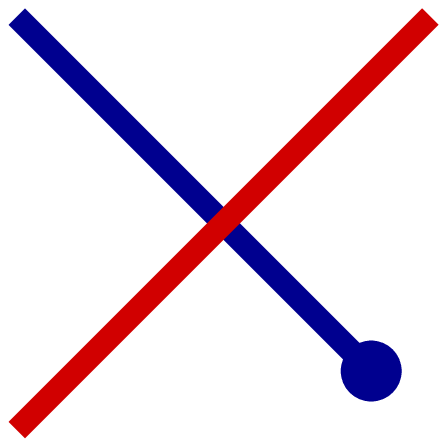}\
=\
\figins{-16}{0.5}{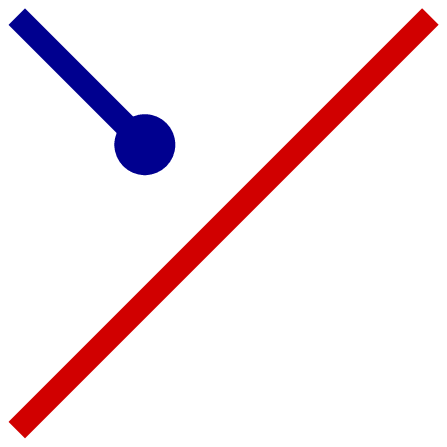}
\end{equation}

\begin{equation}\label{eq:slide3v}
\figins{-17}{0.55}{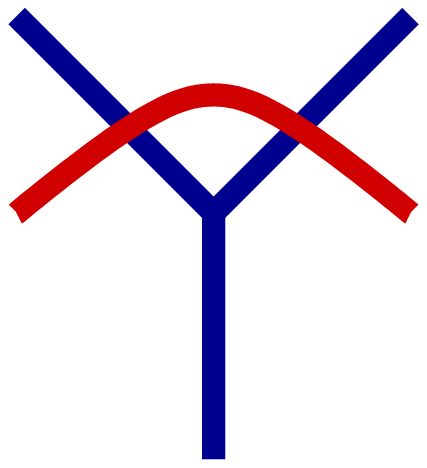}\
=\
\figins{-17}{0.55}{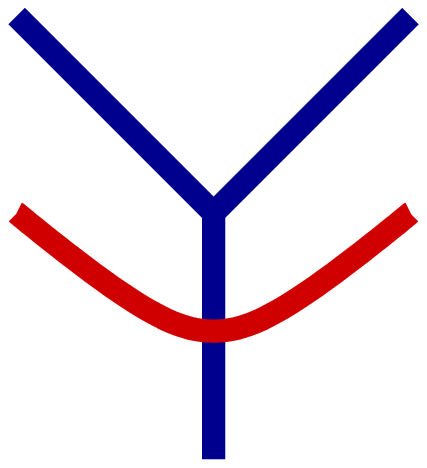}
\end{equation}

\medskip
\n\emph{Two adjacent colors:}
\begin{equation}\label{eq:dot6v}
\figins{-16}{0.5}{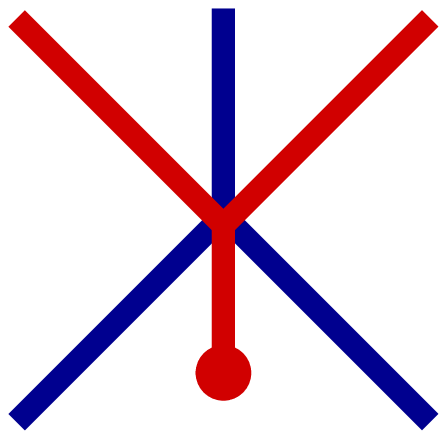}\
=\
\figins{-16}{0.5}{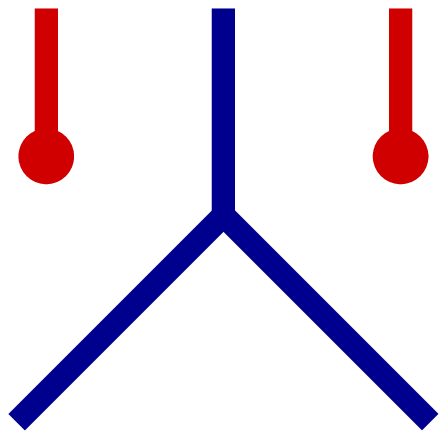}\
+\
\figins{-16}{0.5}{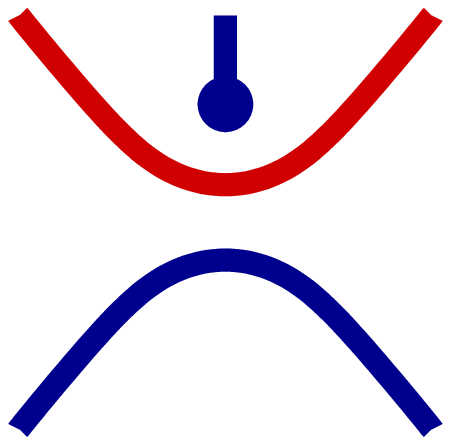}
\end{equation}

\begin{equation}\label{eq:reid3}
\figins{-30}{0.85}{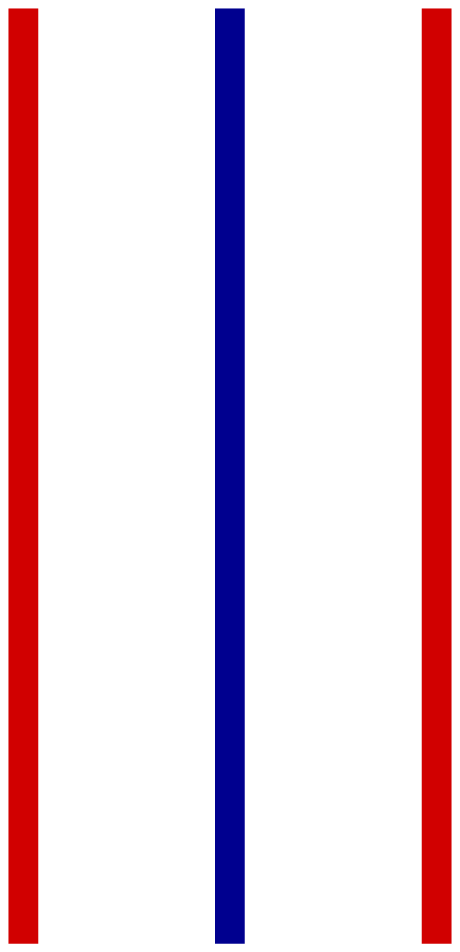}\
=\
\figins{-30}{0.85}{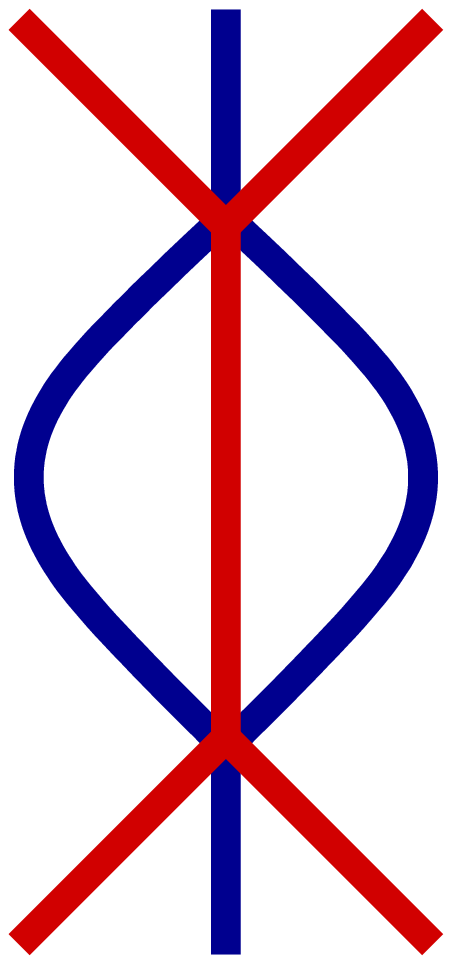}\
-\
\figins{-30}{0.85}{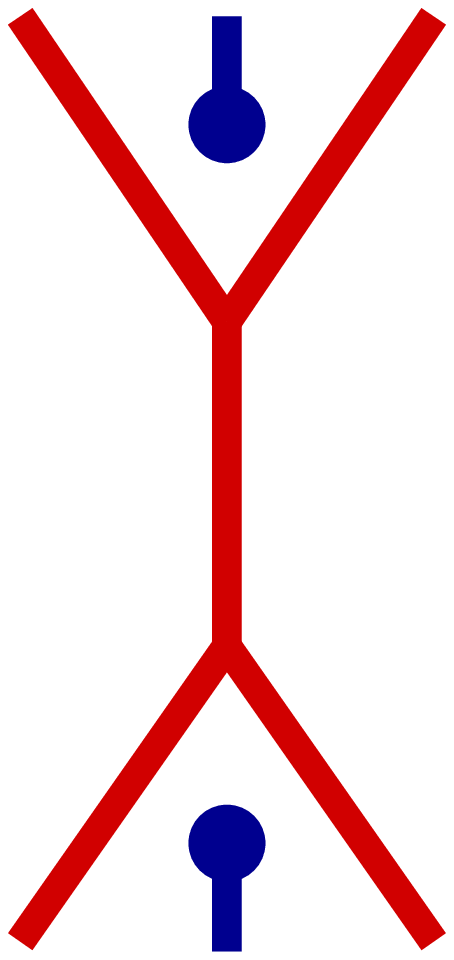}
\end{equation}

\begin{equation}\label{eq:dumbsq}
\figins{-30}{0.85}{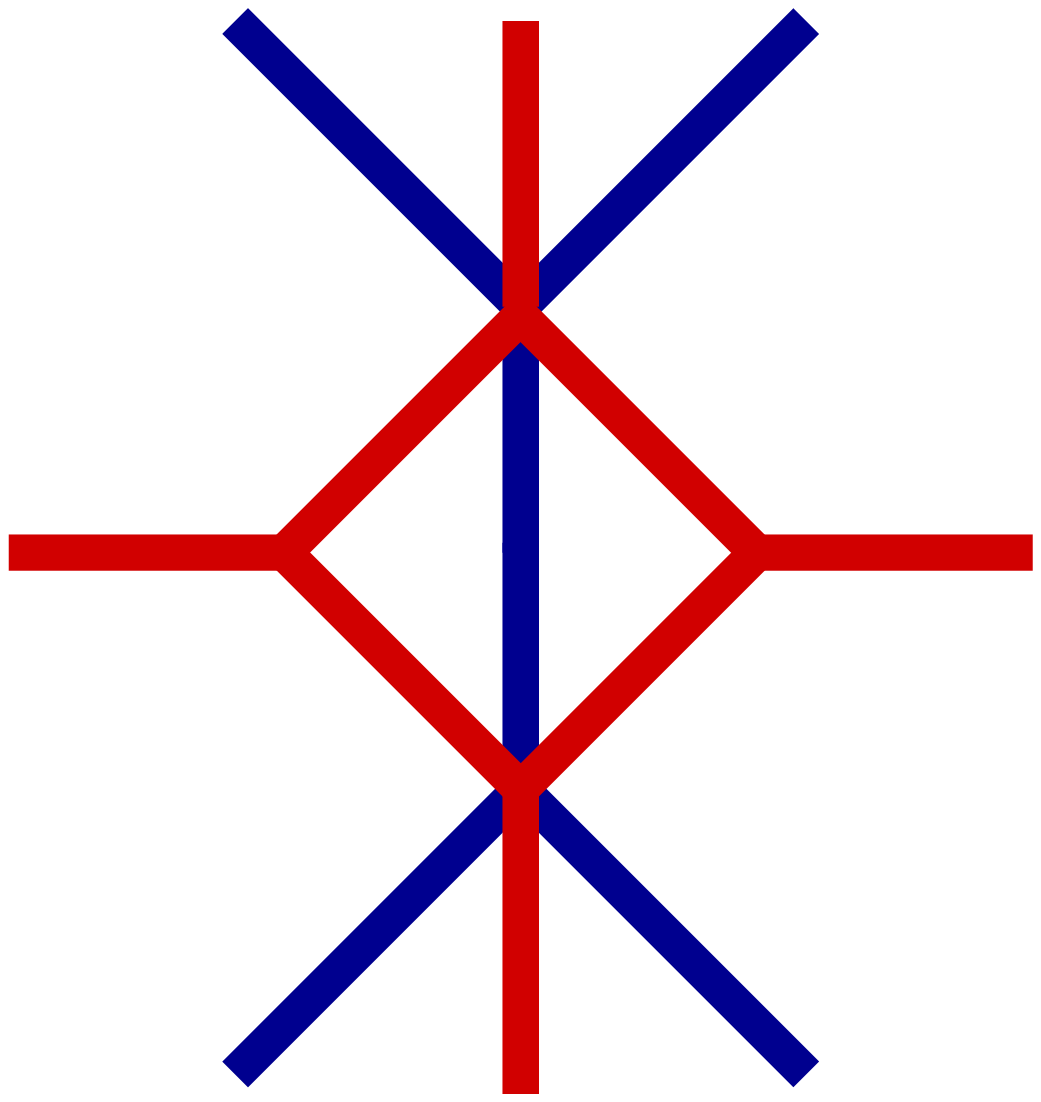}\
=\
\figins{-30}{0.85}{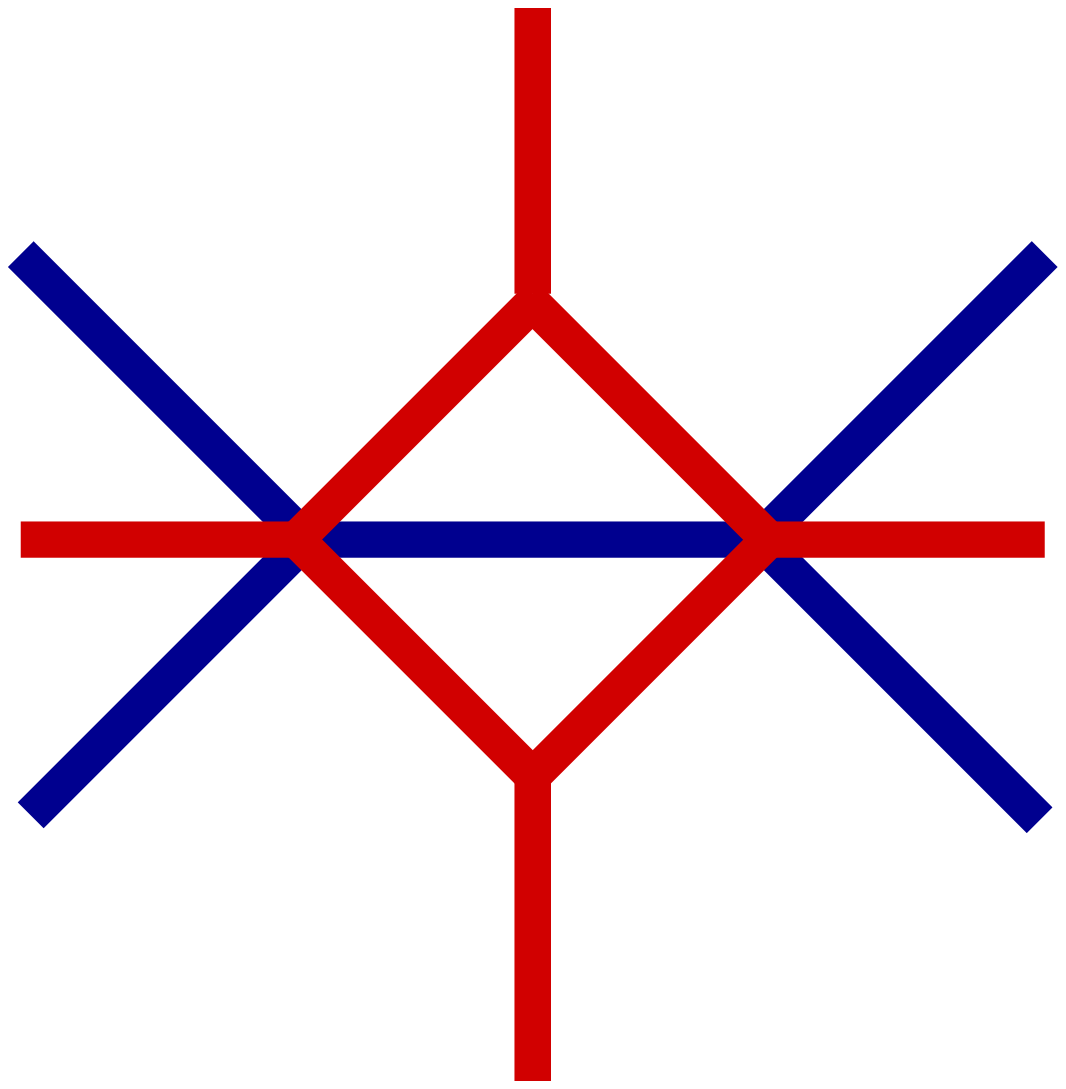}
\end{equation}

\begin{equation}\label{eq:slidenext}
\figins{-17}{0.55}{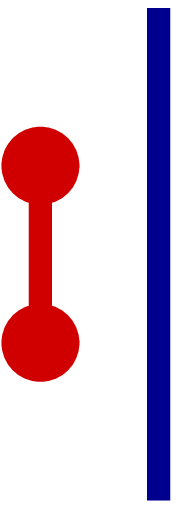}\
-\
\figins{-17}{0.55}{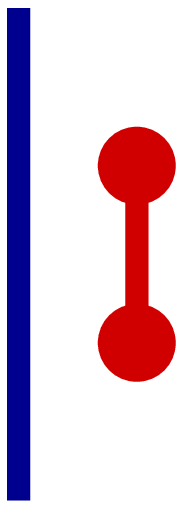}\
=\
\frac{1}{2}
\Biggl(\
\figins{-17}{0.55}{edge-startenddot}\
-\
\figins{-17}{0.55}{startenddot-edge}\
\Biggr)
\end{equation}

\medskip
\n\emph{Relations involving three colors:}
(adjacency is determined by the vertices which appear)
\begin{equation}\label{eq:slide6v}
\figins{-18}{0.6}{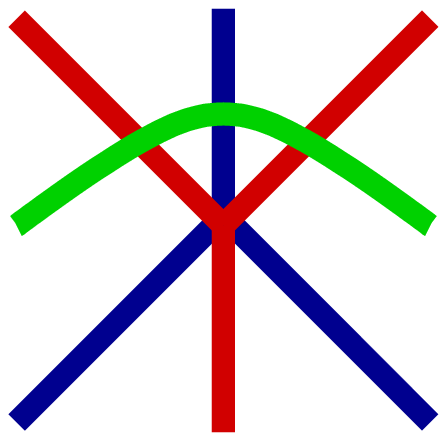}\
=\
\figins{-18}{0.6}{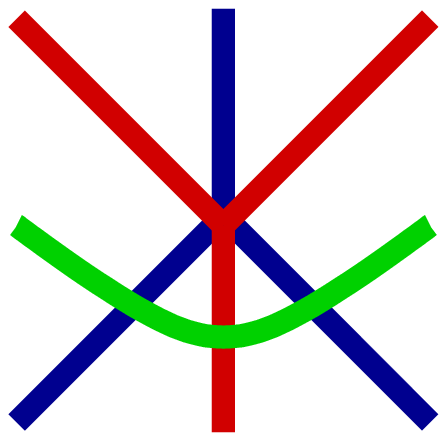}
\end{equation}

\begin{equation}\label{eq:slide4v}
\figins{-18}{0.6}{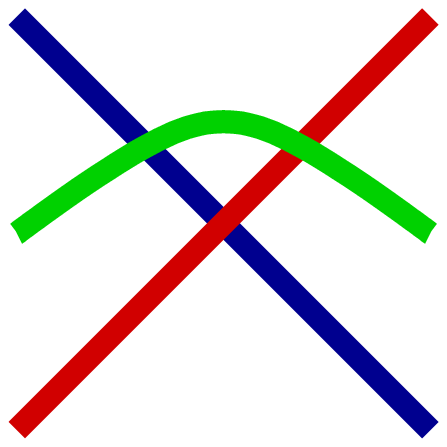}\
=\
\figins{-18}{0.6}{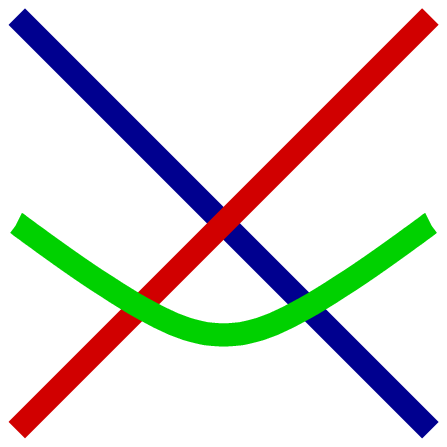}
\end{equation}

\begin{equation}\label{eq:dumbdumbsquare}
\figins{-30}{0.85}{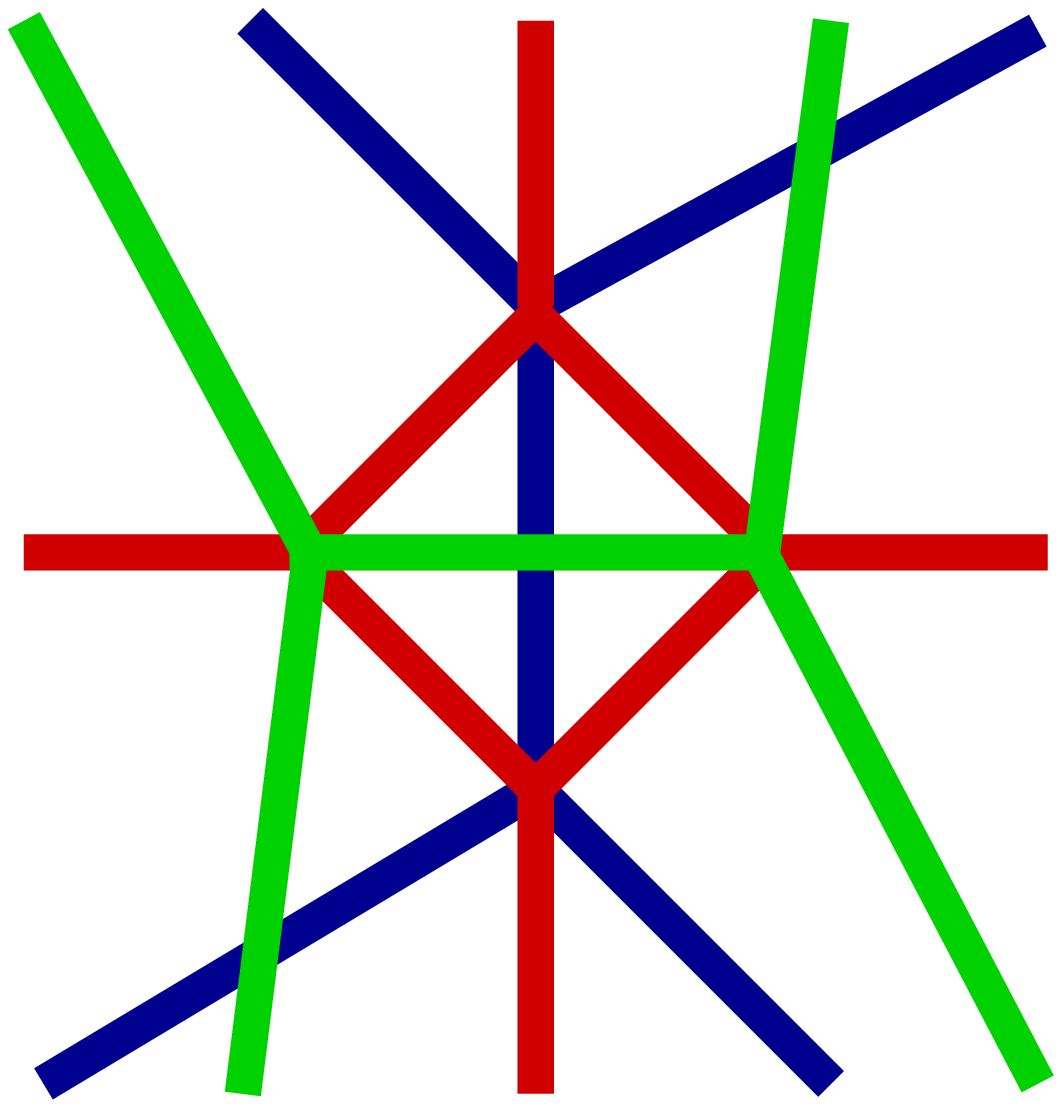}\
=\
\figins{-30}{0.85}{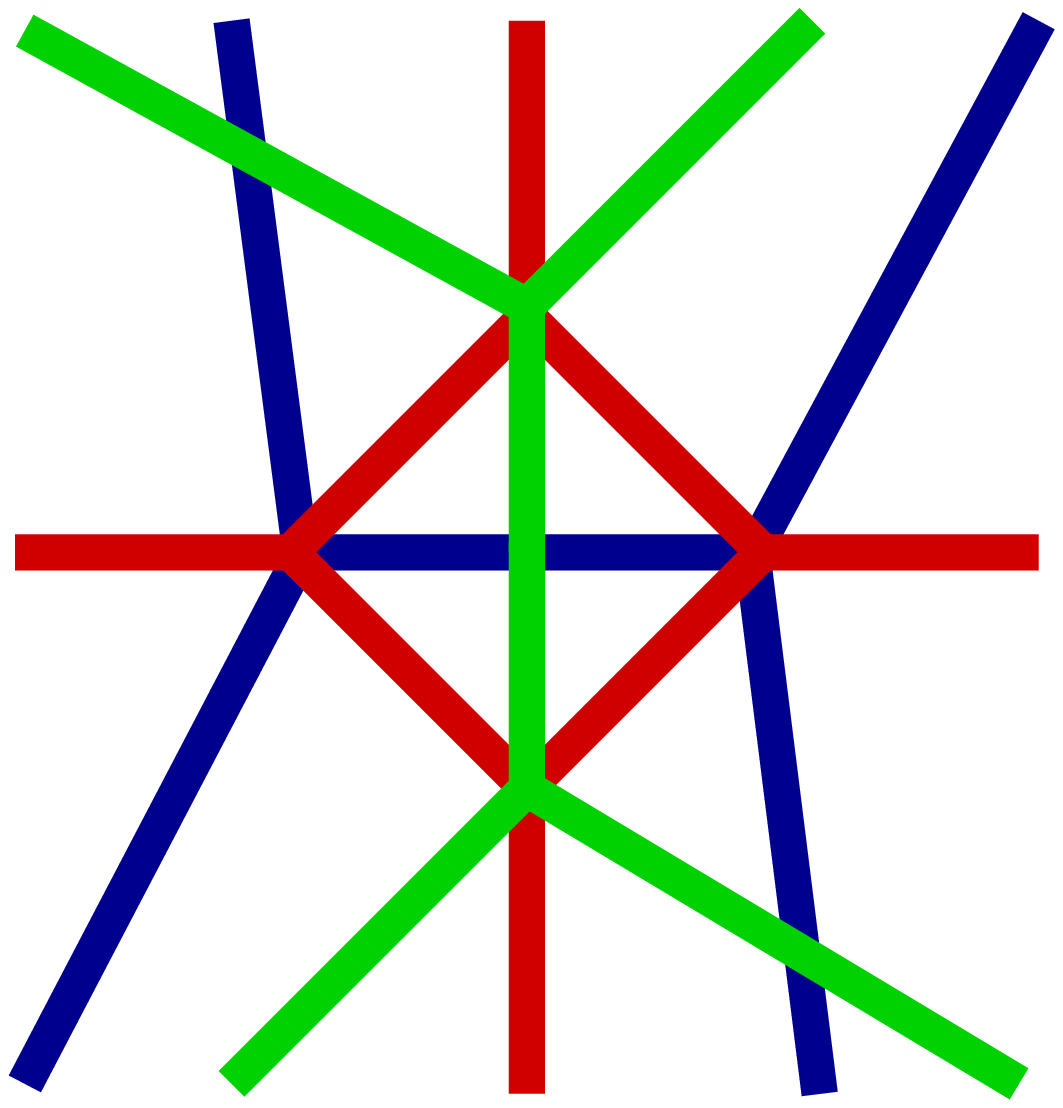}
\end{equation}

Furthermore, we also have a useful implication of relation~\eqref{eq:deltam}

\begin{equation}\label{eq:reid2}
\figins{-17}{0.55}{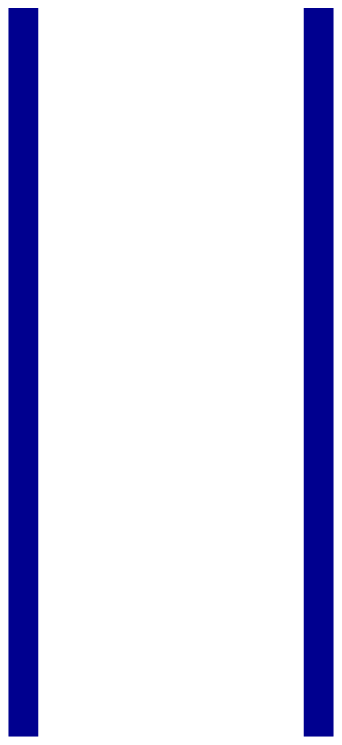}\
=\ \frac{1}{2}\ \Biggl(\ 
\figins{-17}{0.55}{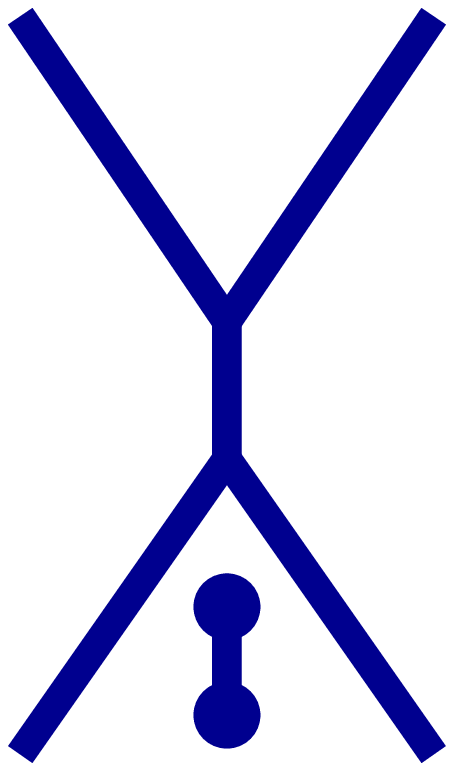}\
+\
\figins{-17}{0.55}{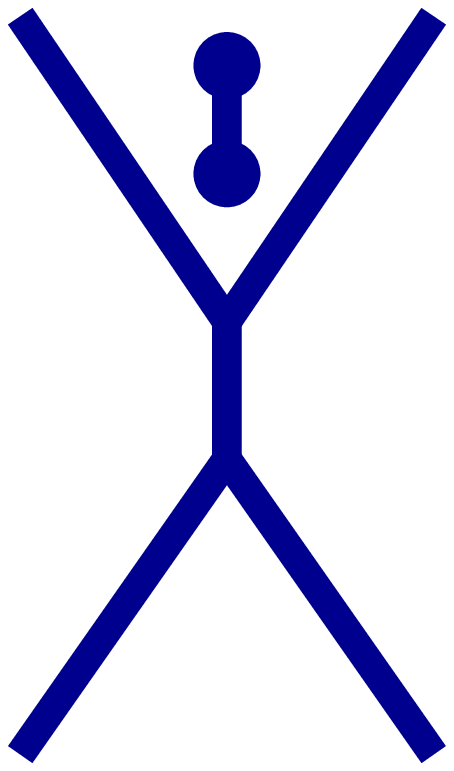}\
\Biggr)
\end{equation}


Introduce a $q$-grading on $\mathcal{SC}_1$ declaring that dots have degree $1$, 
trivalent vertices have degree $-1$ 
and $4$- and $6$-valent vertices have degree $0$.

\begin{defn}
The category $\mathcal{SC}_2$ is the category containing all direct sums and grading 
shifts of objects in 
$\mathcal{SC}_1$ and whose morphisms are the grading preserving morphisms from 
$\mathcal{SC}_1$.
\end{defn}

\begin{defn}
The category $\mathcal{SC}$ is the Karoubi envelope of the category $\mathcal{SC}_2$.
\end{defn}

Elias and Khovanov's main result in~\cite{EKh} is the following theorem.
\begin{thm}[Elias-Khovanov]
The category $\mathcal{SC}$ is equivalent to the Soergel category in~\cite{Soergel}.
\end{thm}

From Soergel's results from~\cite{Soergel} we have the following corollary.
\begin{cor}
The Grothendieck algebra of $\mathcal{SC}$ is isomorphic to the Hecke algebra.
\end{cor}

Notice that $\mathcal{SC}$ is an additive category but not abelian and we are using the 
(additive) split Grothendieck algebra.

In Subsection~\ref{ssec:funct} we will define a a family of functors from $\SSoer$ to the category of 
$sl(N)$ foams, one for each $N\geq 4$. 
These functors are grading preserving, so they obviously extend uniquely to $\SSSoer$. 
By the universality of the Karoubi envelope, they also extend uniquely to functors between the respective 
Karoubi envelopes.

\section{Foams}
\subsection{Pre-foams}
\label{sec:pre-foam}

In this section we recall the basic facts about foams. For 
the definition of the Kapustin-Li formula, for proofs of the relations between 
foams and for other details see~\cite{MSV1} and~\cite{V}. The foams in this paper 
are composed of three types of facets: simple, double and 
triple facets. The double facets are coloured and the triple facets are marked 
to show the difference.
\begin{figure}[h]
\figins{0}{0.7}{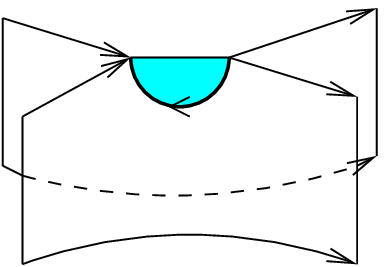} \qquad
\figins{0}{0.7}{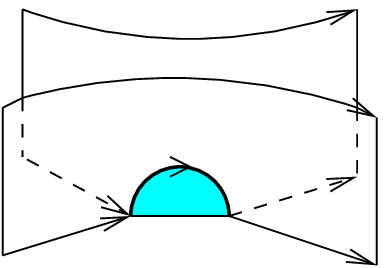} \qquad
\figins{-2}{0.8}{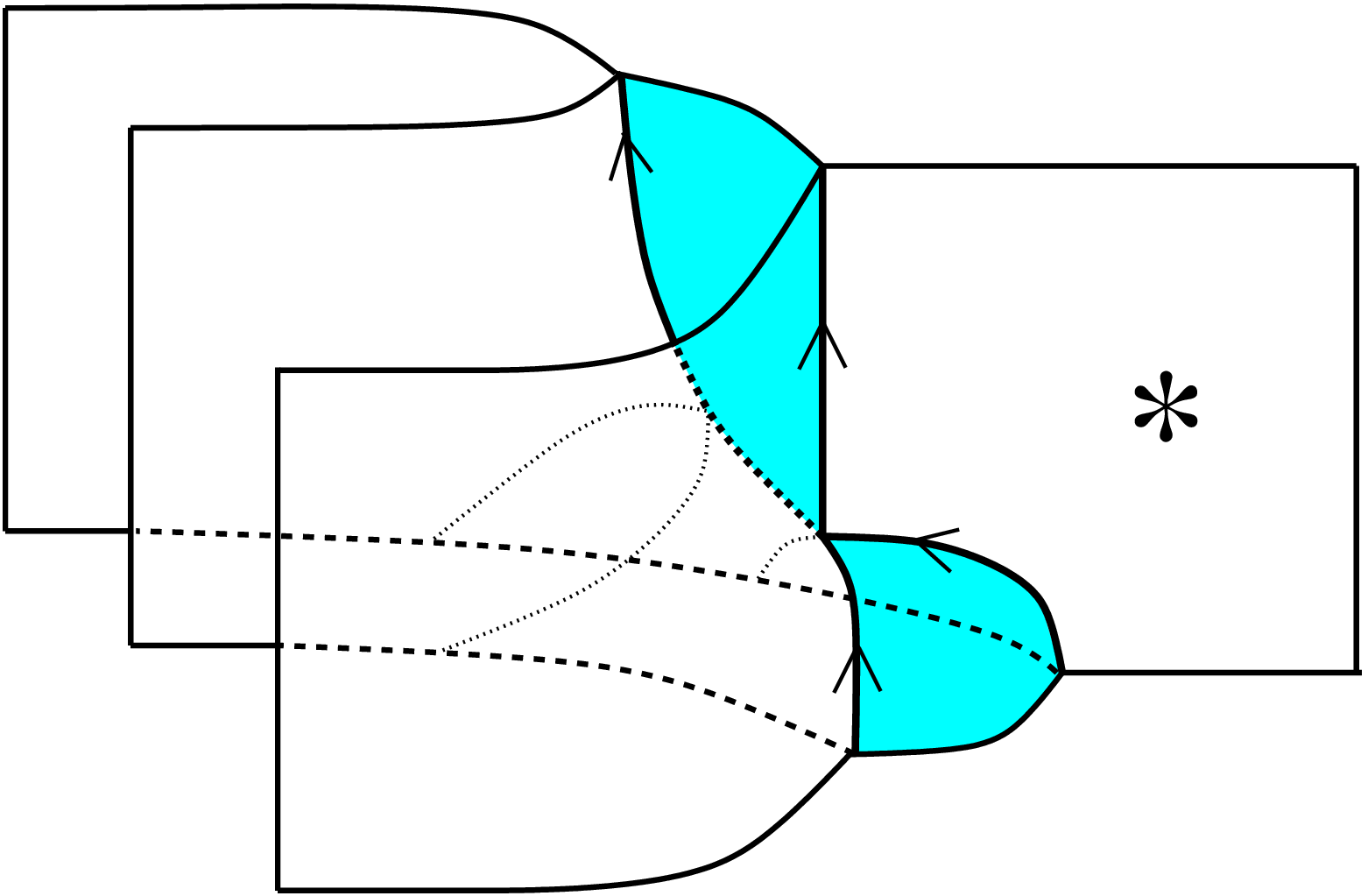} \qquad
\figins{-2}{0.8}{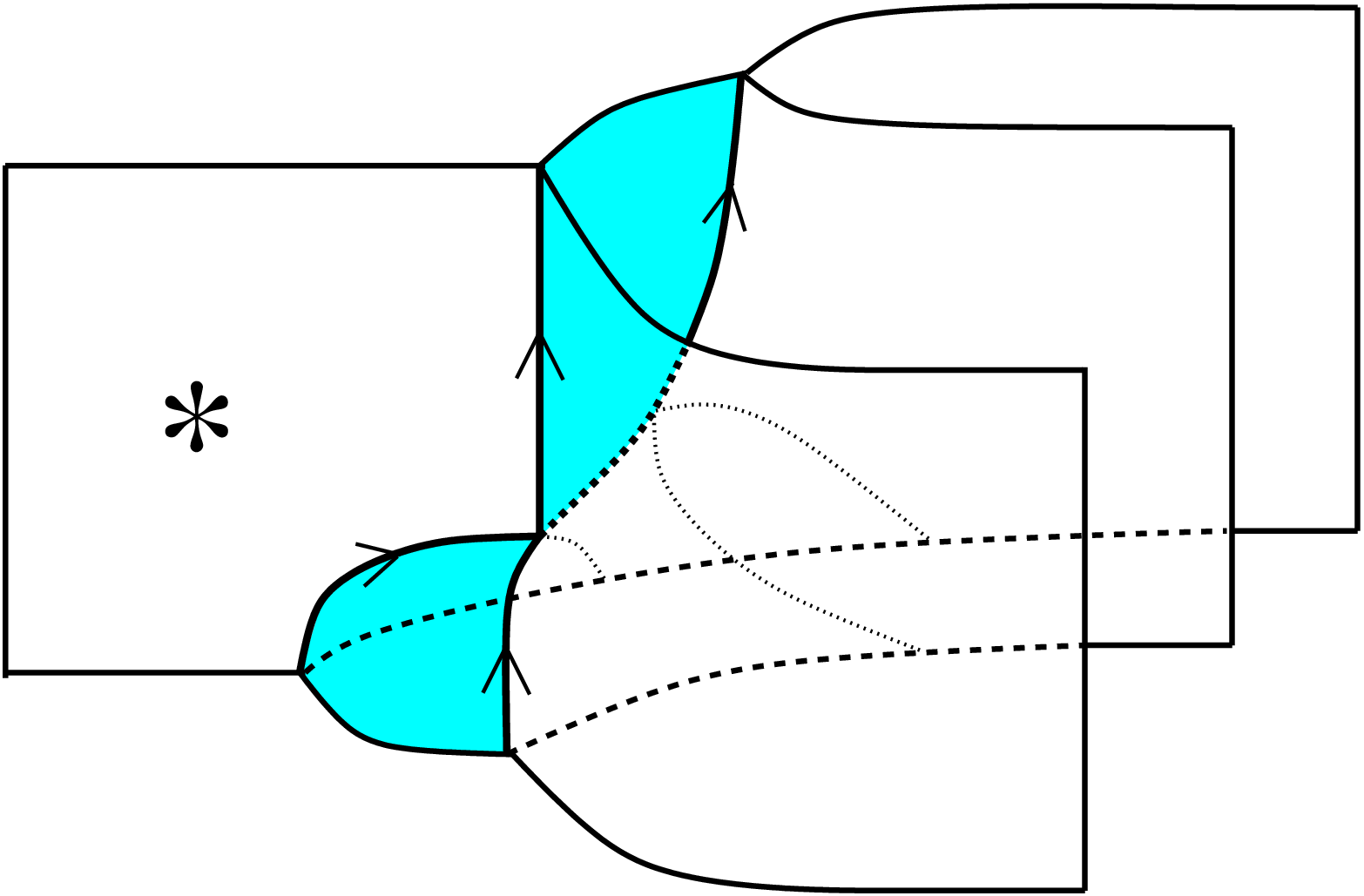}
\caption{Some elementary pre-foams}
\label{fig:elemfoams}
\end{figure}
Intersecting such a foam with a generic plane results in a web, as long as the 
plane avoids the singularities where six facets meet, such as on the right in 
Figure~\ref{fig:elemfoams}. 

\begin{defn}
\label{defn:pre-foam}
Let $\sk$ be a finite oriented closed $4$-valent graph, 
which may contain disjoint circles and loose endpoints. 
We assume that all edges of $\sk$ are oriented.  
A cycle in $\sk$ is defined to be a circle or a closed sequence of edges 
which form a 
piece-wise linear circle. 
Let $\Sigma$ be a compact orientable possibly disconnected surface, 
whose connected components are simple, double or triple, 
denoted by white, coloured or marked. 
Each component can have a boundary consisting of several disjoint circles and can have 
additional decorations which we discuss below.     
A closed {\em pre-foam} $u$ is the identification space $\Sigma/\sk$ obtained by glueing 
boundary circles of $\Sigma$ to cycles in $\sk$ such that every edge and circle 
in $\sk$ is glued 
to exactly three boundary circles of $\Sigma$ and such that for any point $p\in \sk$:    
\begin{enumerate}
\item if $p$ is an interior point of an edge, then $p$ has a neighborhood homeomorphic to the 
letter Y times an interval with exactly one of the facets being double, 
and at most one of them 
being triple. For an example see 
Figure~\ref{fig:elemfoams};
\item if $p$ is a vertex of $\sk$, then it has a neighborhood as shown in 
Figure~\ref{fig:elemfoams}. 
\end{enumerate}
We call $\sk$ the \emph{singular graph}, its edges and vertices \emph{singular arcs} and 
\emph{singular vertices}, and the connected components of $u - \sk$ the \emph{facets}.

Furthermore the facets can be decorated with dots. A simple facet can only have black 
dots ($\bdot$), a double facet can also have white dots ($\wdot$), and a triple facet besides 
black and white dots can have double dots ($\bwdot$). Dots can move freely on a facet 
but are not allowed to cross singular arcs. 
\end{defn}


Note that the cycles to which the boundaries of the simple and the triple facets are 
glued are always oriented, whereas the ones to which the boundaries of the double 
facets are glued are not, as can be seen in Figure~\ref{fig:elemfoams}. 
Note also that there are two types of singular vertices. Given a singular vertex $v$, 
there are precisely two singular edges which meet at $v$ and bound a triple facet: 
one oriented 
toward $v$, denoted $e_1$, and one oriented away from $v$, denoted $e_2$. 
If we use the ``left hand rule'', then the cyclic ordering of the facets 
incident to $e_1$ and $e_2$ is either $(3,2,1)$ and $(3,1,2)$ respectively, or the other way 
around. We say that $v$ is of type I in the first case and of type II in the second case. 
When we go around a triple facet we see that there have to 
be as many singular vertices of 
type I as there are of type II for the cyclic orderings of the facets to match up. This shows 
that for a closed pre-foam the number of singular vertices of type I is equal to the number 
of singular vertices of type II.     

We can intersect a pre-foam $u$ generically by a plane $W$ in order to get a 
closed web, 
as long as the plane avoids the vertices of $\sk$. The orientation of $\sk$ 
determines the orientation of the simple edges of the web 
according to the convention in Figure~\ref{fig:orientations}.

\begin{figure}[h]
$$\xymatrix@R=1mm{
\figins{0}{0.8}{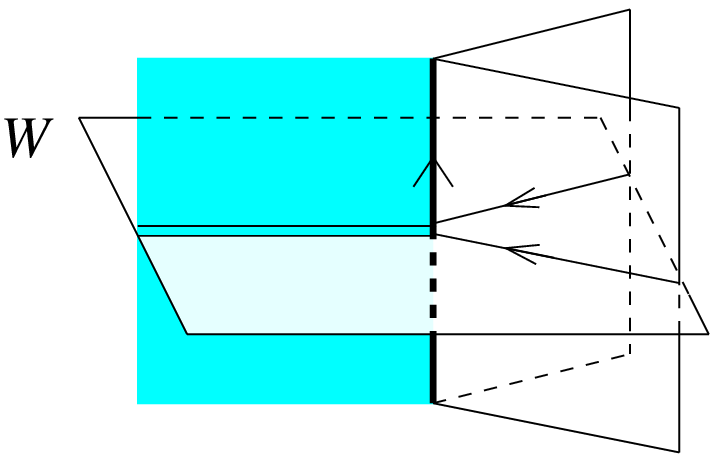} & 
\figins{0}{0.8}{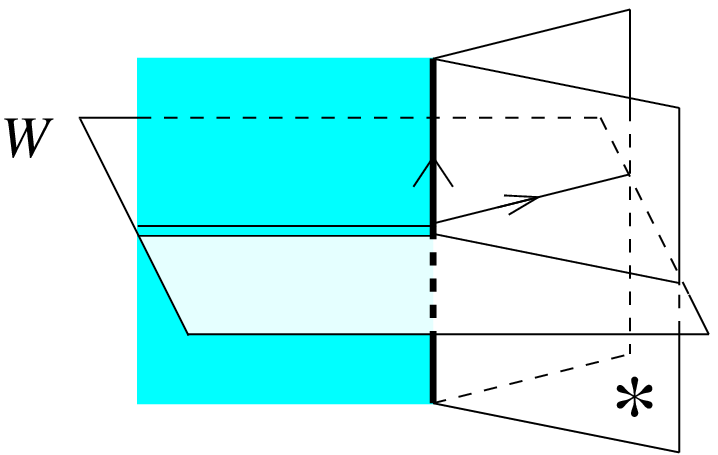} &
\figins{0}{0.8}{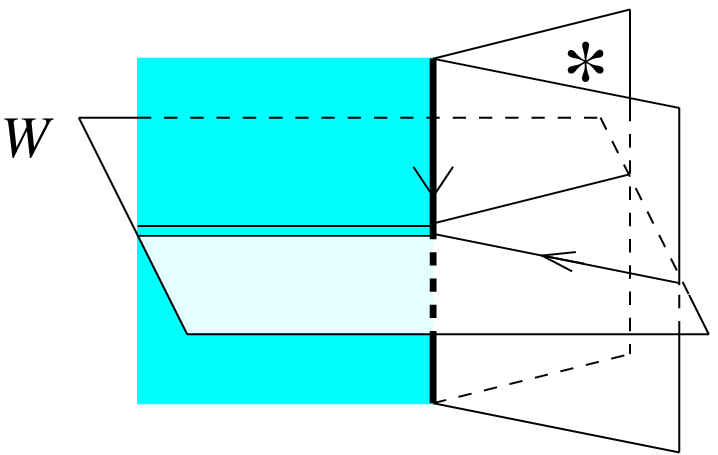} \\
\figins{0}{0.8}{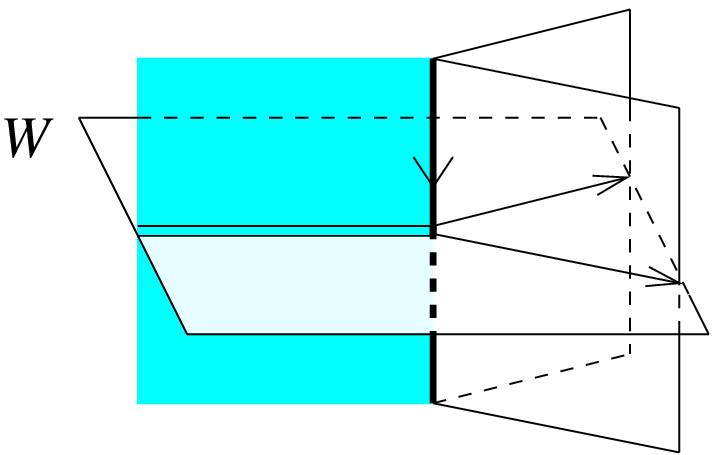}  & 
\figins{0}{0.8}{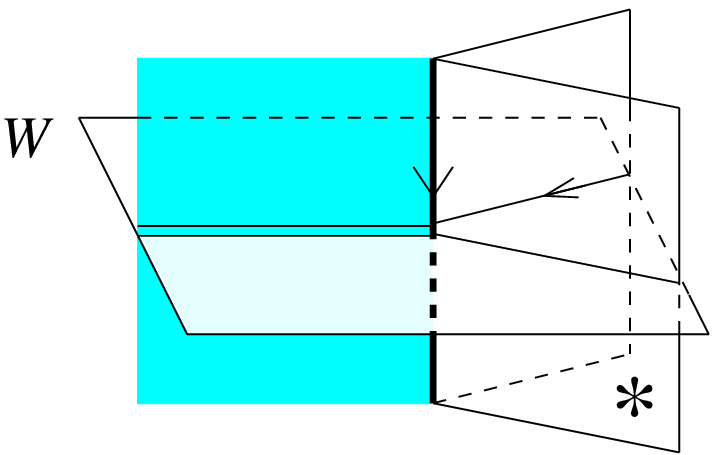} &
\figins{0}{0.8}{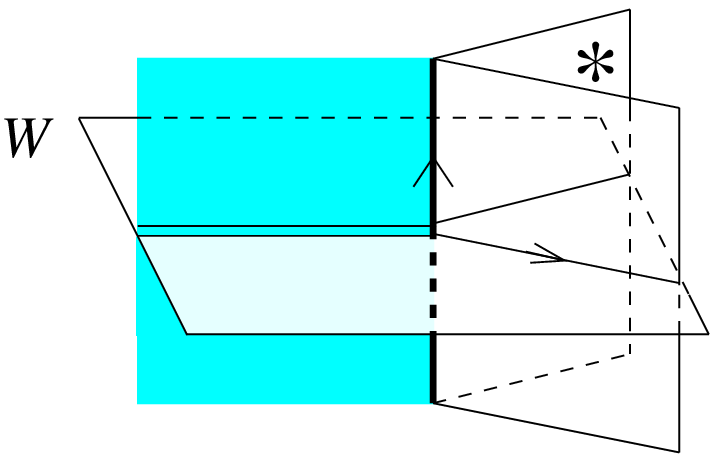}
}$$
\caption{Orientations near a singular arc}
\label{fig:orientations}
\end{figure}

Suppose that for all but a finite number of values $i\in ]0,1[$, 
the plane $W\times {i}$ intersects $u$ generically. Suppose also that 
$W\times 0$ and $W\times 1$ intersect $u$ generically and outside the vertices 
of $\sk$. Furthermore, suppose that $D\subset W$ is a disc in W and 
$C\subset D$ its boundary circle, such that 
$C\times [0,1]\cap u$ is a disjoint union of vertical line segments. 
This means that we are assuming that $s_{\gamma}$ does not intersect 
$C\times [0,1]$. We call $D\times [0,1]\cap u$ an {\em open} pre-foam between 
the {\em open} webs 
$D\times \{0\}\cap u$ and $D\times \{1\}\cap u$.   
Interpreted as morphisms we read open pre-foams from bottom to top, and their 
composition consists of placing one pre-foam on top of the other, as long as 
their boundaries are isotopic and the orientations of the simple edges 
coincide. 

\begin{defn} Let $\PF$ be the category whose objects are 
webs and whose morphisms are $\bQ$-linear combinations of 
isotopy classes of pre-foams with the obvious identity pre-foams and 
composition rule.  
\end{defn}

We now define the $q$-degree of a pre-foam. 
Let $u$ be a pre-foam, $u_1$, $u_2$ and $u_3$ the disjoint union of its 
simple and double and marked facets respectively and 
$\sk(u)$ its singular graph. Furthermore, let $b_1$, $b_2$ and $b_3$ be the 
number of simple, double and marked vertical boundary edges of $u$, 
respectively. 
Define the partial $q$-gradings of $u$ as 
\begin{eqnarray*}
q_i(u)    &=& \chi(u_i)-\frac{1}{2}\chi(\partial u_i\cap\partial u)-
\frac{1}{2}b_i, 
\qquad i=1,2,3 \\
q_{\sk}(u) &=& \chi(\sk(u))-\frac{1}{2}\chi(\partial \sk(u)).
\end{eqnarray*}
where $\chi$ is the Euler characteristic and $\partial$ denotes the boundary.

\begin{defn}
Let $u$ be a pre-foam with $d_\bdot$ dots of type $\bdot$, $d_\wdot$ dots of type 
$\wdot$ and $d_\bwdot$ dots of type $\bwdot$. The $q$-grading of $u$ is given 
by
\begin{equation}
q(u)= -\sum_{i=1}^{3}i(N-i)q_i(u) - 2(N-2)q_{\sk}(u) + 
2d_\bdot + 4d_\wdot +6d_\bwdot.
\end{equation}
\end{defn}

The following result is a direct consequence of the definitions. 
\begin{lem}
$q(u)$ is additive under the glueing of pre-foams.
\end{lem}

We denote a simple facet with $i$ dots by 
$$\figins{-8}{0.3}{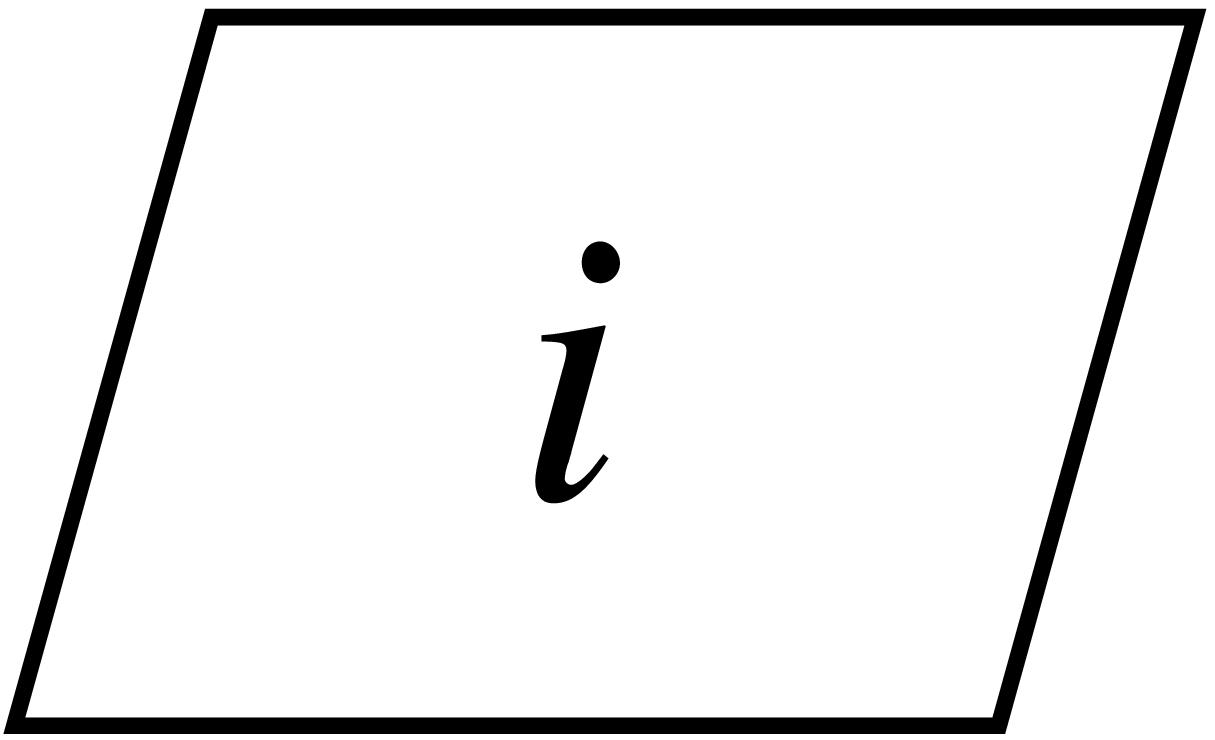}.$$  
Recall that the two-variable Schur polynomial $\pi_{k,m}$ can be expressed in 
terms of the elementary symmetric polynomials $\pi_{1,0}$ and $\pi_{1,1}$. 
By convention, the latter correspond to $\bdot$ and $\wdot$ on a double facet 
respectively, so that   
$$\figins{-8}{0.3}{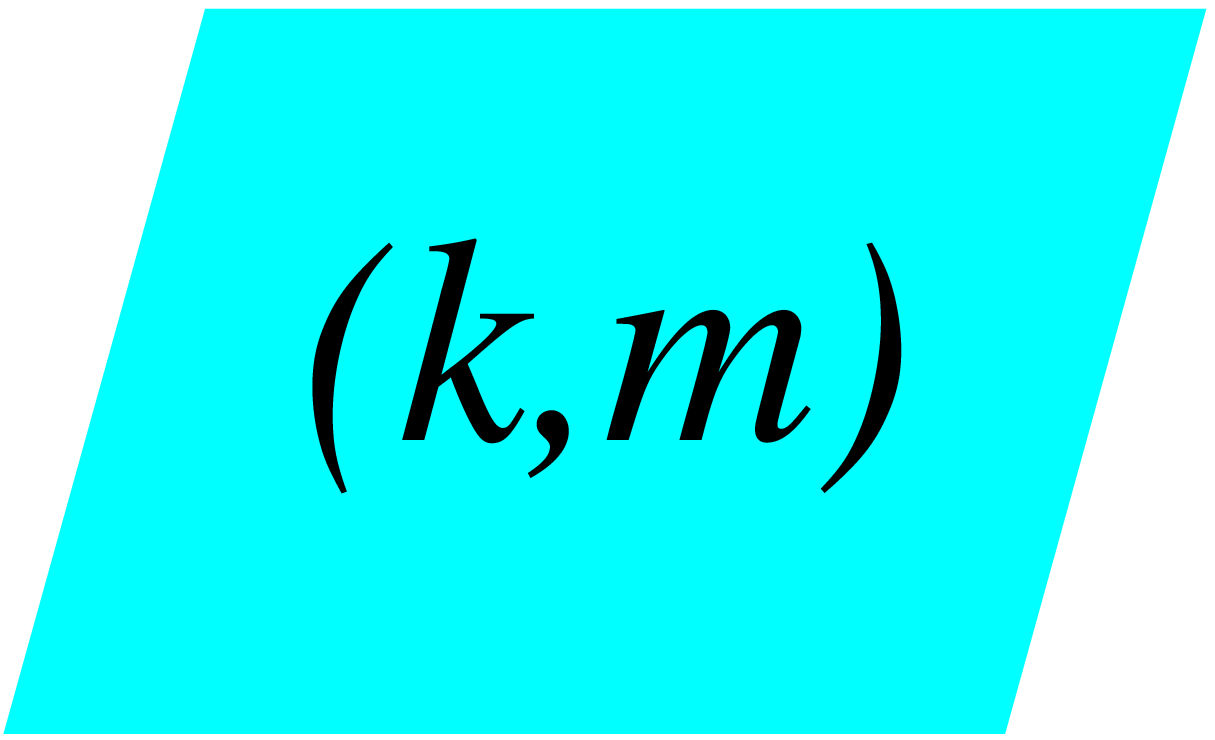}$$
is defined to be the linear combination of dotted double facets corresponding 
to the expression of $\pi_{k,m}$ in terms of $\pi_{1,0}$ and $\pi_{1,1}$. 
Analogously we can express the three-variable Schur polynomial 
$\pi_{p,q,r}$ in terms of the elementary symmetric polynomials 
$\pi_{1,0,0}$, $\pi_{1,1,0}$ and $\pi_{1,1,1}$ 
By convention, the latter correspond to $\bdot$, $\wdot$ and $\bwdot$ on a 
triple facet respectively, so we can make sense of   
$$\figins{-8}{0.3}{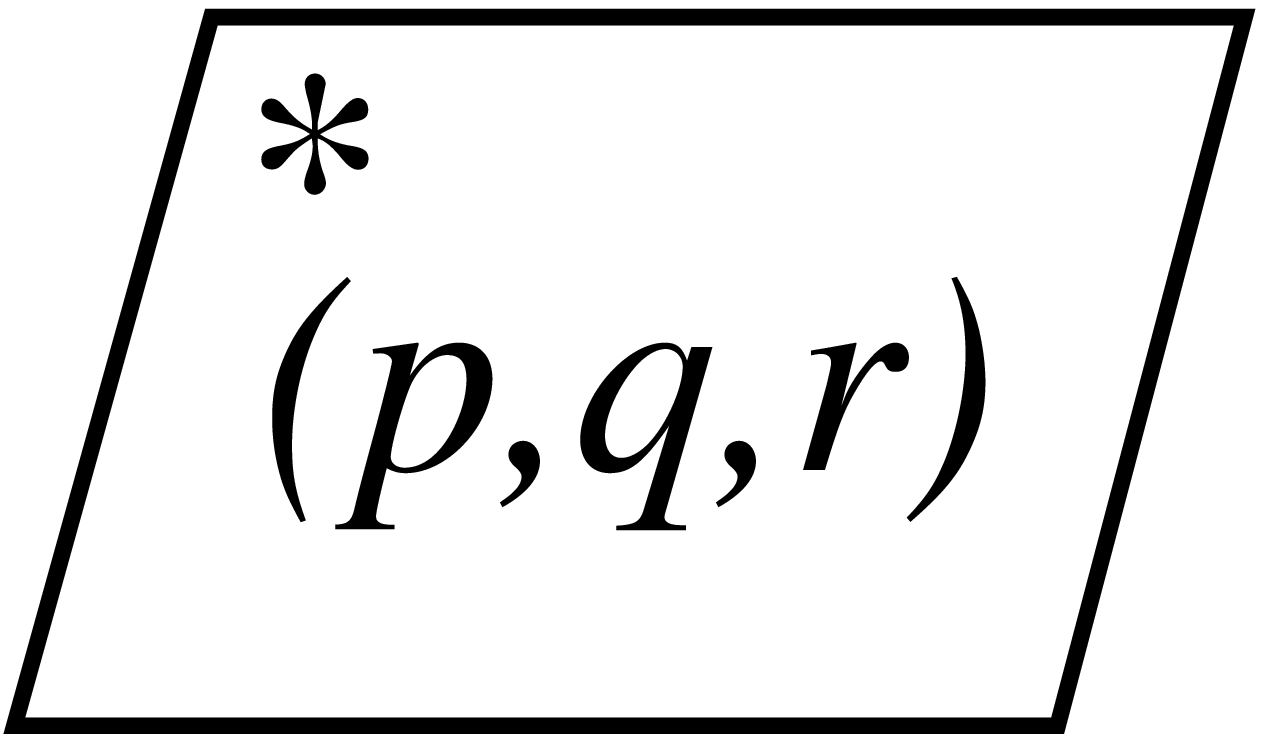}.$$


\subsection{Foams}
\label{sec:foamN}
In~\cite{MSV1, V} we gave a precise definition of the Kapustin-Li formula, 
following Khovanov and Rozansky's work~\cite{KhR3}. We will not repeat that 
definition here, since it is complicated and unnecessary for our purposes in 
this paper. The only thing one needs to remember is that the Kapustin-Li 
formula associates a number to any closed prefoam and that 
those numbers have very special properties, some of which we will recall 
below. By $\langle u \rangle_{KL}$ we denote the Kapustin-Li evaluation 
of a closed pre-foam $u$.  
\begin{defn}
\label{defn:foam}
The category $\foam$ is the quotient of the category $\PF$ by the 
kernel of $\langle\ \rangle_{KL}$, i.e. by the following identifications: 
for any webs $\Gamma$, $\Gamma'$ 
and finite sets $f_i\in\mbox{Hom}_{\PF}\left(\Gamma,\Gamma'\right)$ 
and $c_i\in\bQ$ we impose the relations 
$$\sum\limits_{i}c_if_i=0\quad \Leftrightarrow\quad
\sum\limits_{i}c_i\langle \overline{f_i}\rangle_{KL}=0,$$ 
for any fixed way of closing the $f_i$, denoted by $\overline{f_i}$. 
By ``fixed''  
we mean that all the $f_i$ are closed in the same way. The 
morphisms of $\foam$ are called {\em foams}. 
\end{defn}

In the next proposition we recall those relations in $\foam$ that we 
need in the sequel. For their proofs and other relations we refer to 
\cite{MSV1}.

\begin{prop}
\label{prop:principal rels1}
The following identities hold in $\foam$: 

 
(The {\em dot conversion} relations) 

$$\figins{-8}{0.3}{plan-i}=0\quad\mbox{if}\quad i\geq N.$$ 
$$\quad\ \
\figins{-8}{0.3}{dplan-km}=0\quad\mbox{if}\quad k\geq N-1.$$
$$\quad\ \
\figins{-8}{0.3}{plan-pqr}=0\quad\mbox{if}\quad p\geq N-2.$$


\bigskip

(The {\em dot migration} relations)
\begin{equation}\label{eq:dotmigr}
\begin{split}
\figins{-12}{0.4}{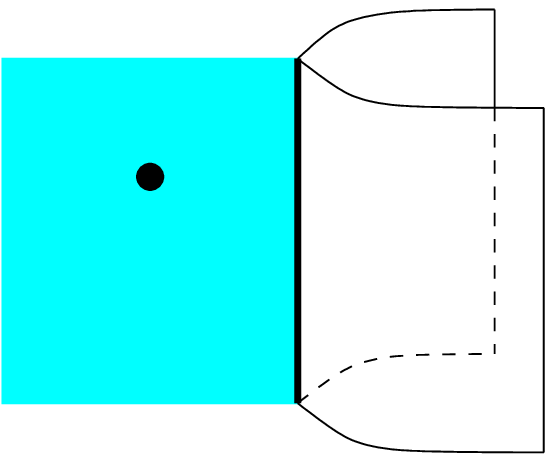} &=  \quad
\figins{-12}{0.4}{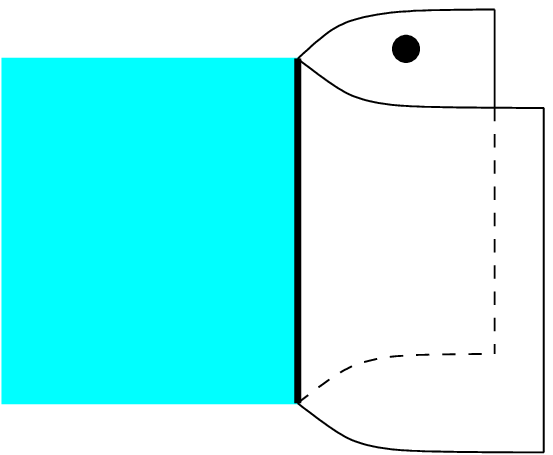} \ + \
\figins{-12}{0.4}{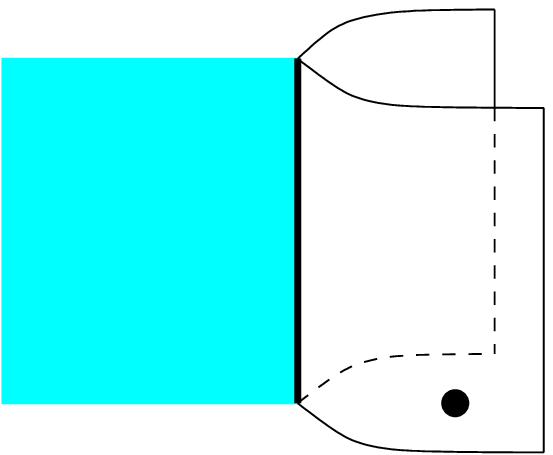}
\\[1.2ex] 
\figins{-12}{0.4}{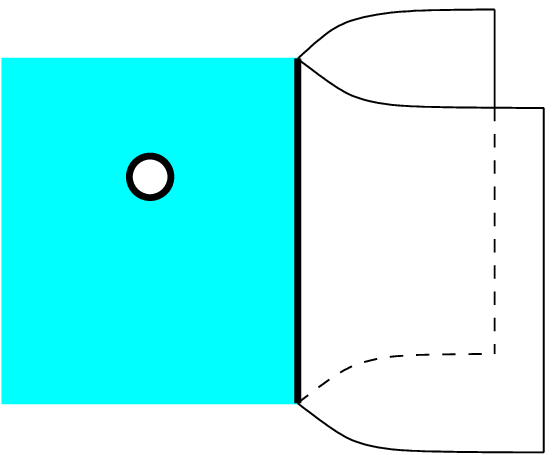} &=  \quad
\figins{-12}{0.4}{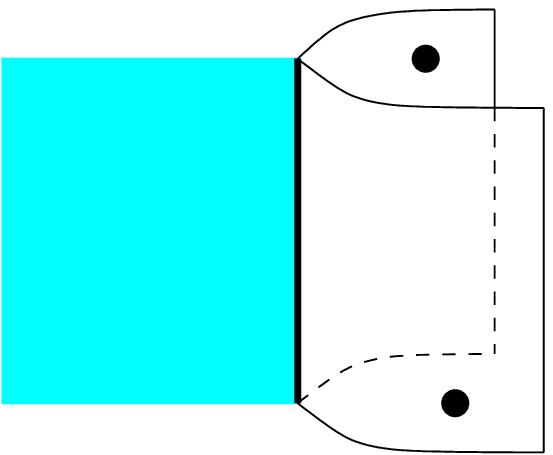}
\\[1.5ex]
\figins{-12}{0.4}{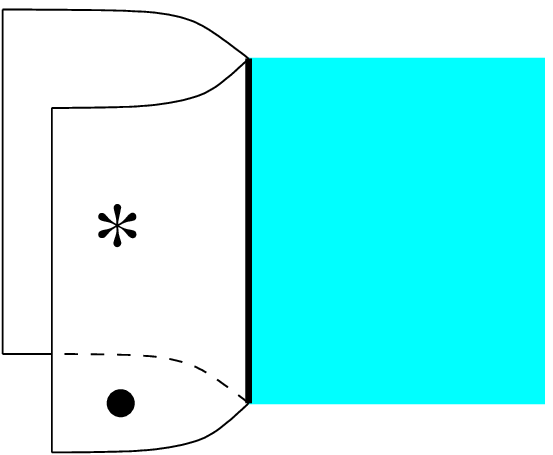} \ \ \ &=  \
\figins{-12}{0.4}{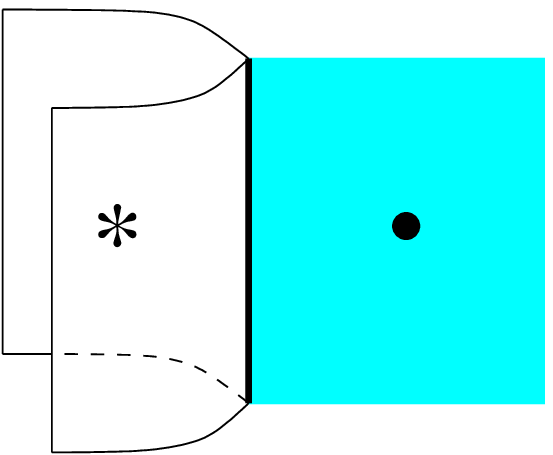} \ + \
\figins{-12}{0.4}{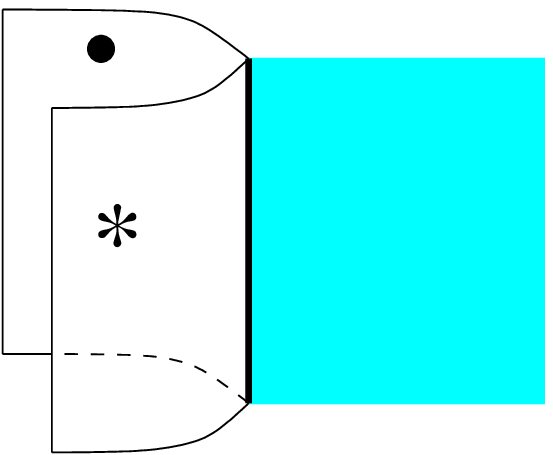} 
\\[1.2ex]
\figins{-12}{0.4}{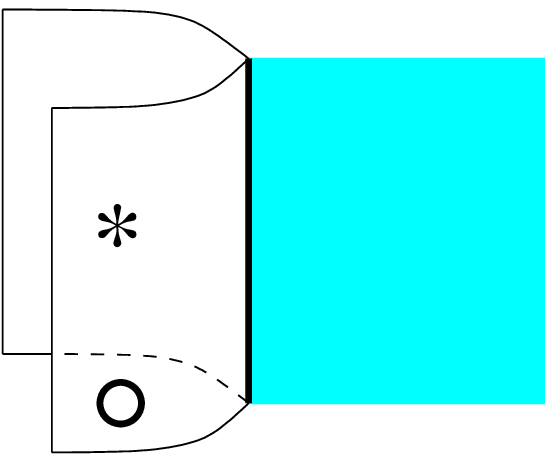} \ \ \ &=  \
\figins{-12}{0.4}{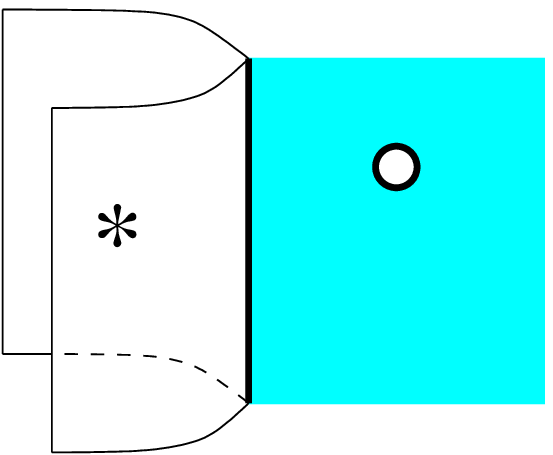} \ + \
\figins{-12}{0.4}{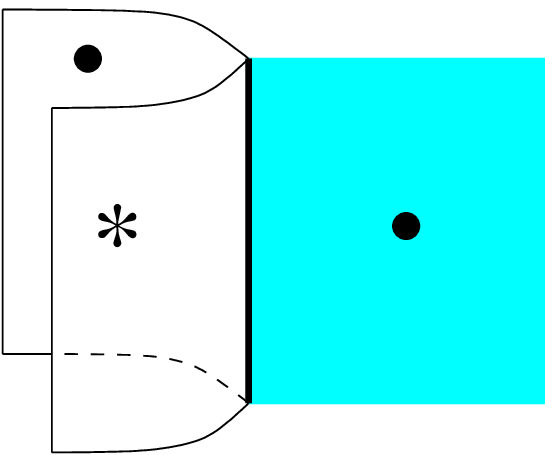}
\\[1.2ex]
\figins{-12}{0.4}{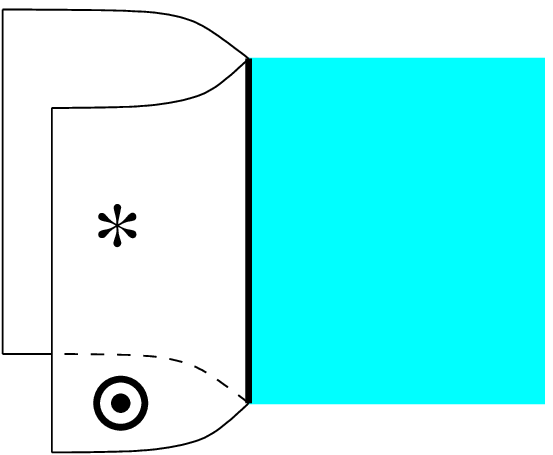} \ \ \ &=  \
\figins{-12}{0.4}{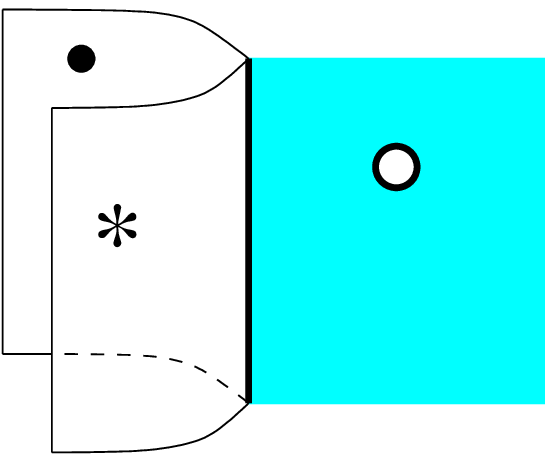}
\end{split}
\end{equation}

\bigskip

(The {\em neck cutting} relations\footnote{These were called 
\emph{cutting neck} relations in~\cite{MSV1, V}.}) 
$$
\figins{-27}{0.8}{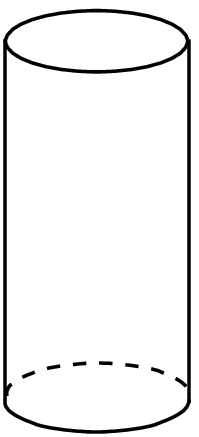}=
\sum\limits_{i=0}^{N-1}
\figins{-27}{0.8}{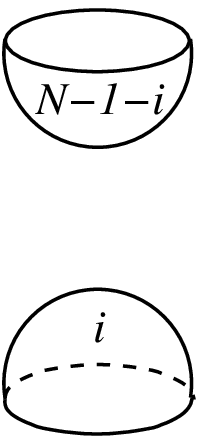}
\quad\text{(NC$_1$)}
$$
$$   
\figins{-28}{0.8}{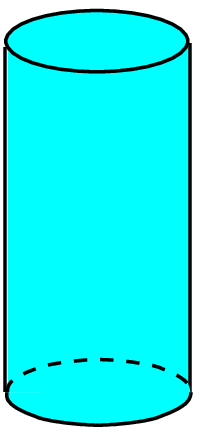}=-
\sum\limits_{0\leq j\leq i\leq N-2}
\figins{-28}{0.8}{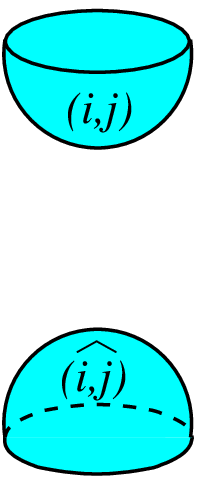}
\quad\text{(NC$_2$)}
\qquad\qquad
\figins{-28}{0.8}{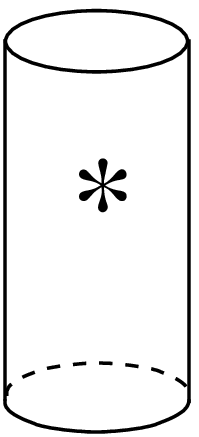}=
-\sum\limits_{0\leq k\leq j\leq i\leq N-3}
\figins{-28}{0.8}{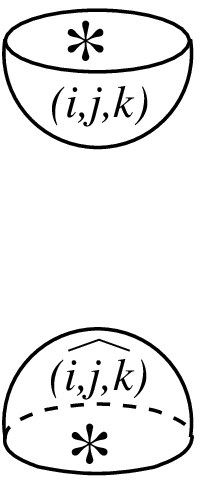}
\quad\text{(NC$_*$)}
$$

\bigskip

(The {\em sphere} relations)
$$
\figins{-10}{0.35}{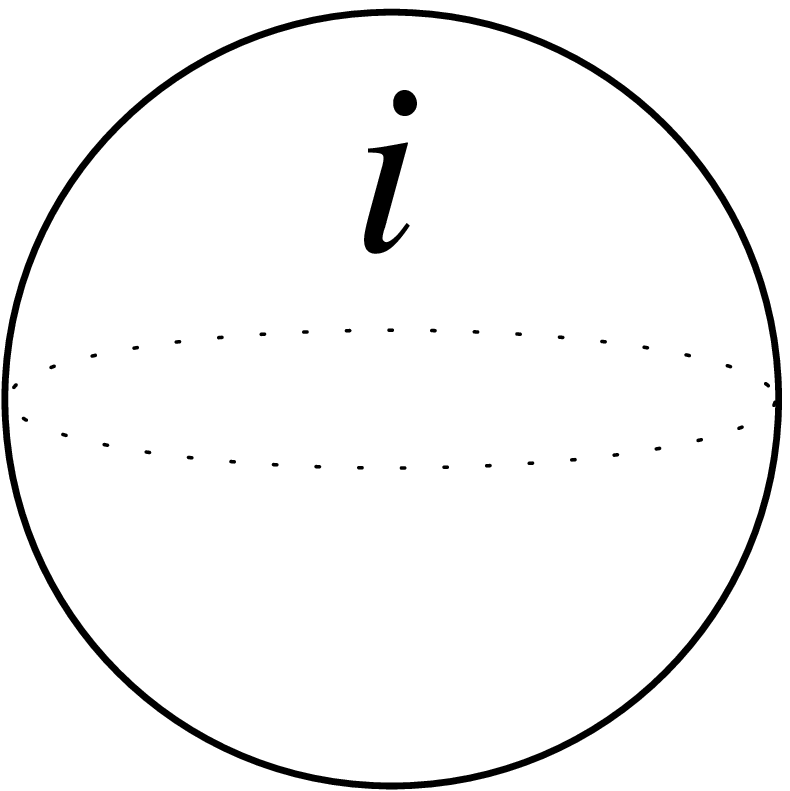}=
\begin{cases}1, & i=N-1 \\ 0, & \text{else}\end{cases}\quad 
\text{(S$_1$)} \qquad\quad
\figins{-10}{0.35}{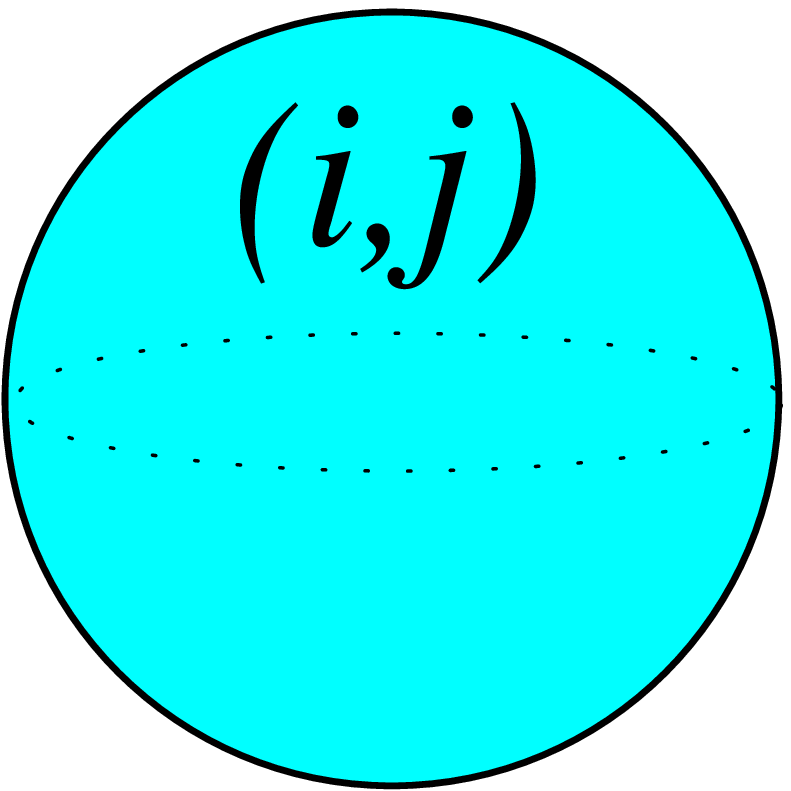}=
\begin{cases}-1, & i=j=N-2 \\ 0, & \text{else}\end{cases}\quad 
\text{(S$_2$)}
$$

$$
\figins{-10}{0.35}{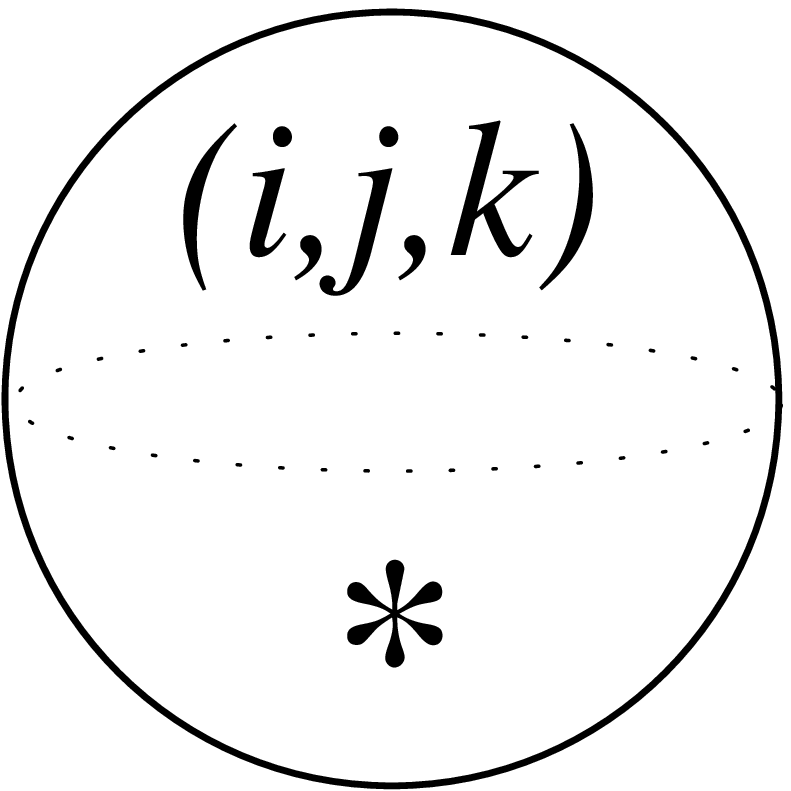}=
\begin{cases}-1, & i=j=k=N-3 \\ 0, & \text{else}.\end{cases}\quad
\text{(S$_*$)}
$$

\bigskip

(The {\em $\Theta$-foam} relations)
$$
\figins{-10}{0.35}{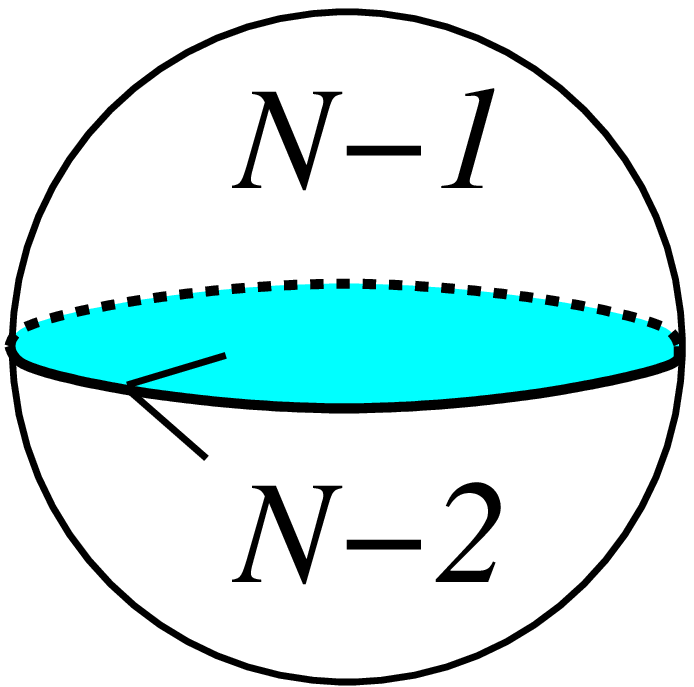} = -1  = - \
\figins{-10}{0.35}{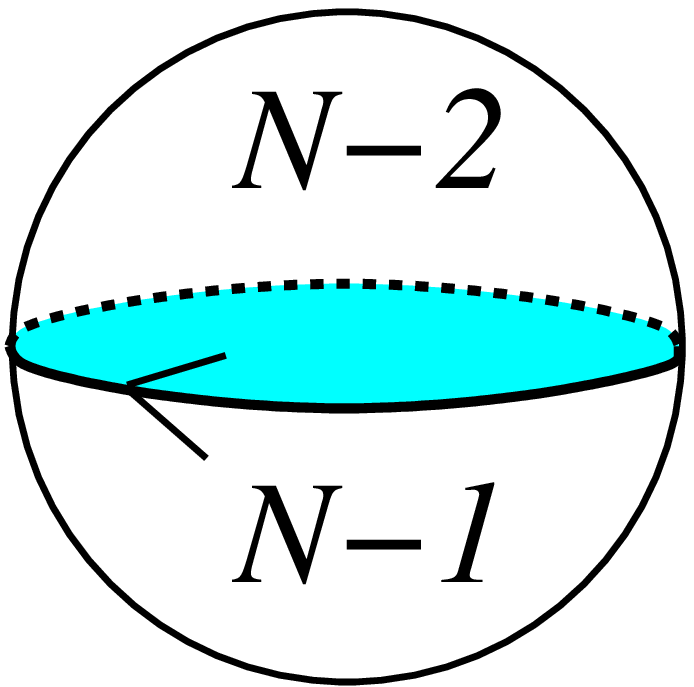} \quad
 (\ThetaGraph)
\qquad \ \text{and} \ \qquad
\figins{-10}{0.42}{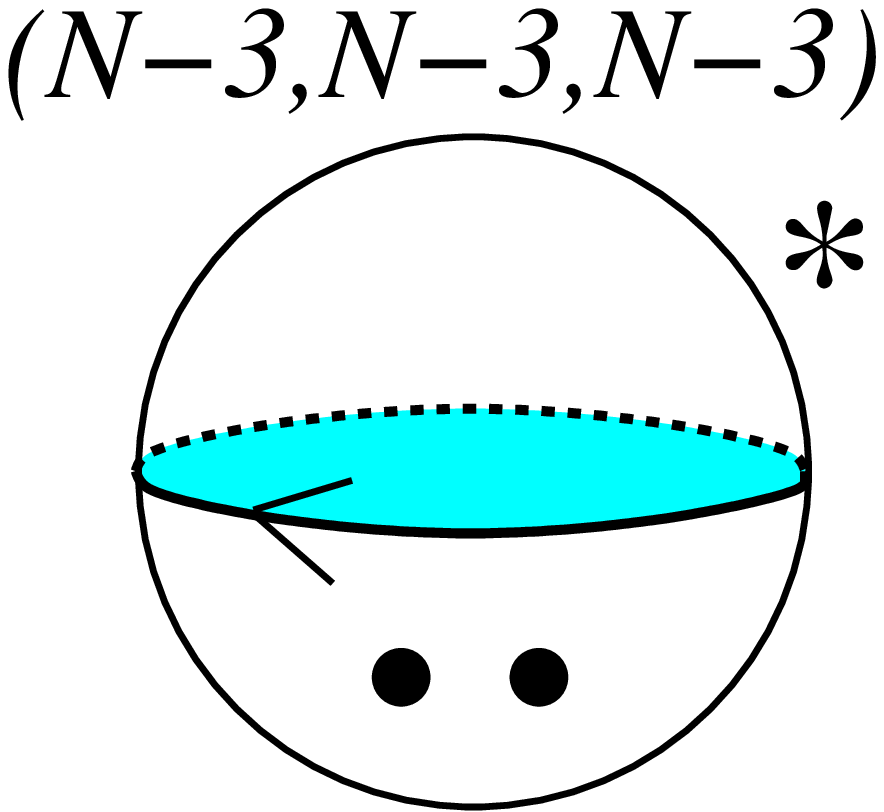} = -1  = - \
\figins{-16}{0.42}{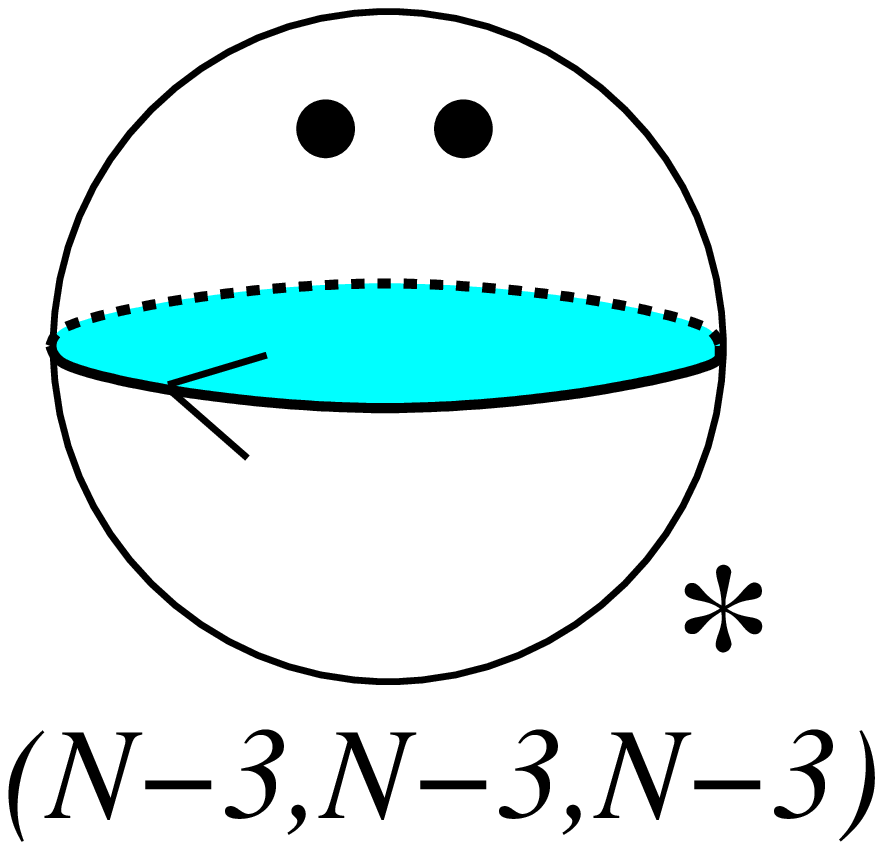} \quad
 (\ThetaGraph_*).$$
Inverting the orientation of the singular circle of $(\ThetaGraph_*)$ inverts the sign of the 
corresponding foam. A theta-foam with dots on the double facet can be transformed into a 
theta-foam with dots only on the other two facets, using the dot migration relations.


\bigskip

(The {\em Matveev-Piergalini} relation) 
\begin{equation}\tag{MP}\label{eq:mp}
\figins{-24}{0.8}{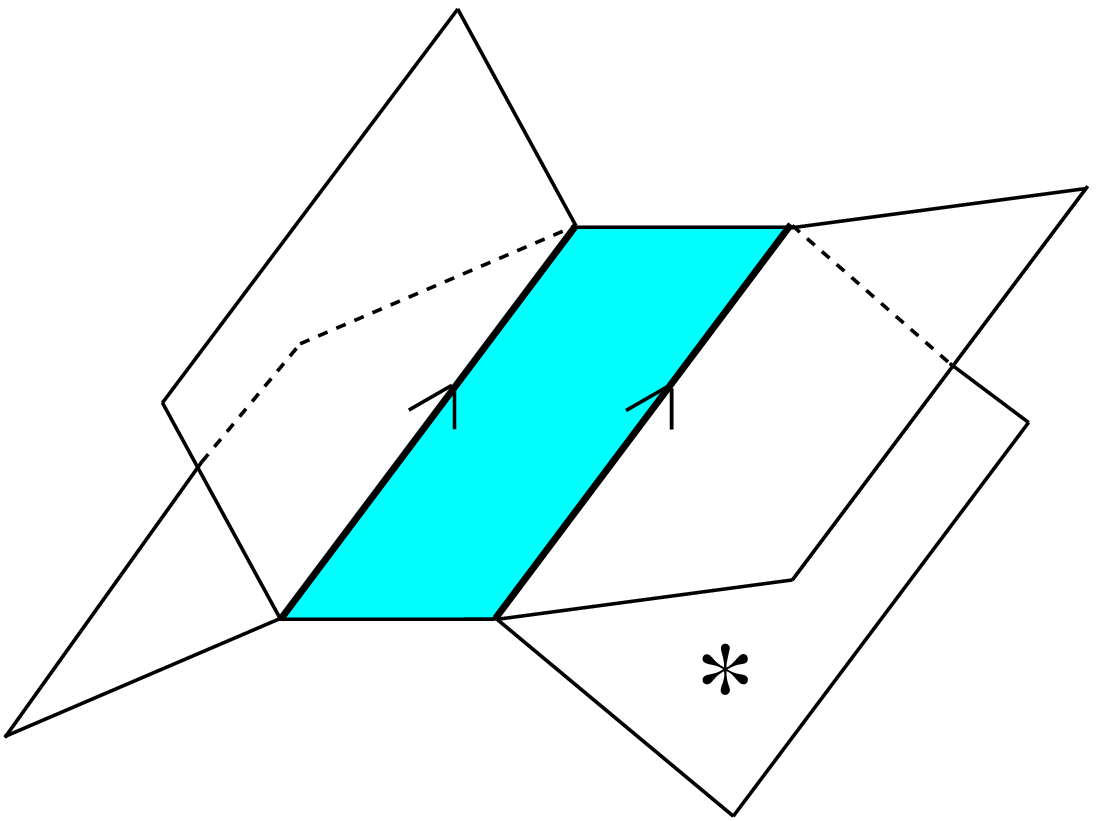} =
\figins{-24}{0.8}{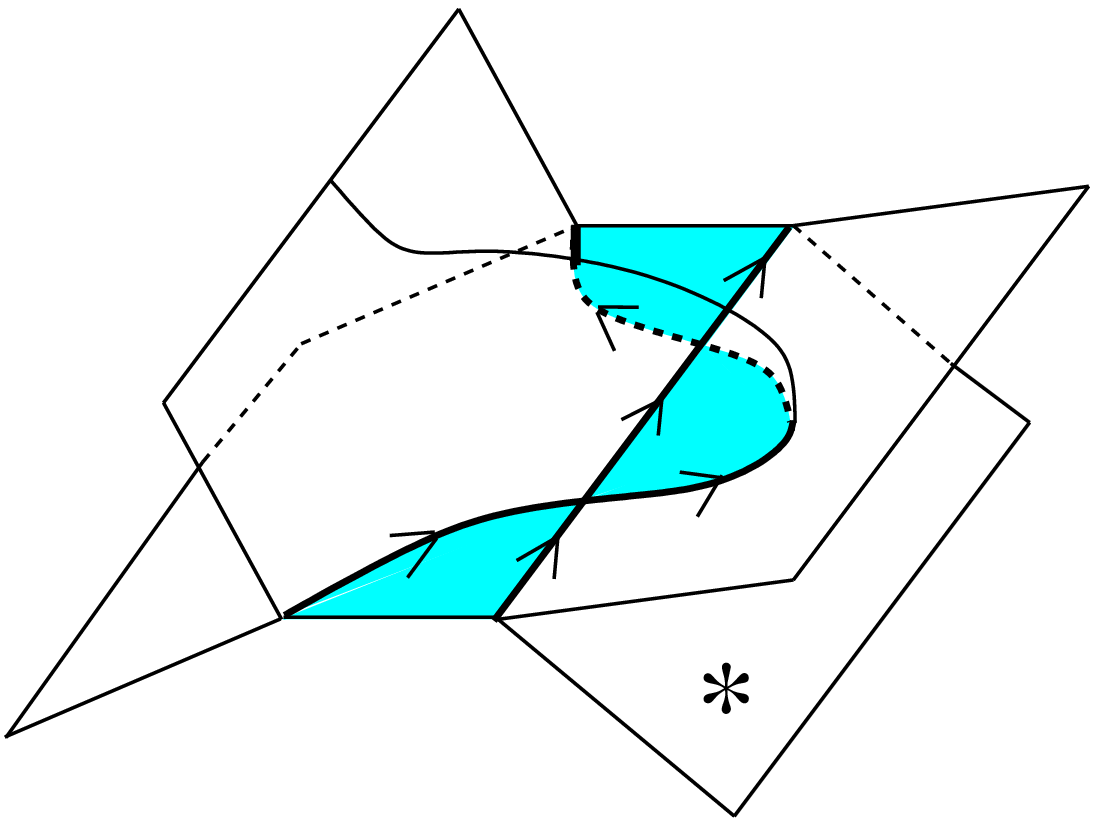},\quad
\figins{-24}{0.8}{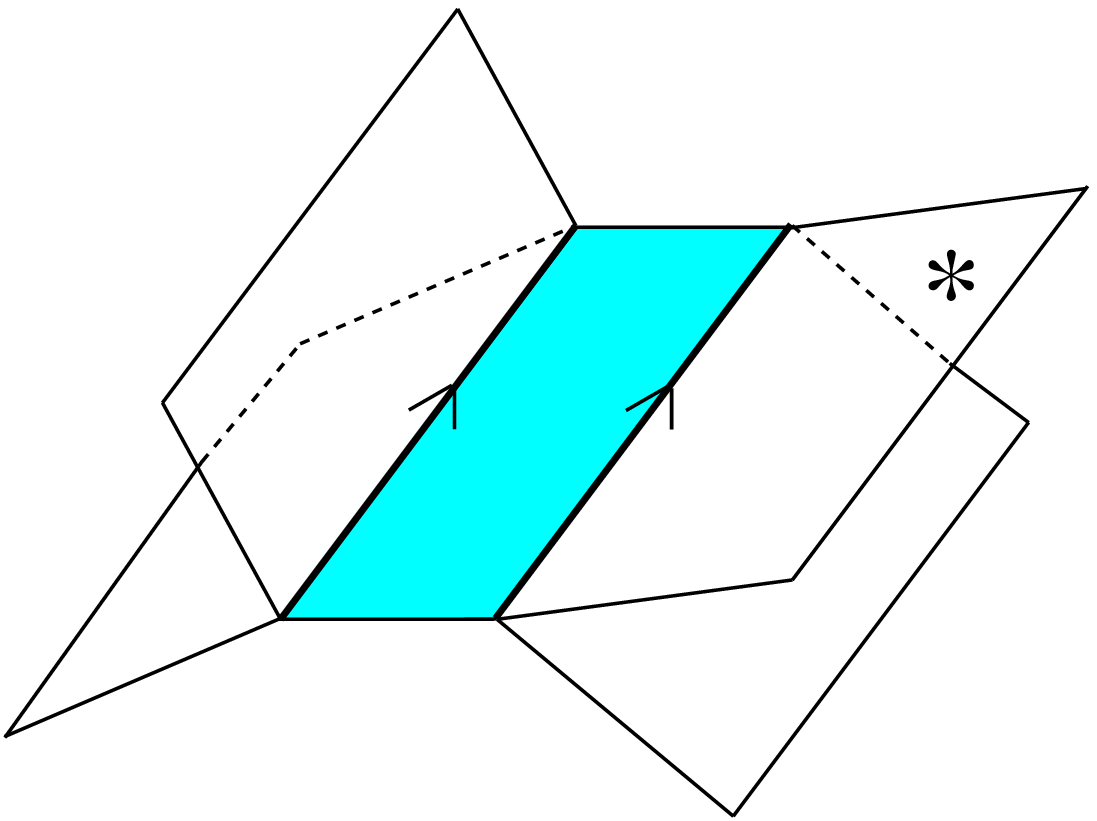} =
\figins{-24}{0.8}{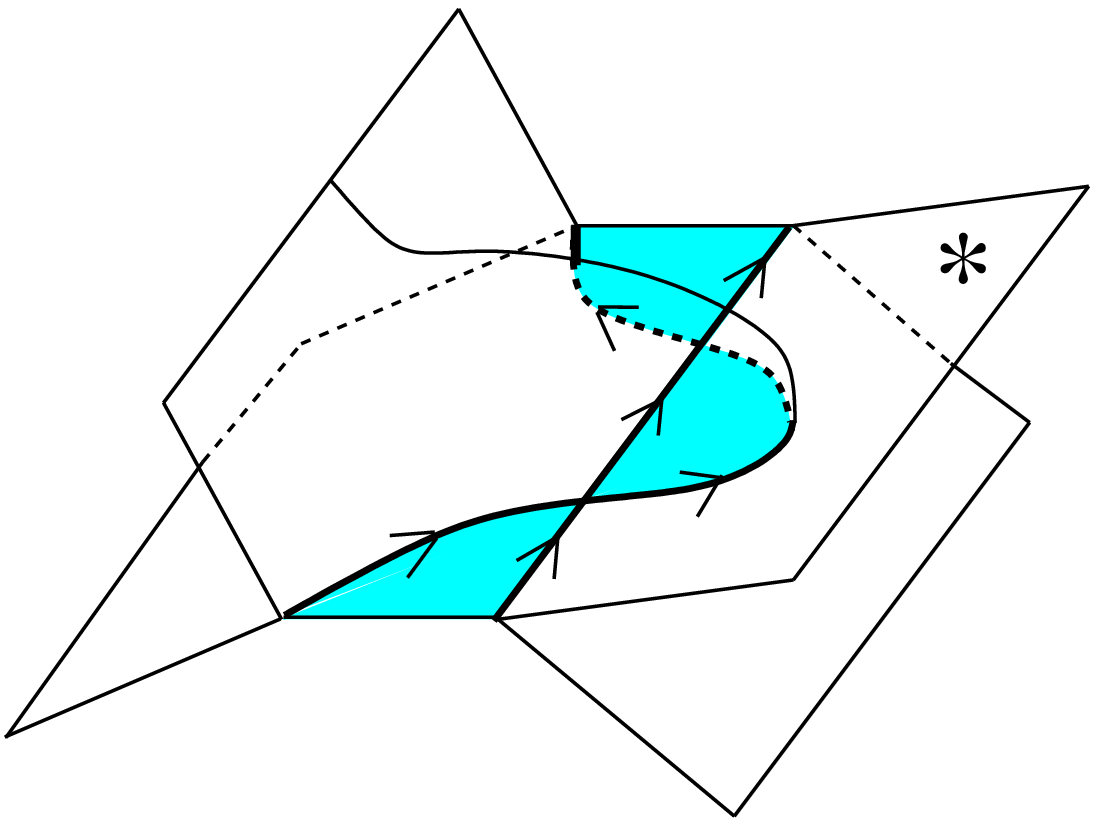}.
\end{equation}

(The {\em disc removal} relations)


\begin{equation}\tag{RD$_1$}\label{eq:rd1}
\figins{-8}{0.3}{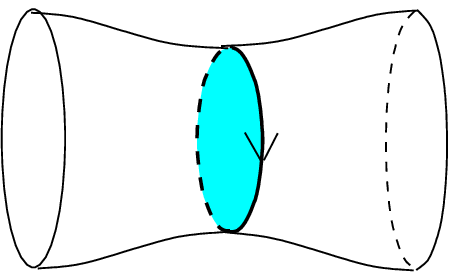} \ =  
\figins{-8}{0.3}{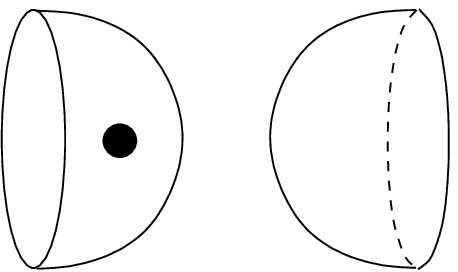} \ - \ 
\figins{-8}{0.3}{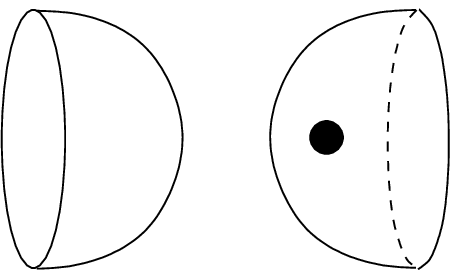}
\end{equation}

\medskip

\begin{equation}\tag{RD$_2$}\label{eq:rd2}
\figins{-8}{0.3}{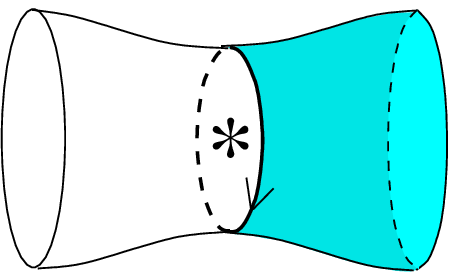} \ =  
\figins{-8}{0.3}{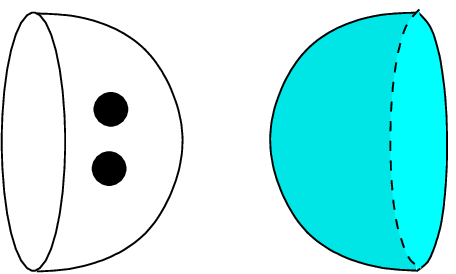} \ - \ 
\figins{-8}{0.3}{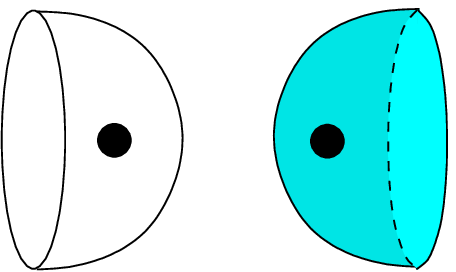} \ + \ 
\figins{-8}{0.3}{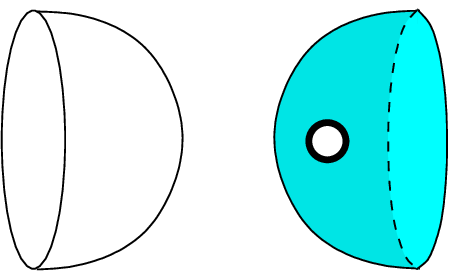}
\end{equation}

\medskip

(The {\em digon removal} relations)
\begin{equation}\tag{DR$_1$}\label{eq:dr1}
\figins{-20}{0.6}{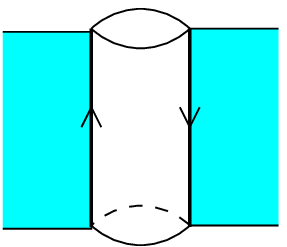}=
\figins{-20}{0.6}{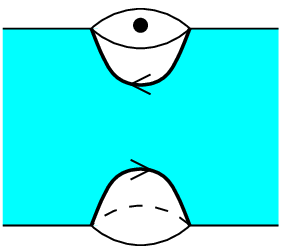}-
\figins{-20}{0.6}{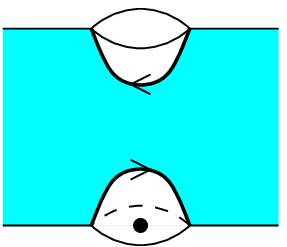}
\end{equation}

\medskip

\begin{equation}\tag{DR$_{3_1}$}\label{eq:dr31}
\figins{-30}{0.7}{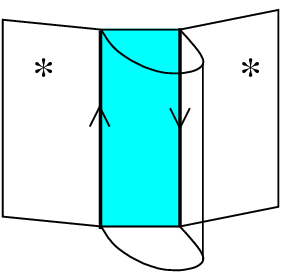} =\ -
\figins{-30}{0.7}{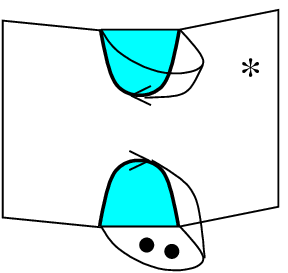}+
\figins{-30}{0.7}{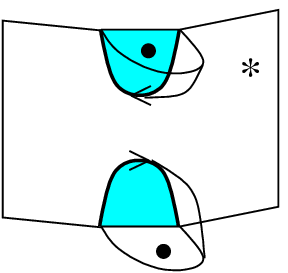}-
\figins{-30}{0.7}{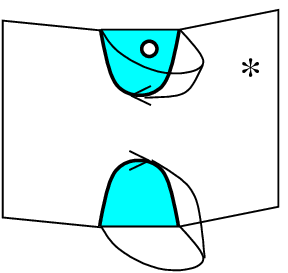}
\end{equation}

\medskip

\begin{equation}\tag{DR$_{3_2}$}\label{eq:dr32}
\figins{-30}{0.7}{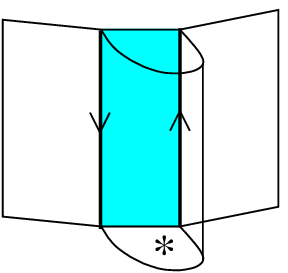} 
= \sum_{0\leq j\leq i\leq N-3}\
\figins{-30}{0.7}{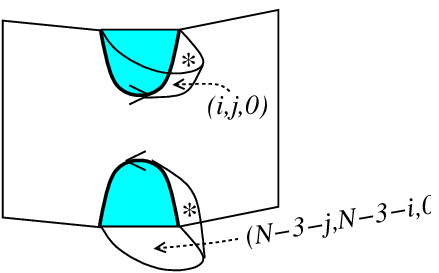}
\end{equation}

\medskip

(The {\em square removal} relations)
\begin{equation}\tag{SqR$_1$}\label{eq:sqr1}
\figins{-31}{0.9}{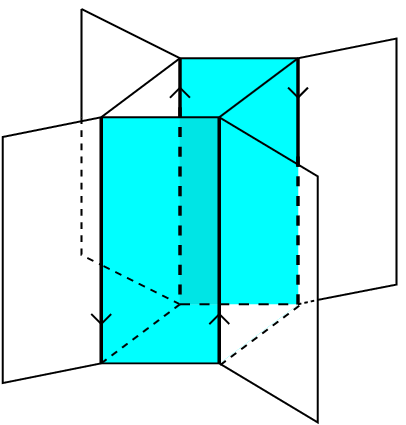}= 
-\;\figins{-31}{0.9}{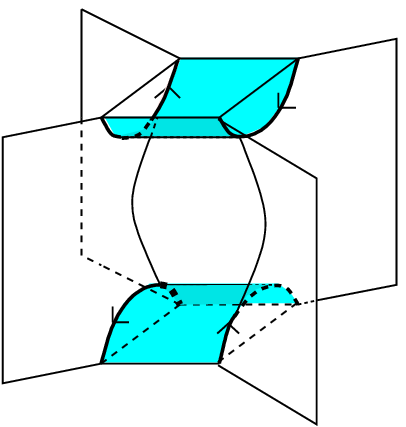}\quad + \sum_{a+b+c+d=N-3}\
\figins{-31}{0.9}{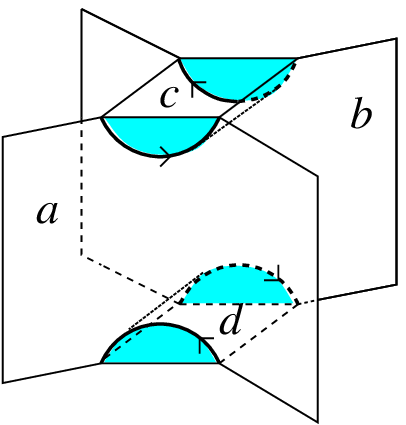}
\end{equation}

\begin{equation}\tag{SqR$_2$}\label{eq:sqr2}
\figins{-34}{0.8}{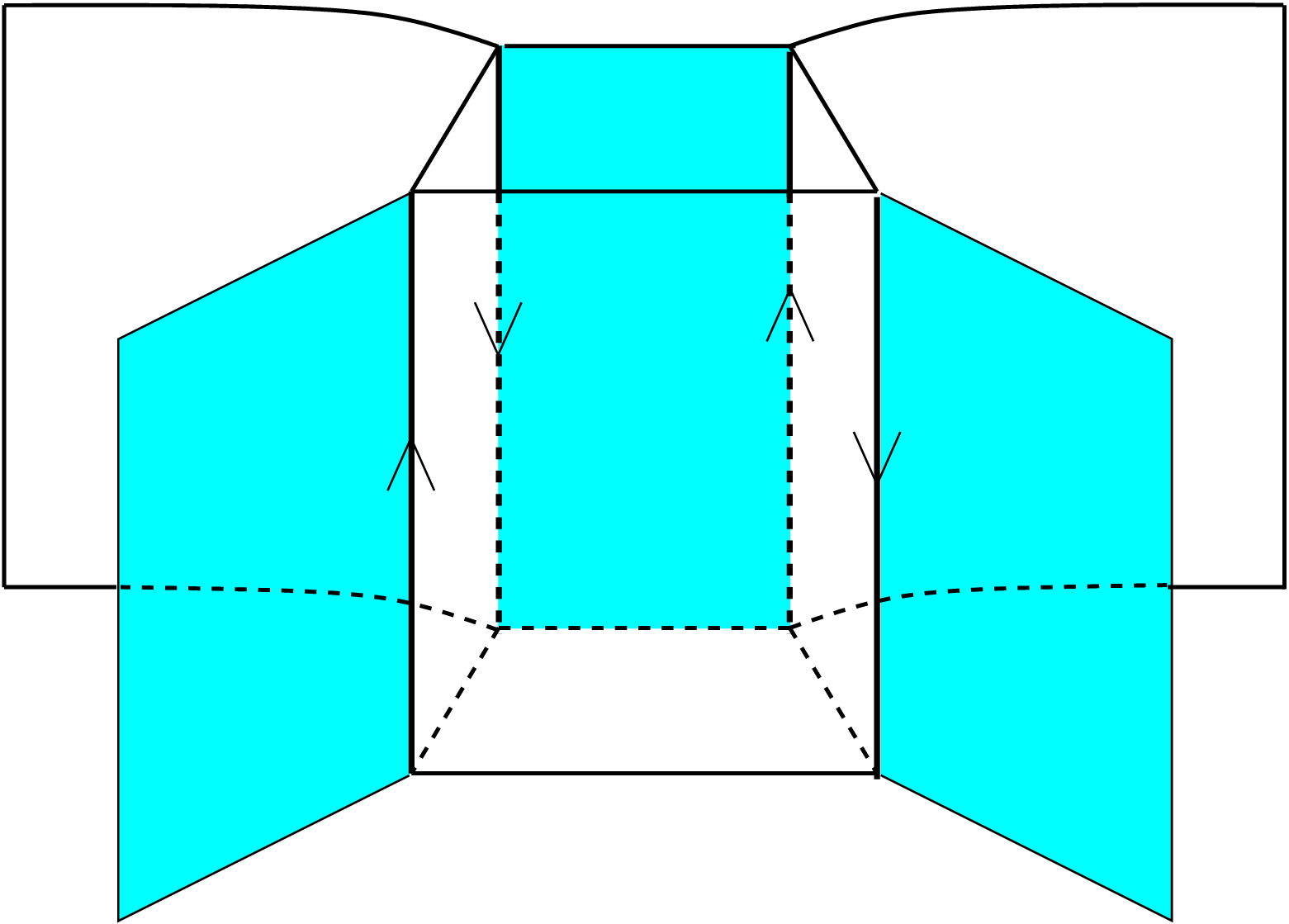}=
-\;\figins{-34}{0.8}{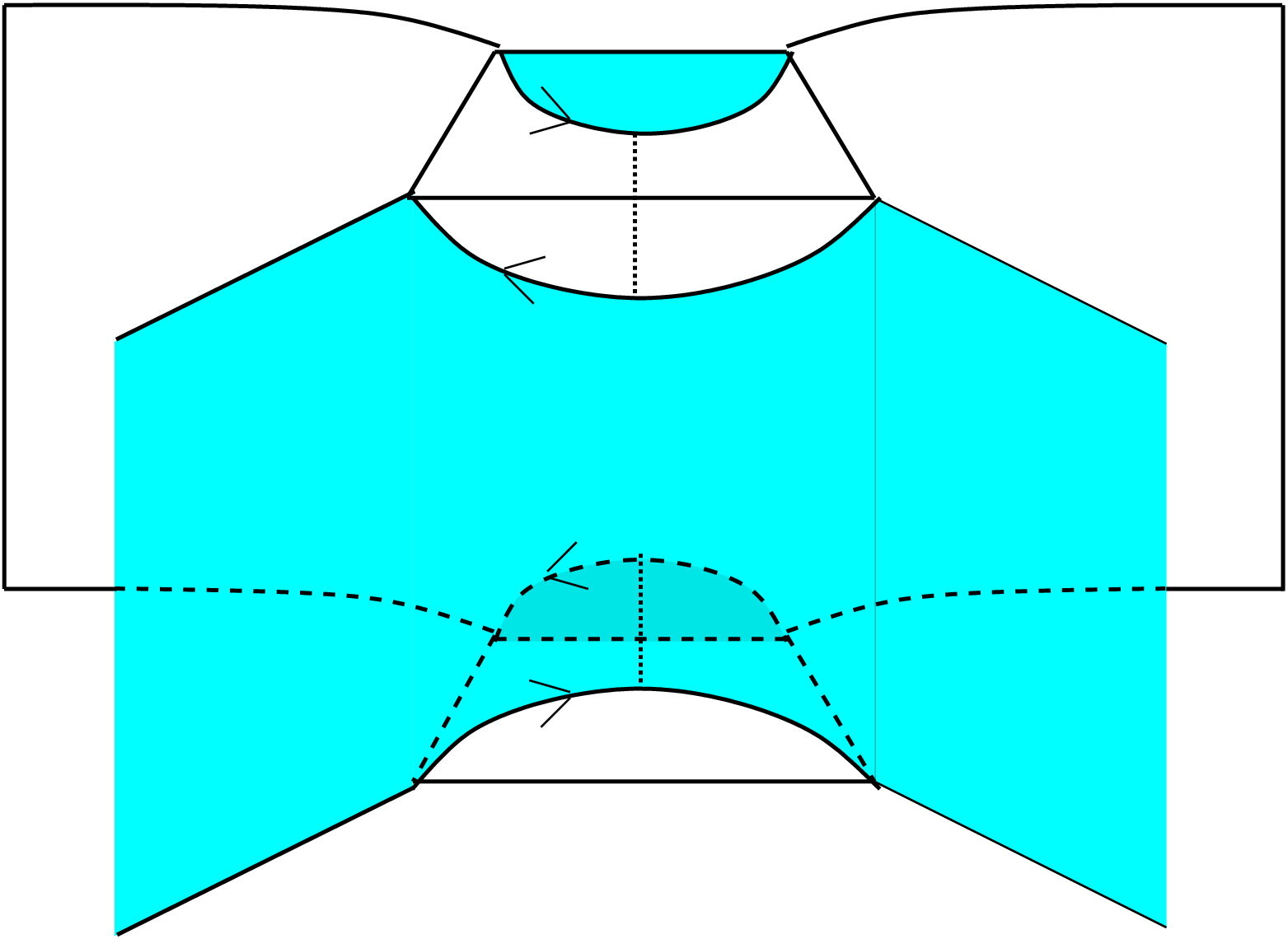}-
\figins{-34}{0.8}{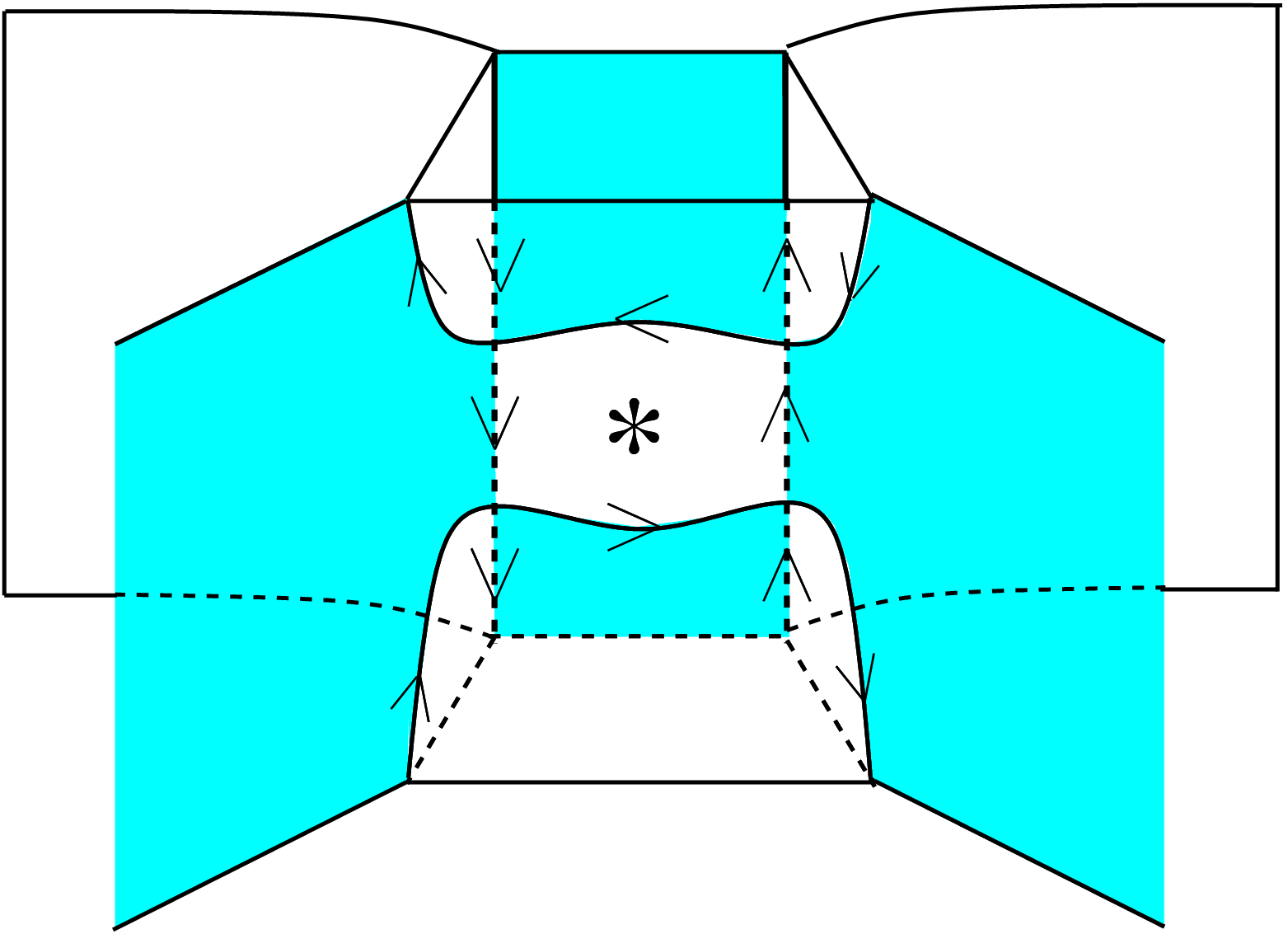}
\end{equation}

\bigskip

\begin{equation}
\label{eq:bubble-pqri}
\figins{-18}{0.6}{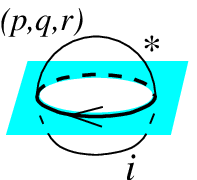}=
\begin{cases}
-\figins{-17}{0.4}{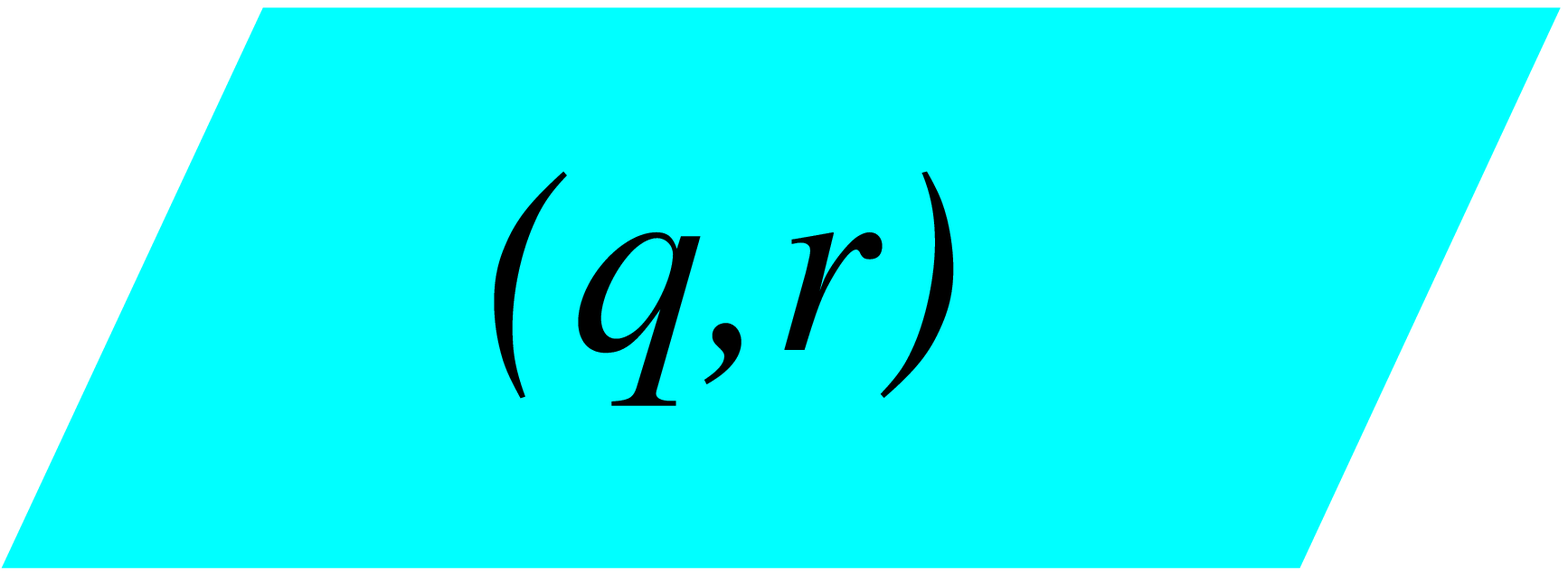}   &\text{if}\quad p=N-3-i\\ 
-\figins{-17}{0.4}{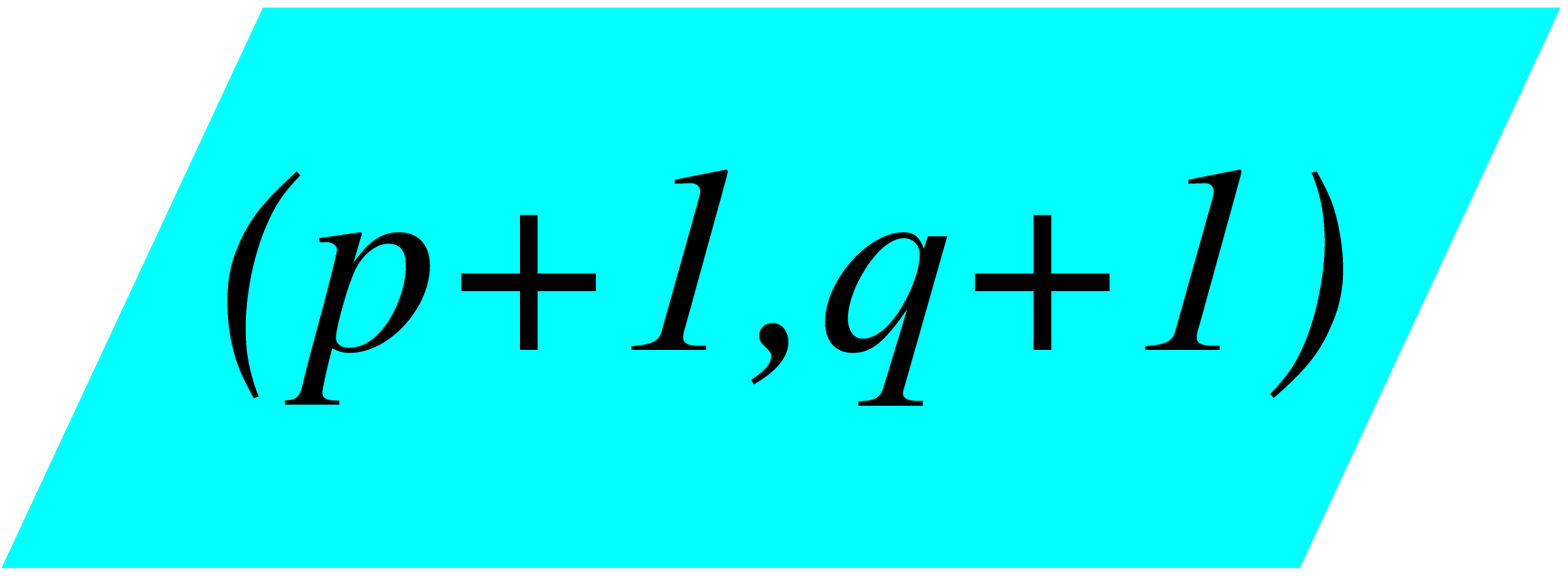} &\text{if}\quad r=N-1-i\\ 
 \figins{-17}{0.4}{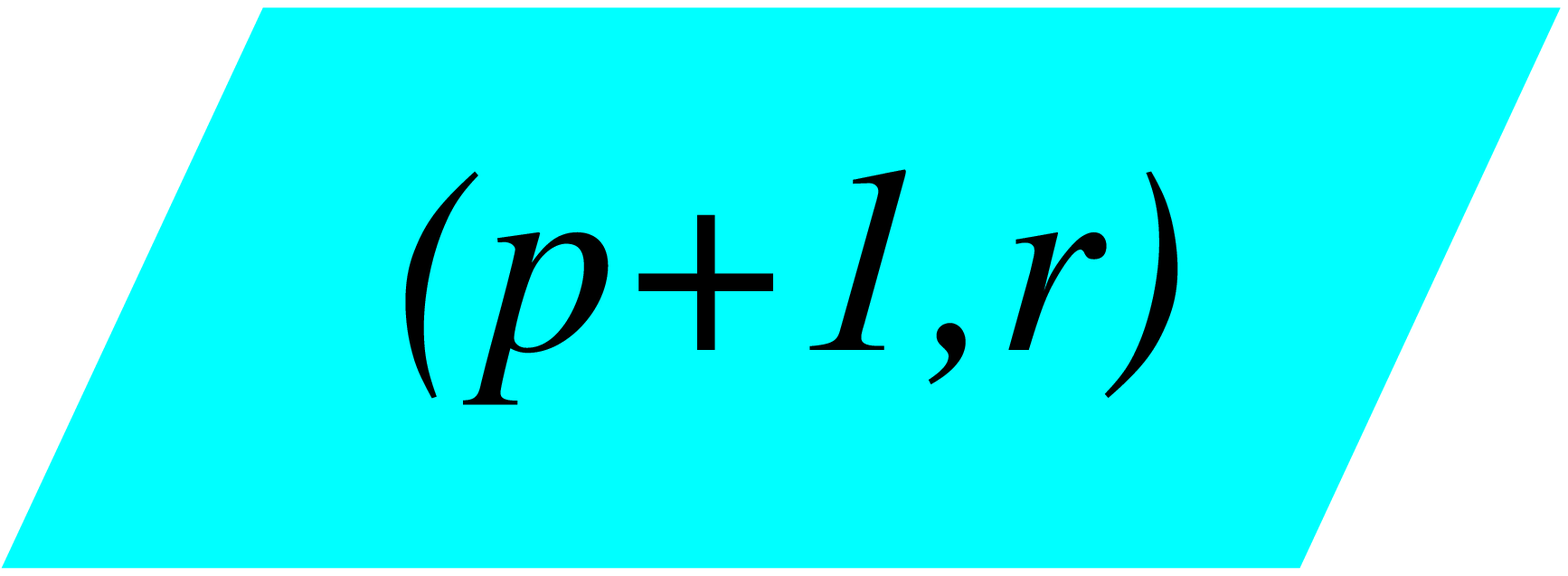}  &\text{if}\quad q=N-2-i\\
 \qquad 0                      &\text{else}
\end{cases}
\end{equation}

\bigskip

\begin{equation}
\label{eq:bubble-ij}
\figins{-20}{0.6}{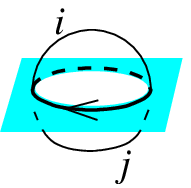}=
\begin{cases}
\quad -\figins{-13}{0.35}{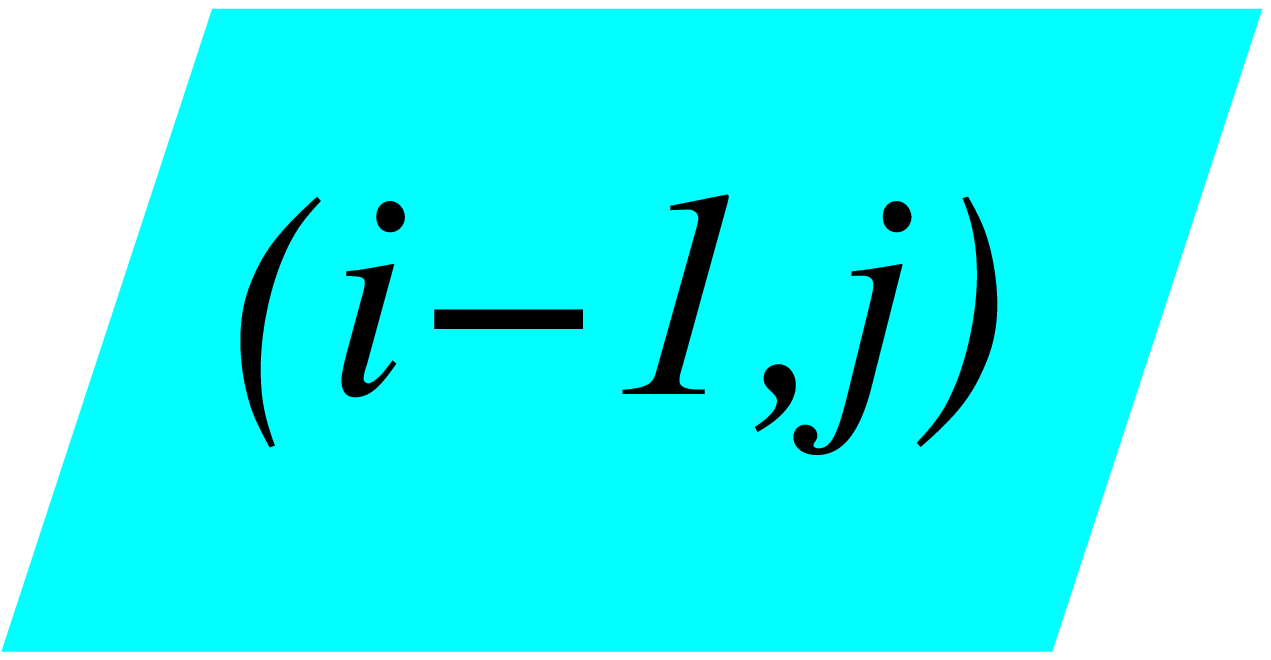} & \text{if}\quad i>j\geq 0 \\
\figins{-13}{0.35}{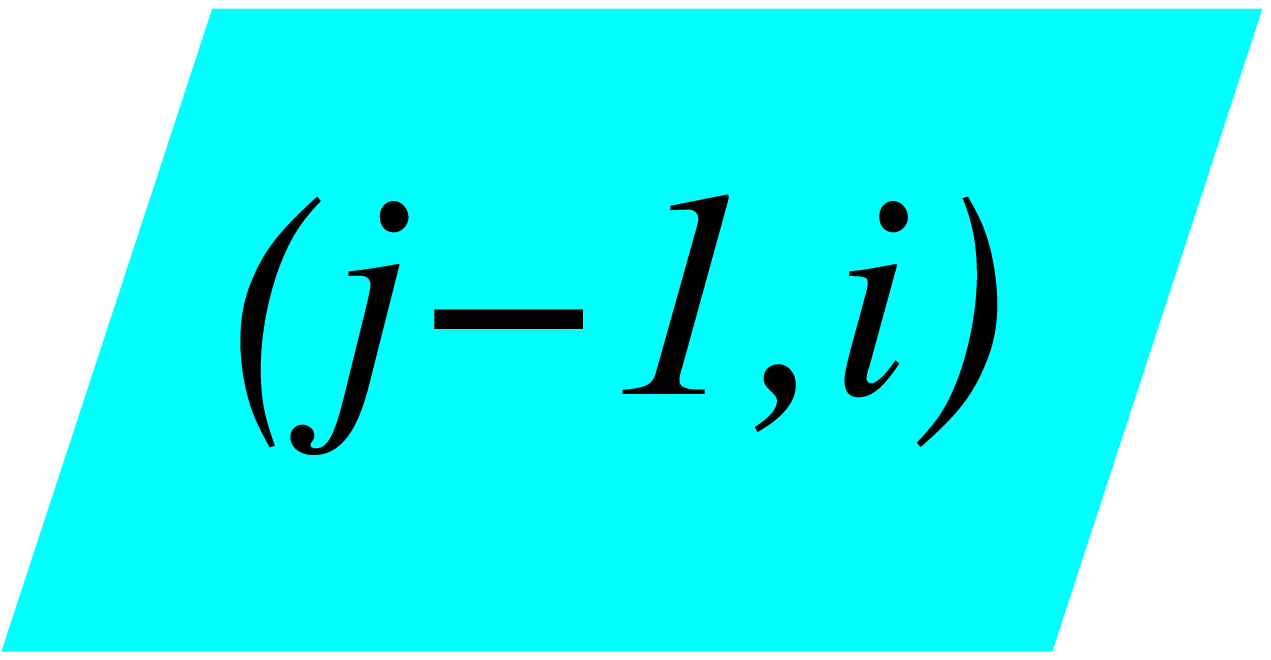} & \text{if}\quad j>i\geq 0 \\
\qquad 0 & \text{if}\quad i=j
\end{cases} 
\end{equation}
\end{prop}

\section{The functors $\Fmv$}
\label{ssec:funct}

Let $n\geq 1$ and $N\geq 4$ be arbitrary but fixed. In this section we define a monoidal functor $\Fmv$ between the categories 
$\SSoer$ and $\foam$. 

\n\emph{On objects:} $\Fmv$ sends the empty sequence to $1_n$ and the one-term sequence $(j)$ to $w_j$:
\begin{equation*}
\labellist
\pinlabel $\dotsm$ at 120 66
\pinlabel $\dotsm$ at 544 66
\pinlabel $\dotsm$ at 830 66
\tiny\hair 2pt
\pinlabel $1$   at   6 -15
\pinlabel $2$   at  62 -15
\pinlabel $n$   at 176 -16
\pinlabel $1$   at 488 -15
\pinlabel $j$   at 655 -15
\pinlabel $j+1$ at 717 -16
\pinlabel $n$   at 886 -15
\endlabellist
(\emptyset)\ \longmapsto\figins{-20}{0.7}{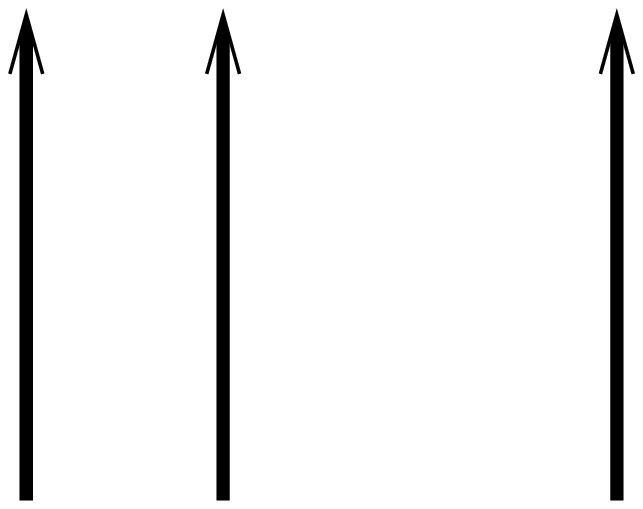}
\mspace{100mu}
(j)\ \longmapsto \figins{-20}{0.7}{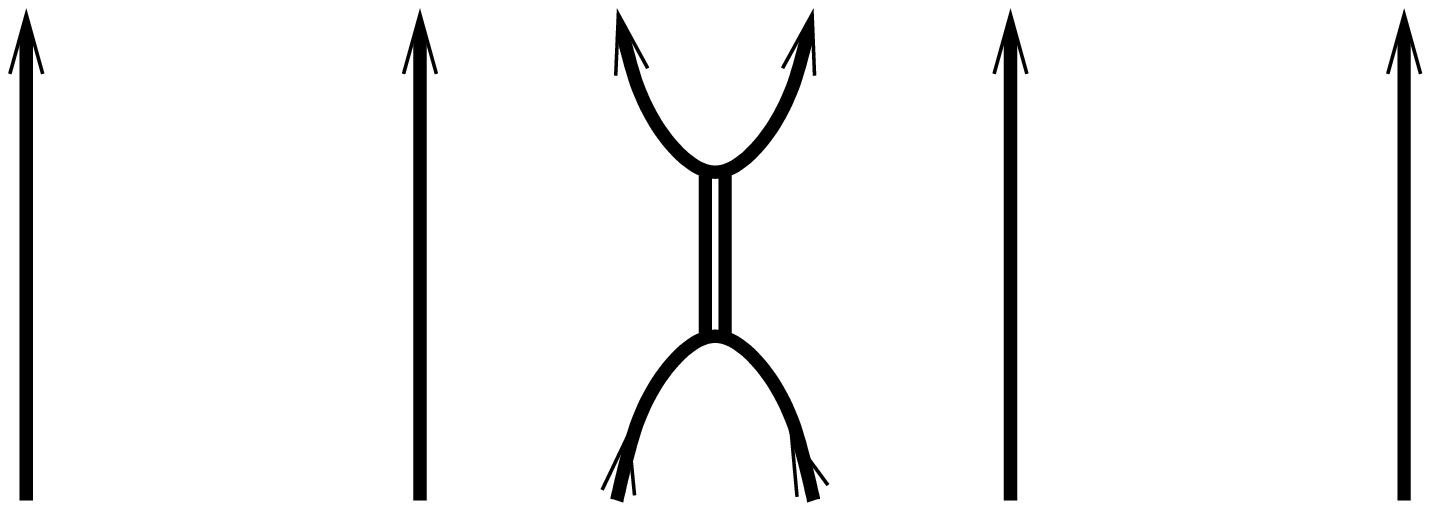}\vspace*{2ex}
\end{equation*}
with $\Fmv(jk)$ given by the vertical composite $w_{j}w_{k}$.

\n\emph{On morphisms:}

\begin{itemize}
\item The empty diagram is sent to $n$ parallel vertical sheets:
\begin{equation*}
\labellist
\pinlabel $\dotsm$ at 185 120
\tiny\hair 2pt
\pinlabel $1$   at   0 -14
\pinlabel $2$   at  59 -14
\pinlabel $n-1$ at 219 -14
\pinlabel $n$   at 282 -14
\endlabellist
\emptyset \ \longmapsto \figins{-28}{1.0}{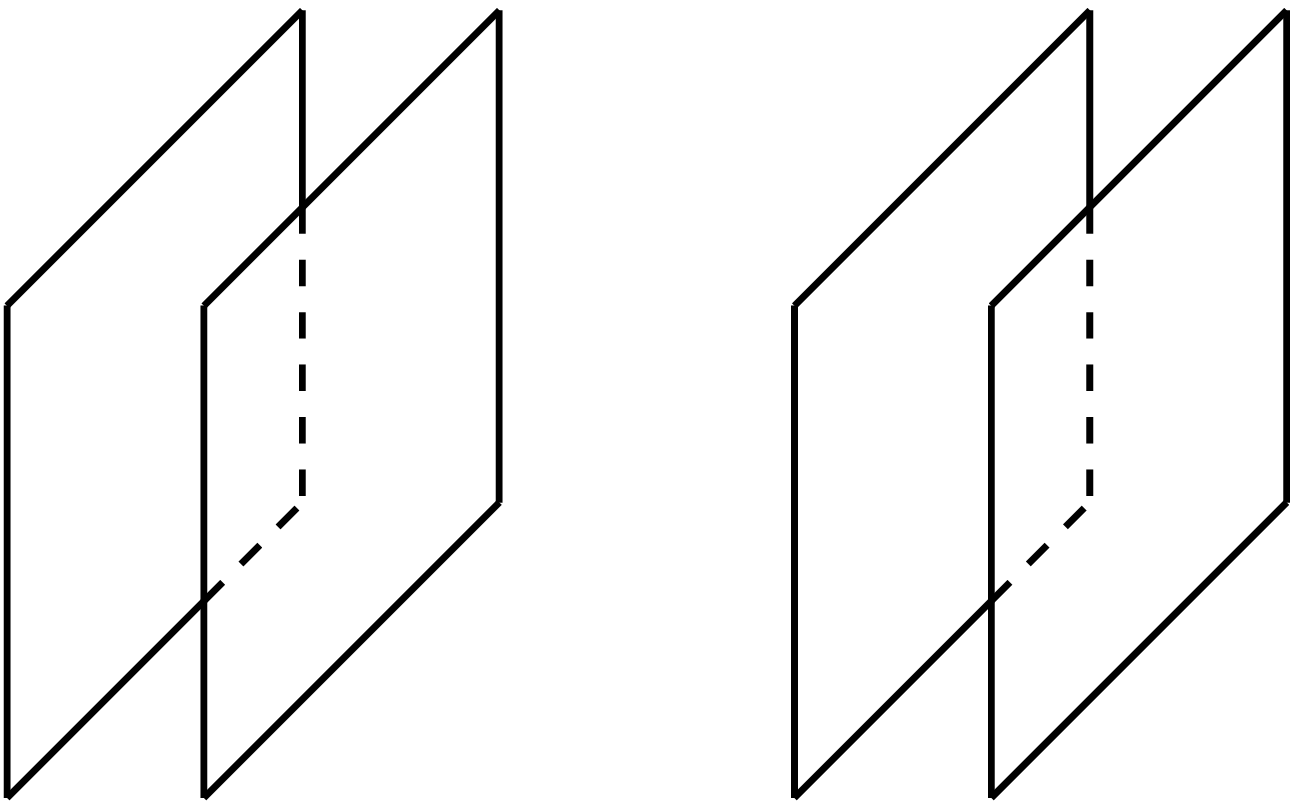}\vspace*{2ex}
\end{equation*}
\item The vertical line coloured $j$ is sent to the identity cobordism of $w_j$:
\begin{equation*}
\labellist
\tiny\hair 2pt
\pinlabel $j$   at -10  60
\pinlabel $j$   at 105 -65
\pinlabel $j+1$ at 170 -66
\endlabellist
\figins{-16}{0.5}{line}\ \
\longmapsto\
\figins{-32}{1.0}{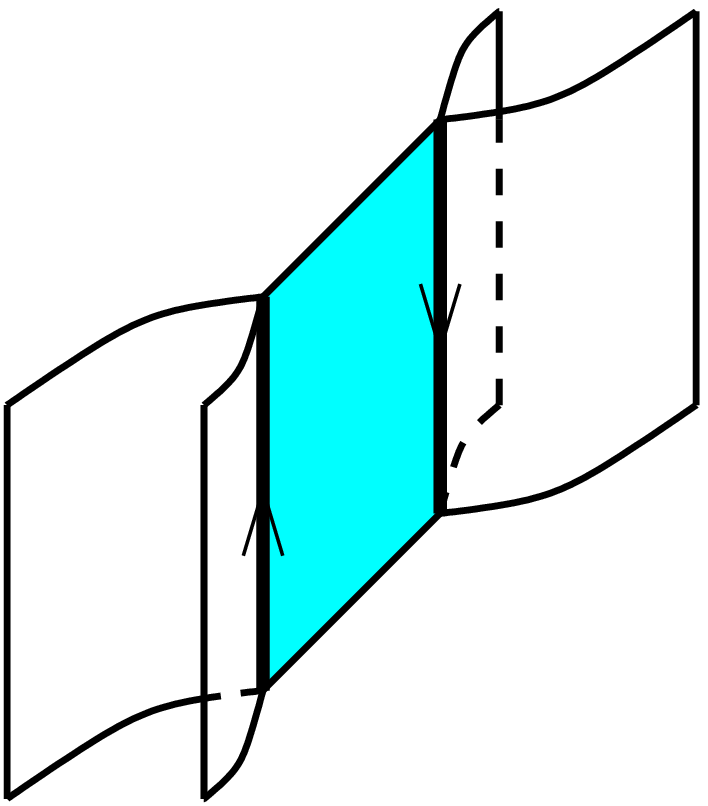}
\vspace*{2ex}
\end{equation*}
The remaining $n-2$ vertical parallel sheets on the r.h.s are not shown for simplicity, a convention that we will follow from now on.

\item The \emph{StartDot} and \emph{EndDot} morphisms are sent to the zip and the unzip respectively:
\begin{equation*}
\labellist
\tiny\hair 2pt
\pinlabel $j$   at -10  60
\pinlabel $j$   at  95 -65
\pinlabel $j+1$ at 154 -66
\endlabellist
\figins{-16}{0.5}{startdot}
\longmapsto\
\figins{-32}{1.0}{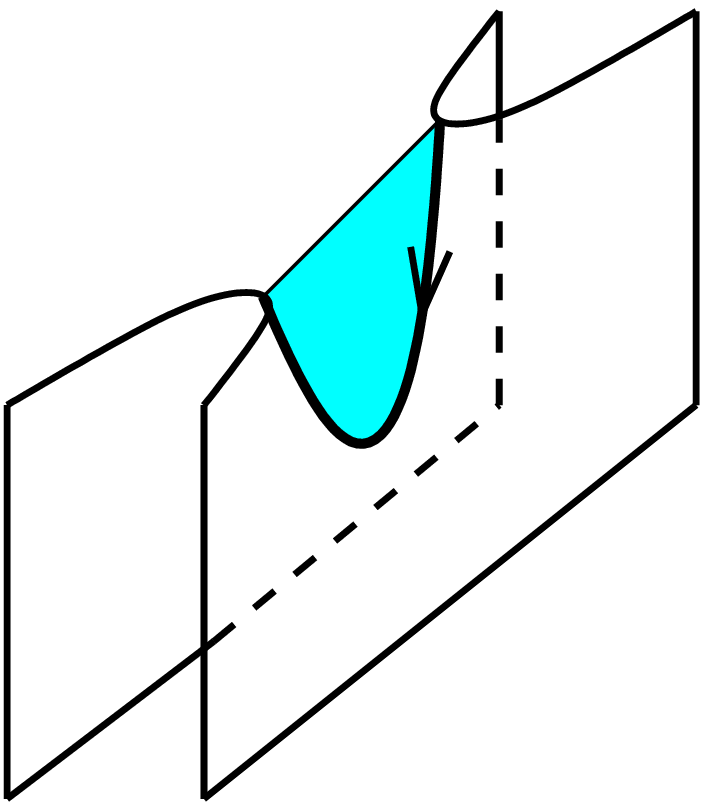}\
\mspace{140mu}
\labellist
\tiny\hair 2pt
\pinlabel $j$   at -10  60
\pinlabel $j$   at  95 -65
\pinlabel $j+1$ at 154 -66
\endlabellist
\figins{-16}{0.5}{enddot}
\longmapsto\
\figins{-32}{1.0}{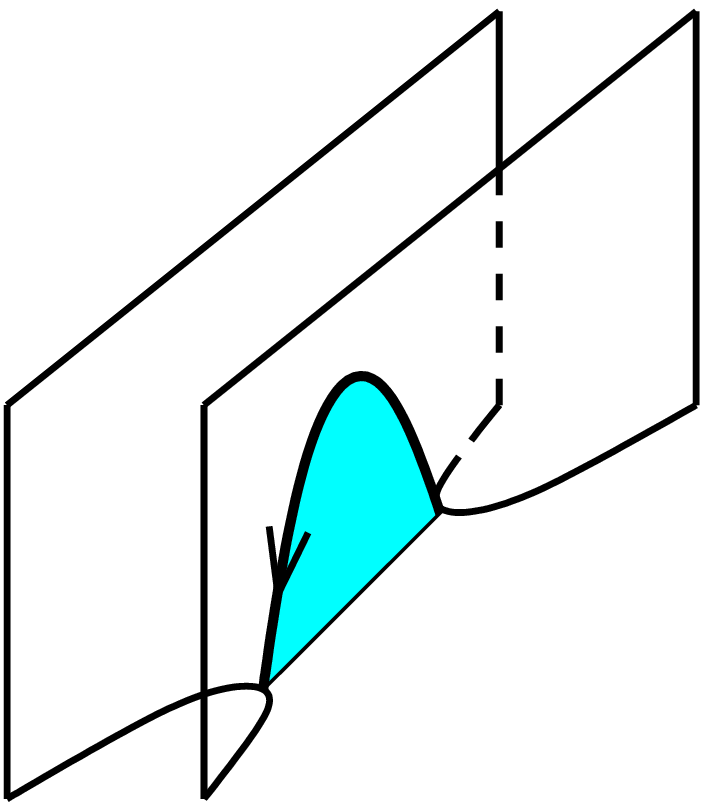}
\vspace*{2ex}
\end{equation*}

\item \emph{Merge} and \emph{Split} are sent to cup and cap cobordisms:
\begin{equation*}
\labellist
\tiny\hair 2pt
\pinlabel $j$   at  45  95
\pinlabel $j$   at 209 -65
\pinlabel $j+1$ at 279 -66
\endlabellist
\figins{-16}{0.6}{merge}
\longmapsto\
\figins{-32}{1.1}{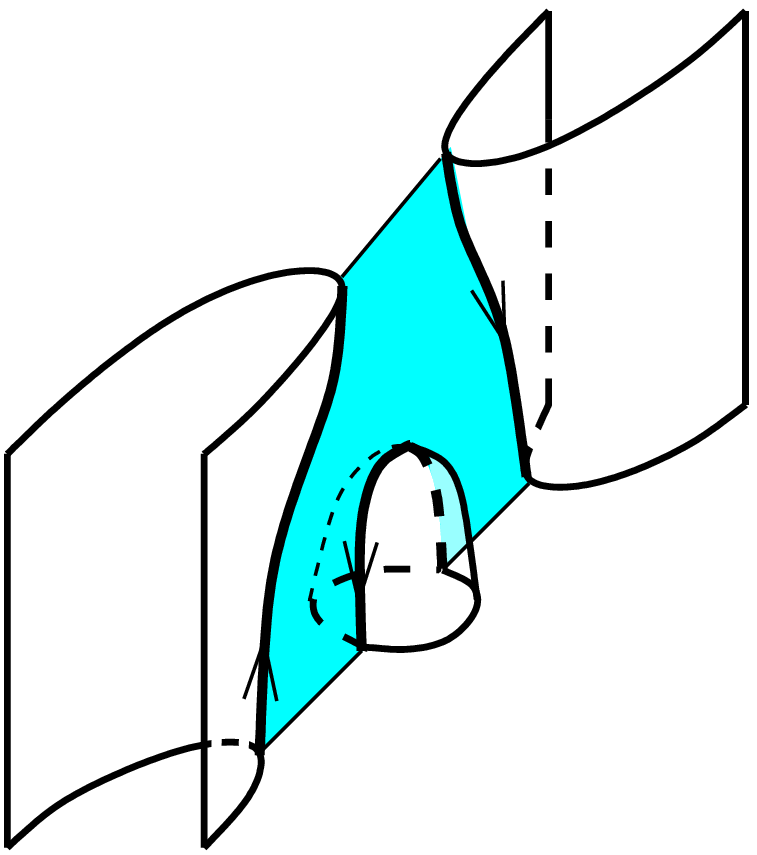}\
\mspace{80mu}
\labellist
\tiny\hair 2pt
\pinlabel $j$   at  45  45
\pinlabel $j$   at 209 -65  
\pinlabel $j+1$ at 279 -66
\endlabellist
\figins{-16}{0.6}{split}
\longmapsto \
\figins{-32}{1.1}{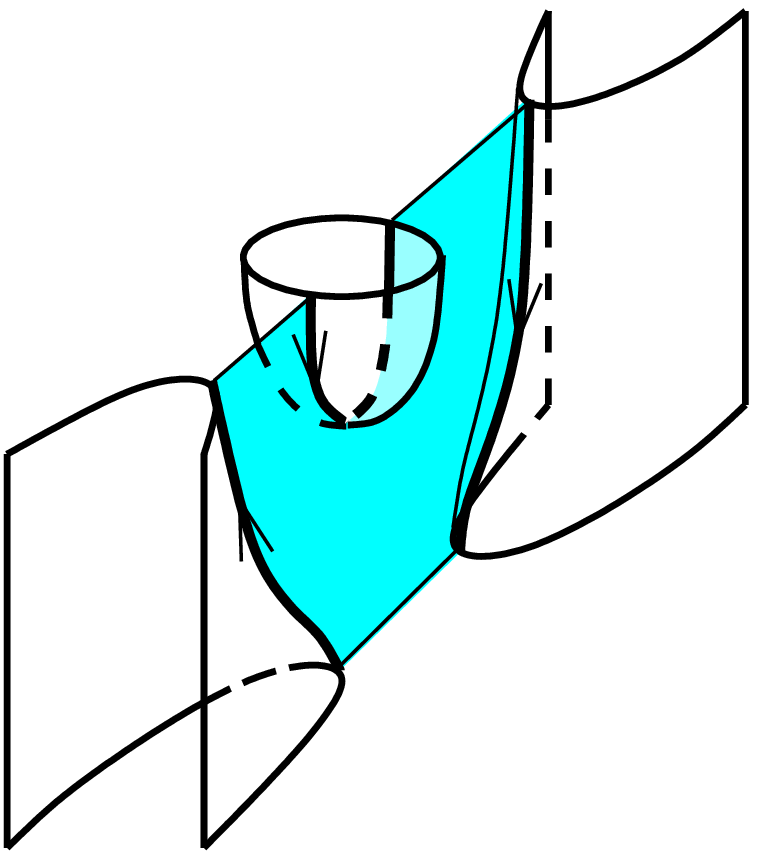}
\vspace*{2ex}
\end{equation*}

\item The \emph{4-valent vertex} with distant colors. For $j+1<k$ we have:
\begin{equation*}
\labellist
\tiny\hair 2pt
\pinlabel $k$   at  -5 -12
\pinlabel $j$   at 128 -10
\pinlabel $j$      at 205 -61
\pinlabel $j+1$    at 270 -62
\pinlabel $k$      at 385 -61
\pinlabel $k+1$    at 445 -62
\pinlabel $\dotsm$ at 335 -62
\endlabellist
\figins{-16}{0.6}{4vert}
\longmapsto\ 
\figins{-32}{1.1}{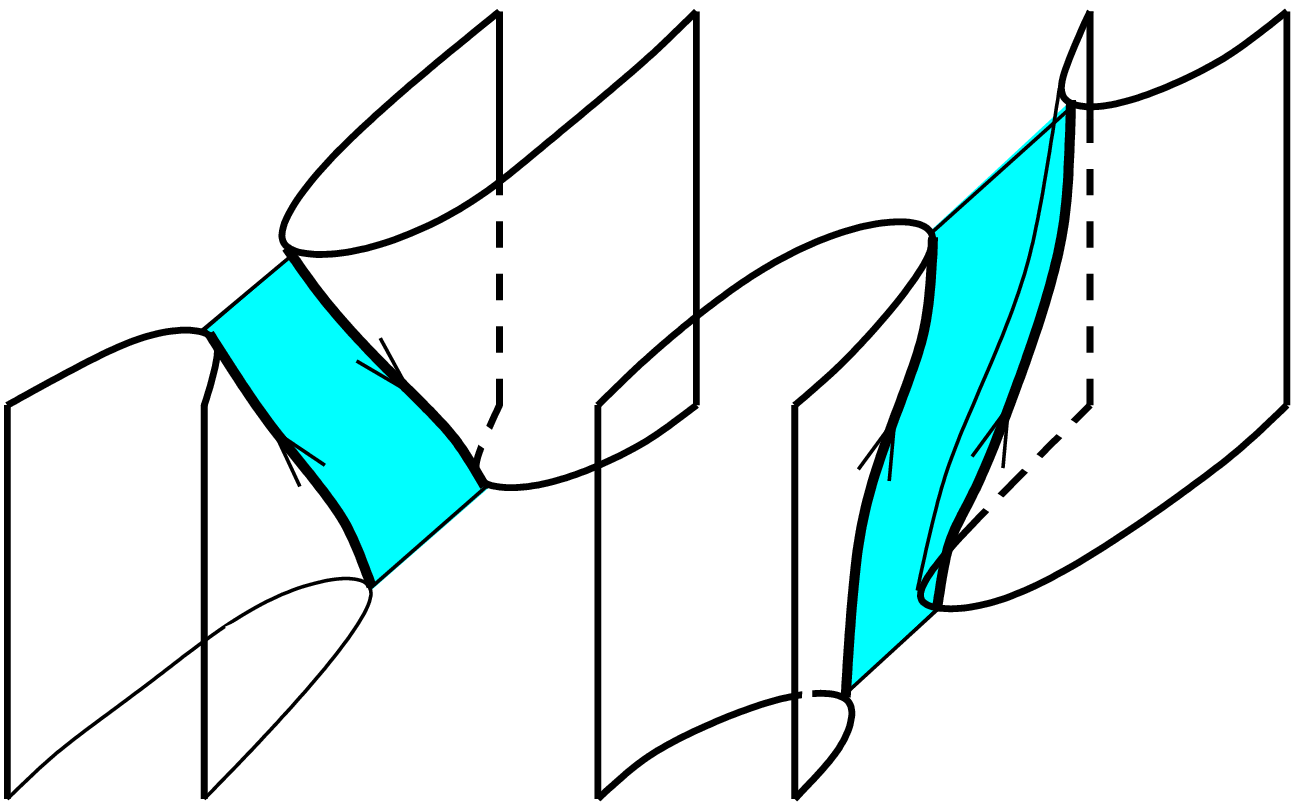}
\vspace*{2ex}
\end{equation*}
The case $j>k+1$ is given by reflexion around a horizontal plane.

\item For the \emph{6-valent vertices} we have:
\begin{equation*}
\labellist
\tiny\hair 2pt
\pinlabel $j+1$ at -5 -10
\pinlabel $j$   at 65 -10
\pinlabel $j$   at 205  17  
\pinlabel $j+1$ at 210  -5  
\pinlabel $j+2$ at 220 -29
\endlabellist
\figins{-18}{0.6}{6vertu}\
\longmapsto\ \ \
\figins{-28}{0.9}{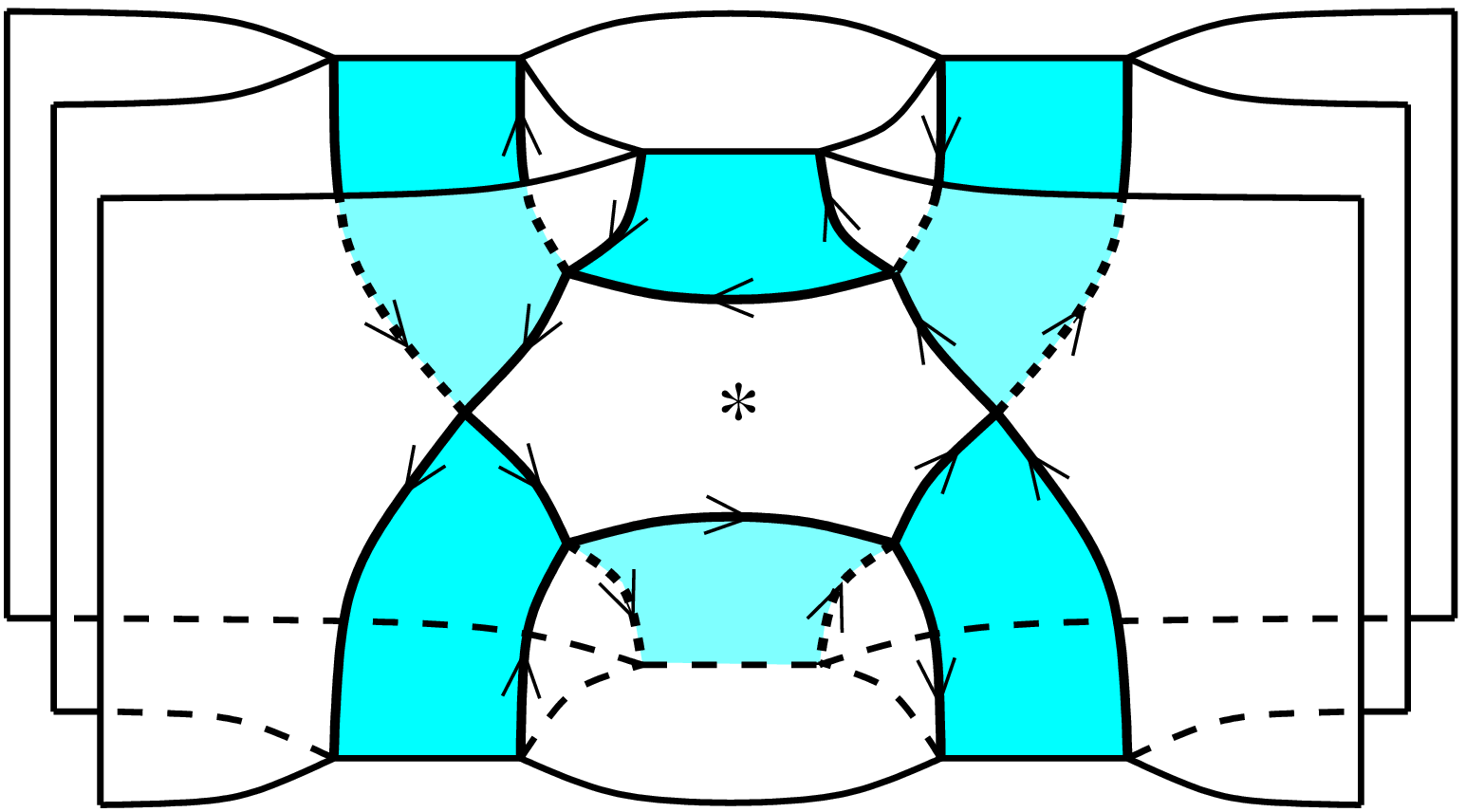}
\vspace*{2ex}
\end{equation*}
\end{itemize}
The case with the colors switched is given by reflection in a vertical plane.
Notice that $\Fmv$ respects the gradings of the morphisms.

\begin{prop}
$\Fmv$ is a monoidal functor.
\end{prop}

\begin{proof}
The assignment given by $\Fmv$ clearly respects the monoidal structures of $\SSoer$ and $\foam$. 
So we only need to show that $\Fmv$ is a functor, i.e. it respects the relations~\eqref{eq:adj} to~\eqref{eq:dumbdumbsquare} of Section~\ref{sec:soergel}.

\medskip
\n\emph{''Isotopy relations'':} 
Relations~\eqref{eq:adj} to~\eqref{eq:v6rot} are straightforward to check and correspond to isotopies of their images under $\Fmv$. For the sake of completeness we show the first equality in~\eqref{eq:adj}. We have
\begin{equation*}
\labellist
\tiny\hair 2pt
\pinlabel $j$  at -6  40        \pinlabel $j$ at 720 40
\pinlabel $j$ at 160 -52        \pinlabel $j+1$ at 209 -53
\pinlabel $j$ at 378 -52        \pinlabel $j+1$ at 431 -53
\endlabellist
\Fmv\Biggl(\
\figins{-17}{0.55}{biadj-l}\
\Biggr)
=\
\figins{-32}{1.0}{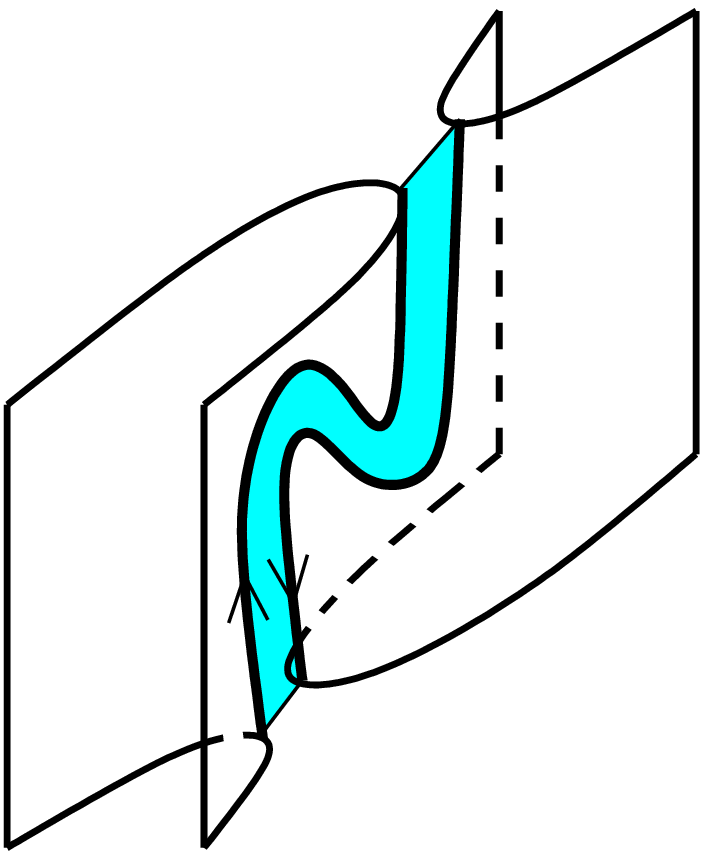}\ 
\cong\
\figins{-32}{1.0}{idwi}\ 
=\
\Fmv
\Biggl(\
\figins{-17}{0.55}{line}\
\Biggr)\vspace*{2ex}
\end{equation*}

\medskip
\n\emph{One color relations:} 
For relation~\eqref{eq:dumbrot} we have
\begin{equation*}
\Fmv\Biggl(\
\figins{-14}{0.45}{dumbellh}\
\Biggr)
\cong
\Fmv\Biggl(\
\figins{-14}{0.45}{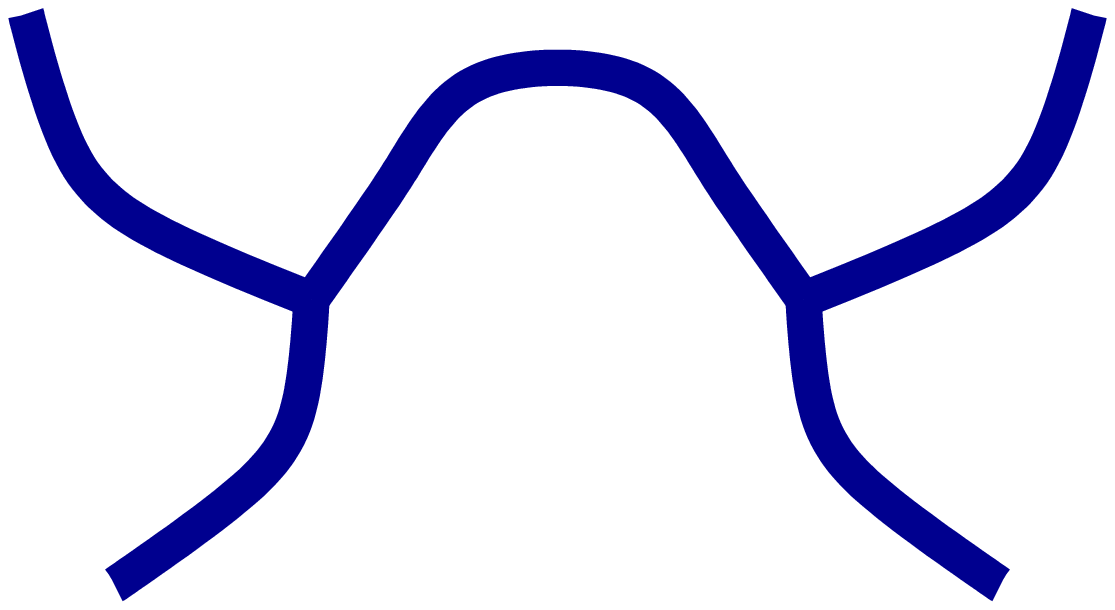}\
\Biggr)
\cong
\Fmv\Biggl(\
\figins{-14}{0.45}{dumbells}\
\Biggr).
\end{equation*}
where the first equivalence follows from relations~\eqref{eq:adj} and~\eqref{eq:v3rot} and the second from isotopy of 
the foams involved.

For relation~\eqref{eq:lollipop} we have
\begin{equation*}
\labellist
\tiny\hair 2pt
\pinlabel $j$ at 26 45            
\pinlabel $j$   at 235 -93 
\pinlabel $j+1$ at 340 -94
\endlabellist
\Fmv\Biggl(\
\figins{-16}{0.5}{lollipop-u}\
\Biggr)
=\
\figins{-32}{1.0}{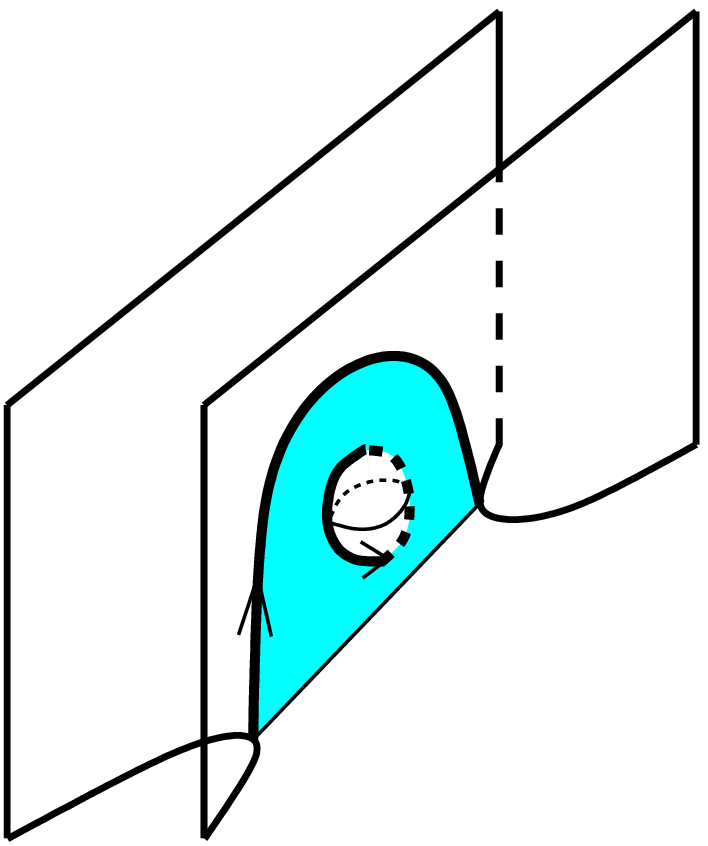}\
=\
0\mspace{30mu}\text{by Equation~\eqref{eq:bubble-ij}}.
\vspace*{2ex}
\end{equation*}

Relation~\eqref{eq:deltam} requires some more work. We have
\begin{align*}
\labellist
\tiny\hair 2pt
\pinlabel $j$ at -2 22    \pinlabel $j$ at -2 130  
\pinlabel $j$ at 130 -70 \pinlabel $j+1$ at 200 -71
\pinlabel $j$ at 405 -70 \pinlabel $j+1$ at 475 -71
\pinlabel $j$ at 720 -70 \pinlabel $j+1$ at 790 -71
\endlabellist
\Fmv\Biggl(\
\figins{-17}{0.55}{matches-ud}\
\Biggr)
&=\ 
\figins{-32}{1.0}{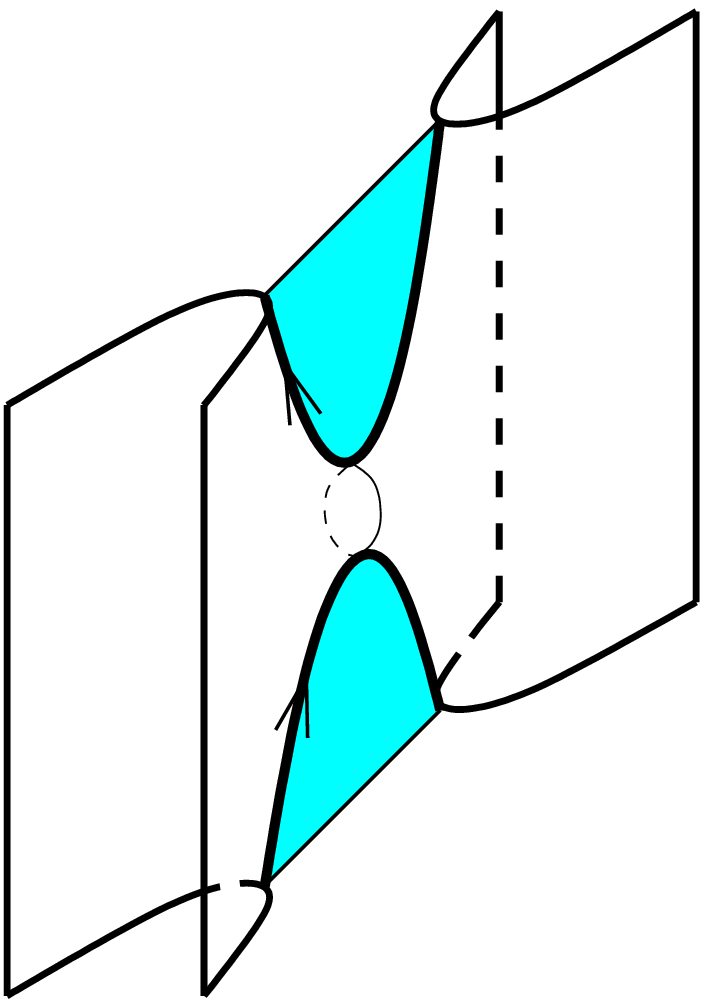}\
=\
\figins{-32}{1.0}{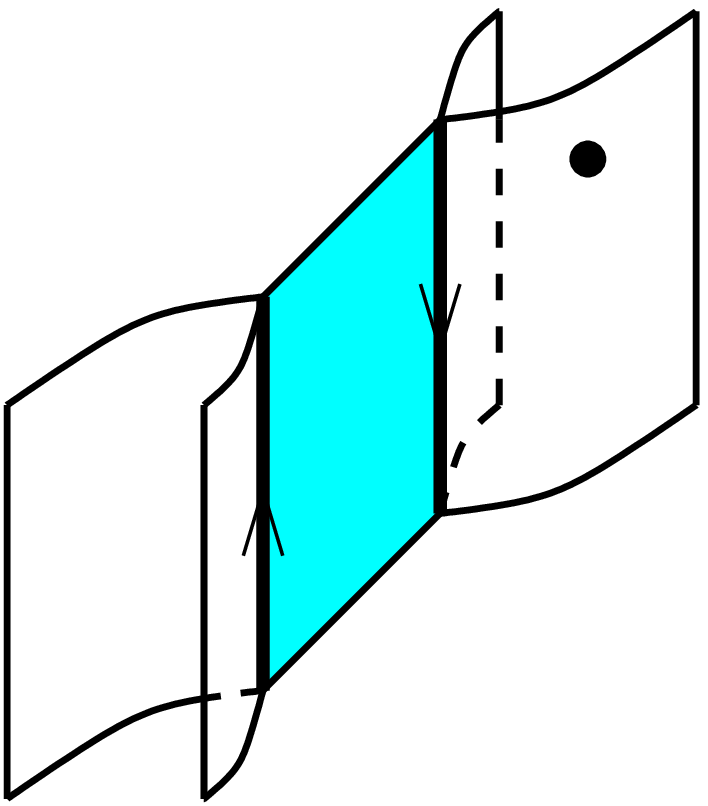}\
-\
\figins{-32}{1.0}{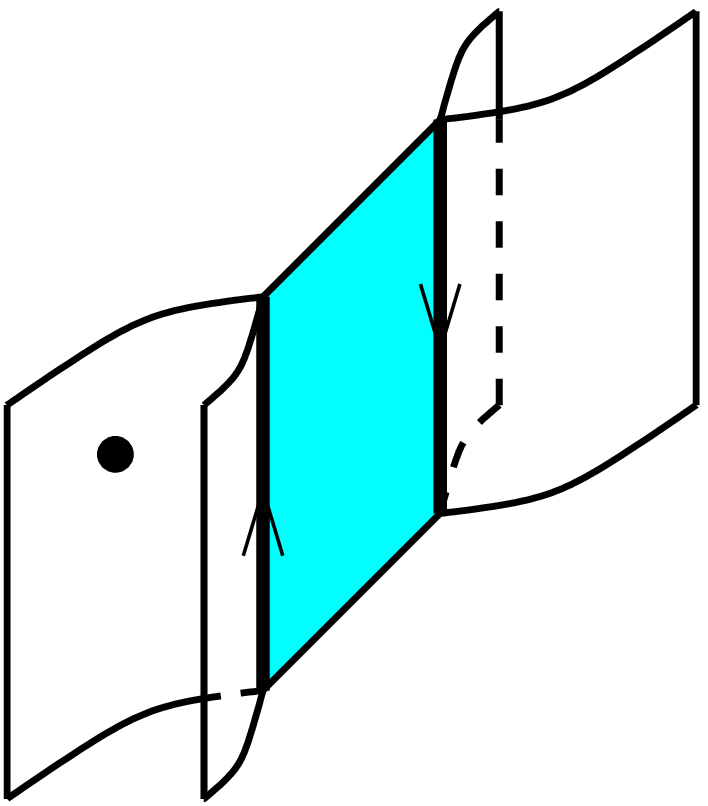}
\vspace*{2ex}
\end{align*}

\bigskip

\n where the second equality follow from the~\eqref{eq:dr1} relation. We also have
\begin{equation}\label{eq:Fdot}
\labellist
\tiny\hair 2pt
\pinlabel $j$ at -13   16
\pinlabel $j$ at 140 -130   \pinlabel $j+1$ at 220 -131
\pinlabel $j$ at 440 -130   \pinlabel $j+1$ at 590 -131
\pinlabel $j$ at 785 -130   \pinlabel $j+1$ at 940 -131
\endlabellist
\Fmv\biggl(\
\figins{-6}{0.25}{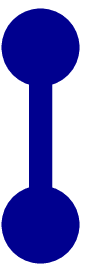}\
\biggr)\
=\
\figins{-32}{1.0}{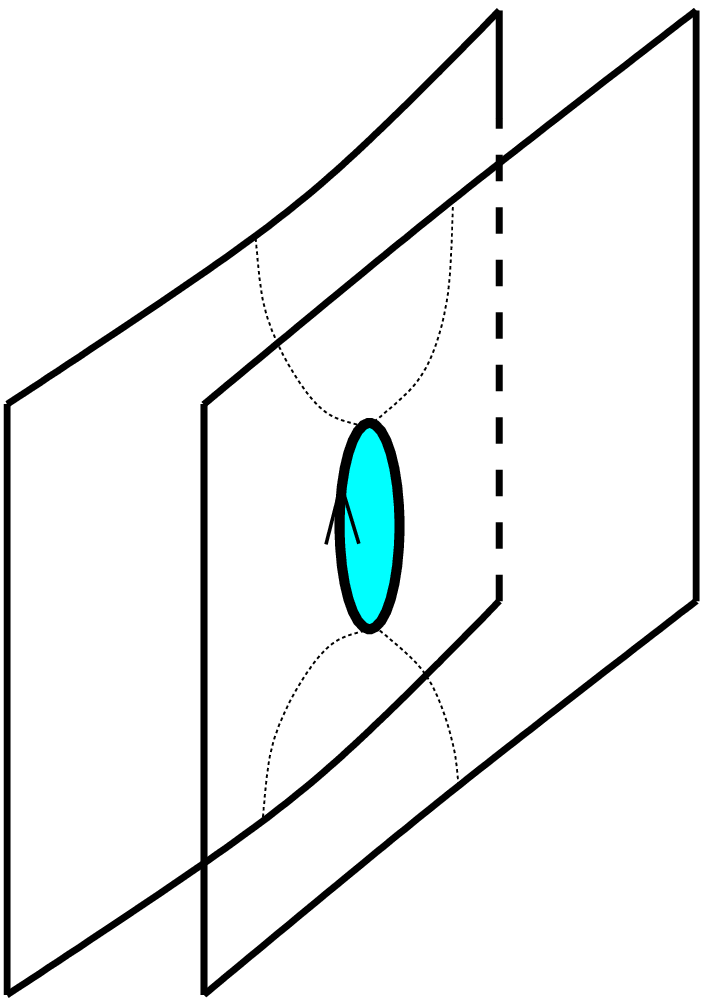}\
=\
\figins{-32}{1.0}{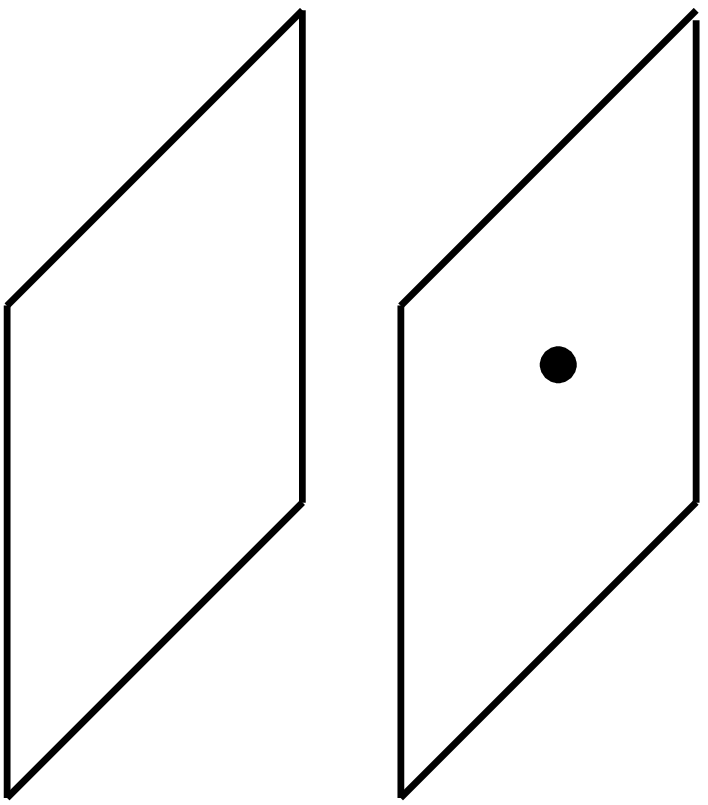}\
-\
\figins{-32}{1.0}{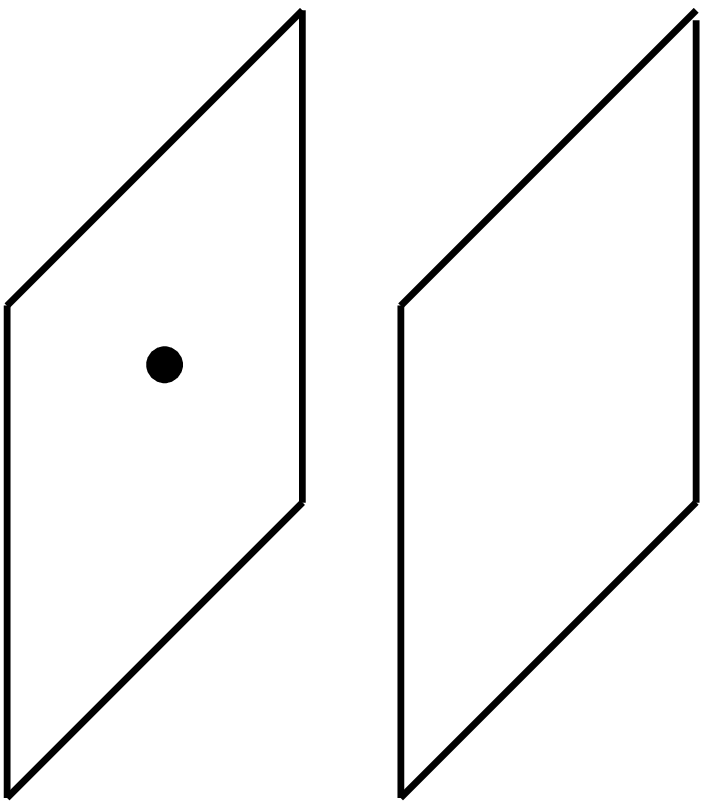}
\vspace*{2ex}
\end{equation}

\medskip

Using~\eqref{eq:dotmigr} we obtain
\begin{align}
\label{eq:edge-dots}
\labellist
\tiny\hair 2pt
\pinlabel $j$ at -10  60
\pinlabel $j$ at 220 -78  \pinlabel $j+1$ at 301 -79
\pinlabel $j$ at 540 -78  \pinlabel $j+1$ at 621 -79
\endlabellist
\Fmv\Biggl(\
\figins{-17}{0.55}{edge-startenddot}\
\Biggr)
&=\ 2\
\figins{-32}{1.0}{idwi-dot}\
-\
\figins{-32}{1.0}{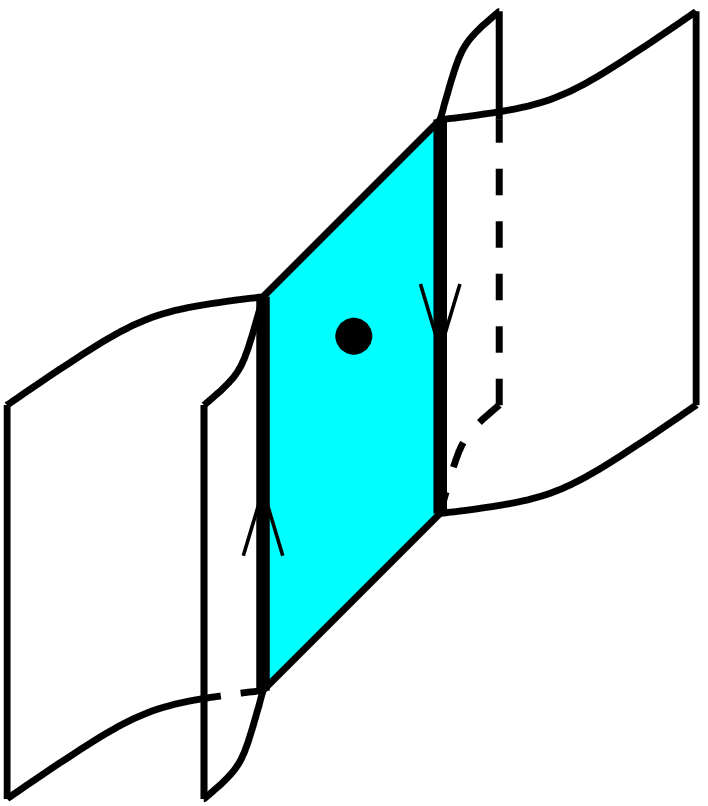}
\displaybreak[0]
\intertext{and}
\label{eq:dots-edge}
\labellist
\tiny\hair 2pt
\pinlabel $j$ at  65  60
\pinlabel $j$ at 220 -78  \pinlabel $j+1$ at 301 -79
\pinlabel $j$ at 540 -78  \pinlabel $j+1$ at 621 -79
\endlabellist
\Fmv\Biggl(\
\figins{-17}{0.55}{startenddot-edge}\
\Biggr)
&=\ -2\
\figins{-32}{1.0}{dot-idwi}\
+\
\figins{-32}{1.0}{idwid}
\end{align}
and, therefore, we have that
\begin{equation*}
\Fmv
\Biggl(
\figins{-17}{0.55}{startenddot-edge}\
\Biggr)
+\
\Fmv
\Biggl(\
\figins{-17}{0.55}{edge-startenddot}\
\Biggr)
=\ 2\ 
\Fmv
\Biggl(\
\figins{-17}{0.55}{matches-ud}\
\Biggr).
\vspace*{2ex}
\end{equation*}

\medskip
\n\emph{Two distant colors:} 
Relations~\eqref{eq:reid2dist} to~\eqref{eq:slide3v} correspond to isotopies of the foams 
involved and are straightforward to check.

\medskip
\n\emph{Adjacent colors:} 
We prove the case where ``blue'' corresponds to $j$ and ``red'' corresponds to 
$j+1$. The relations with colors reversed are proved the same way.
To prove relation~\eqref{eq:dot6v} we first notice that using the~\eqref{eq:mp} move we get
\begin{equation*}
\labellist
\tiny\hair 2pt
\pinlabel $j$   at 280  0
\pinlabel $j+1$ at 282 -35
\pinlabel $j+2$ at 300 -70
\endlabellist
\Fmv\Biggl(\
\figins{-12}{0.4}{6vertdotd}\
\Biggr)\
\cong\mspace{22mu}
\figins{-28}{0.9}{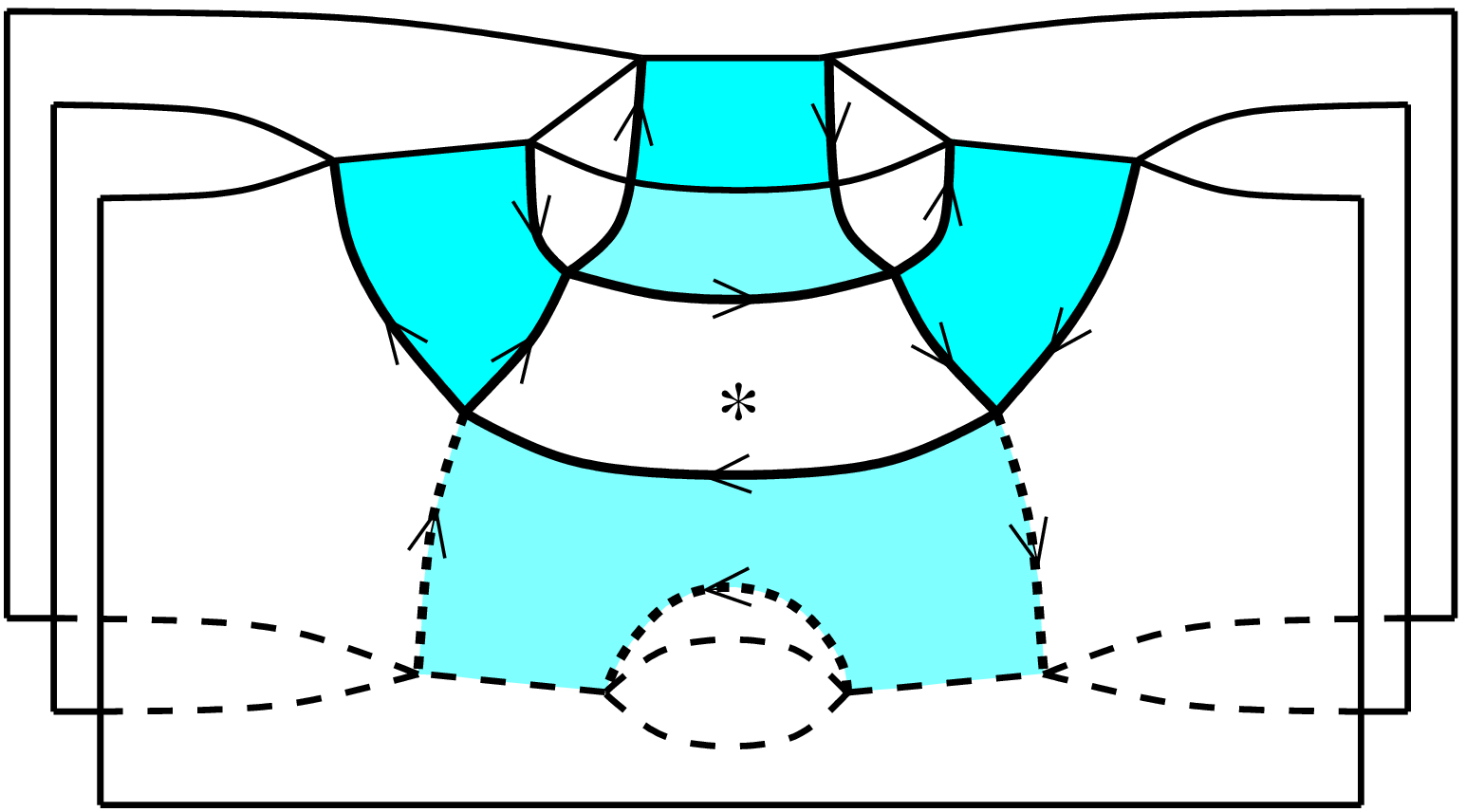}\ .
\vspace*{2ex}
\end{equation*}
Apply~\eqref{eq:sqr1} to the simple-double square tube perpendicular to the triple facet to
obtain two terms. 
The first term contains a double-triple digon tube  which is the left hand side of 
the~\eqref{eq:dr32} relation rotated by $180^0$ around a vertical axis.
Next apply the~\eqref{eq:dr32} relation and use~\eqref{eq:mp} to remove the four 
singular vertices which results in simple-triple bubbles (with dots) in the double facets. 
Using Equation~\eqref{eq:bubble-pqri} to remove these bubbles gives
\begin{equation*}
\labellist
\tiny\hair 2pt
\pinlabel $j$   at  -9 -19
\pinlabel $j+1$ at  54 -20
\pinlabel $j+2$ at 122 -20
\endlabellist
\figins{-34}{1.2}{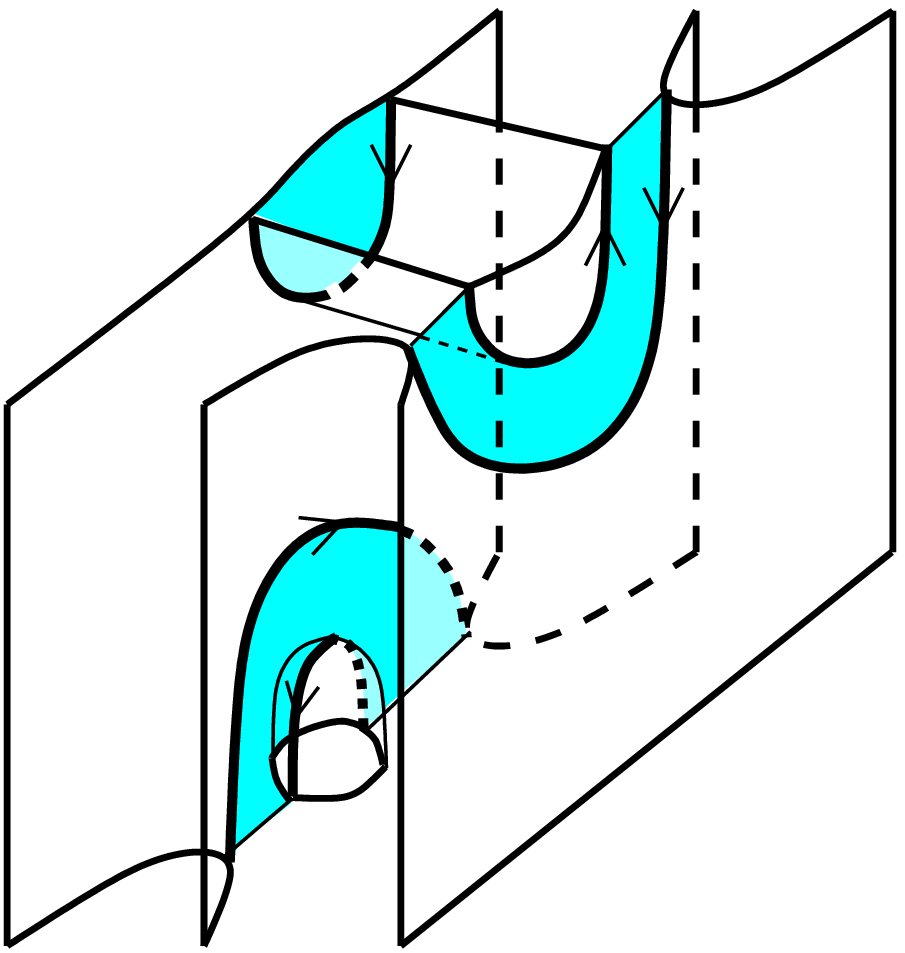}
\vspace*{2ex}
\end{equation*}
which is $\Fmv\bigl(\figins{-6}{0.22}{capcupdot}\bigr)$. The second term contains
\begin{equation*}
\sum\limits_{a+b+c+d=N-3}\;
\labellist
\tiny\hair 2pt
\pinlabel $a$ at  40 150
\pinlabel $b$ at 204 153
\pinlabel $c$ at 150  90
\endlabellist
\figins{-30}{0.8}{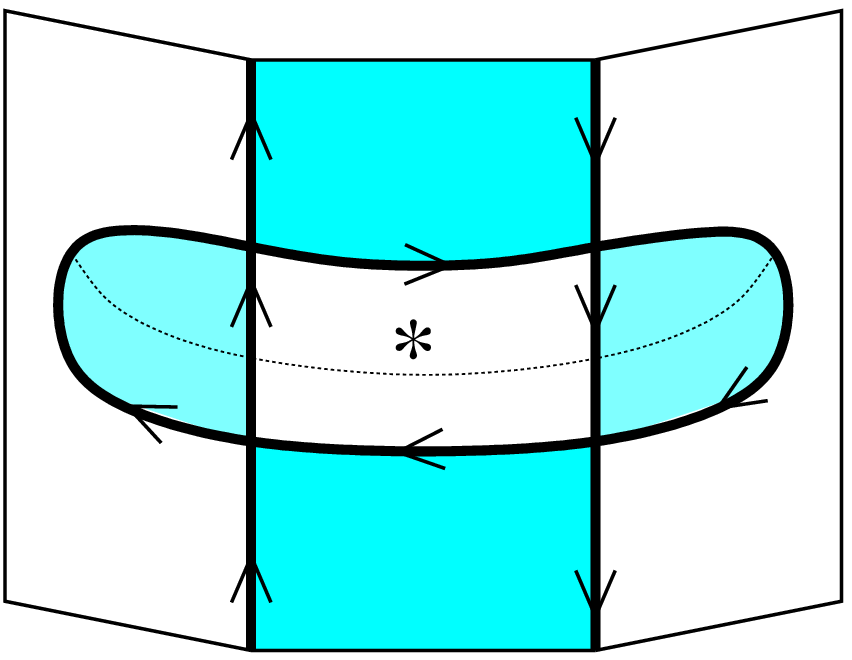}\
\vspace*{2ex}
\end{equation*}
behind a simple facet with $d$ dots (notice that all dots are on simple facets).
Using the~\eqref{eq:mp} relation to get a simple-triple bubble in the double facet, 
followed by~\eqref{eq:rd2} and (S$_1$) we obtain
\begin{equation*}
\labellist
\tiny\hair 2pt
\pinlabel $j$   at -10 -19
\pinlabel $j+1$ at  46 -20
\pinlabel $j+2$ at 126 -20
\endlabellist
\figins{-34}{1.2}{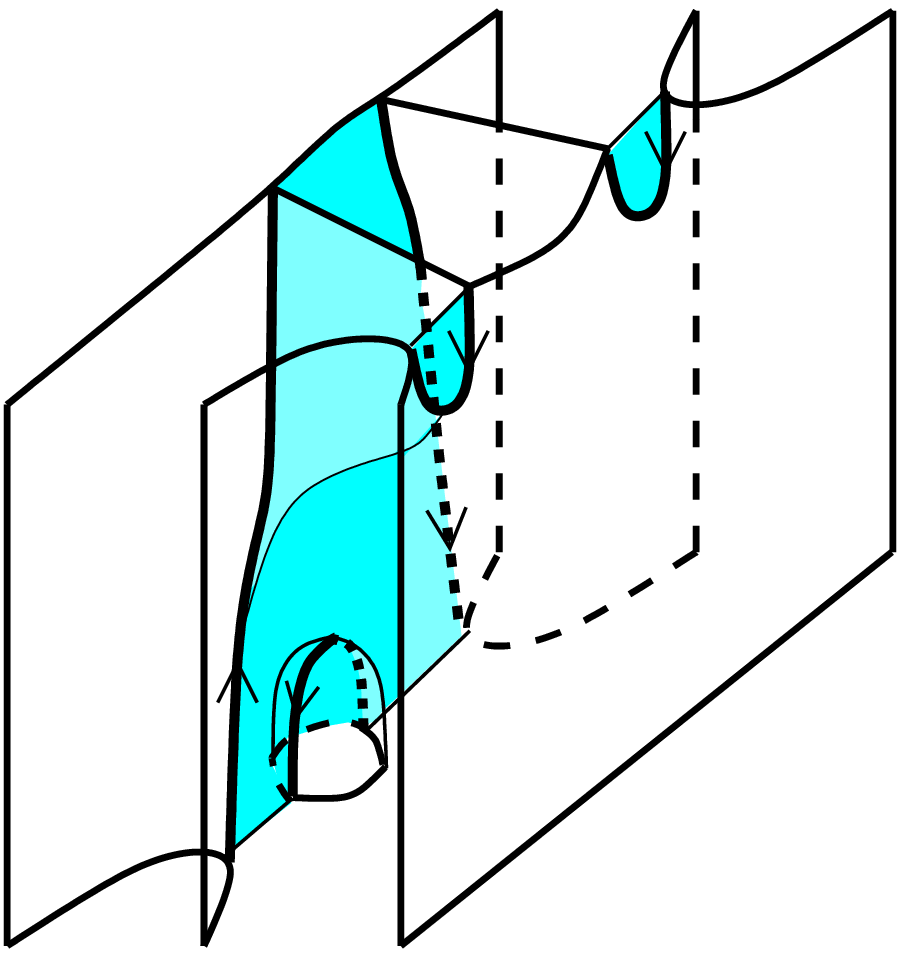}\
\vspace*{2ex}
\end{equation*}
which equals $\Fmv\bigl(\figins{-6}{0.22}{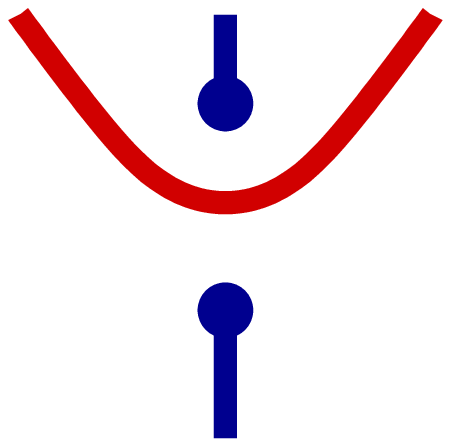}\bigr)$.

We now prove relation~\eqref{eq:reid3}. We have an isotopy equivalence
\begin{equation*}
\labellist
\tiny\hair 2pt
\pinlabel $j$   at 272  25  
\pinlabel $j+1$ at 284 -18 
\pinlabel $j+2$ at 321 -59
\endlabellist
\Fmv
\left(\
\figins{-26}{0.8}{id-r3}\
\right)
 \cong\ \
\figins{-34}{1.2}{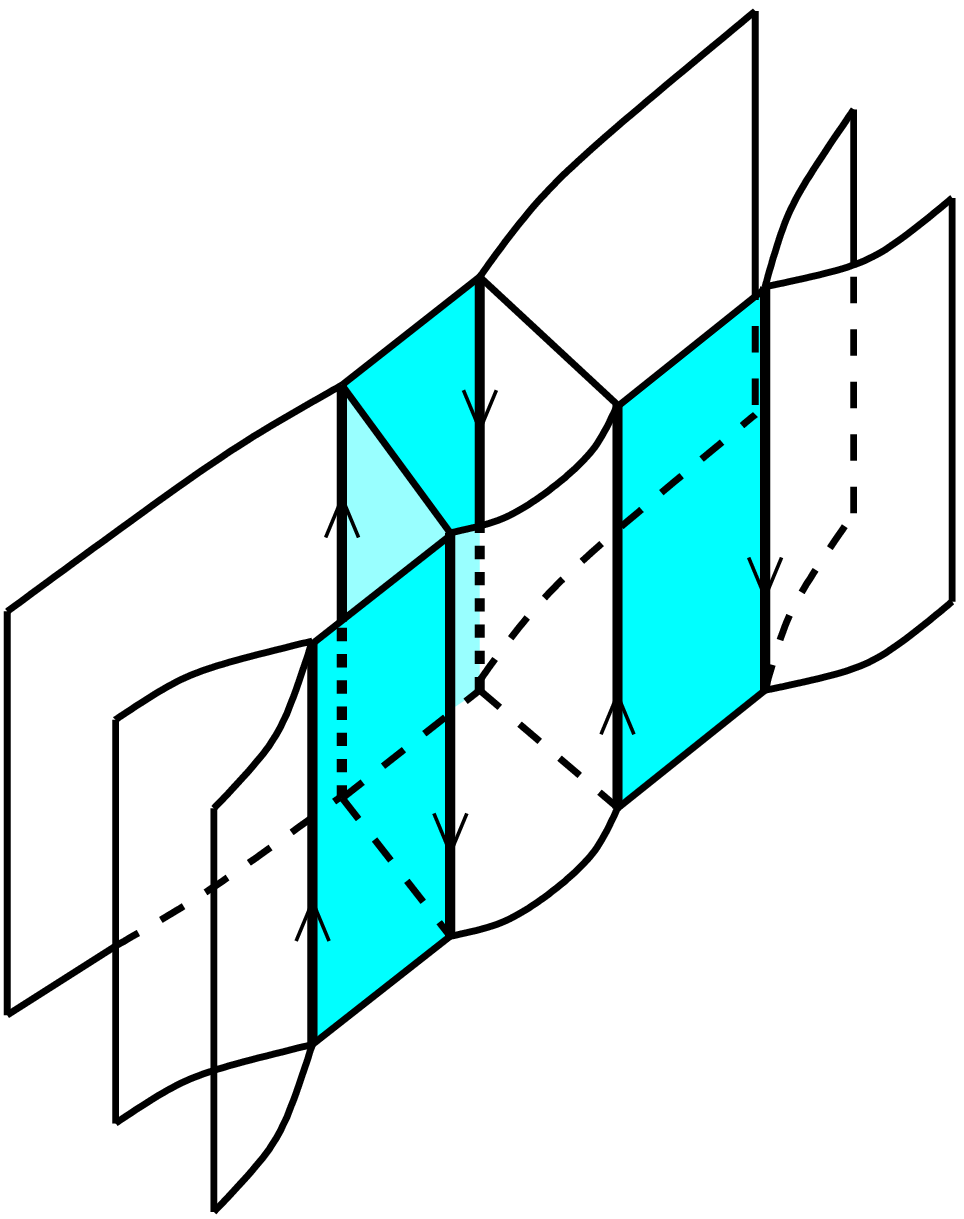}\ 
\vspace*{2ex}
\end{equation*}
Notice that $\Fmv\bigl(\figins{-5}{0.21}{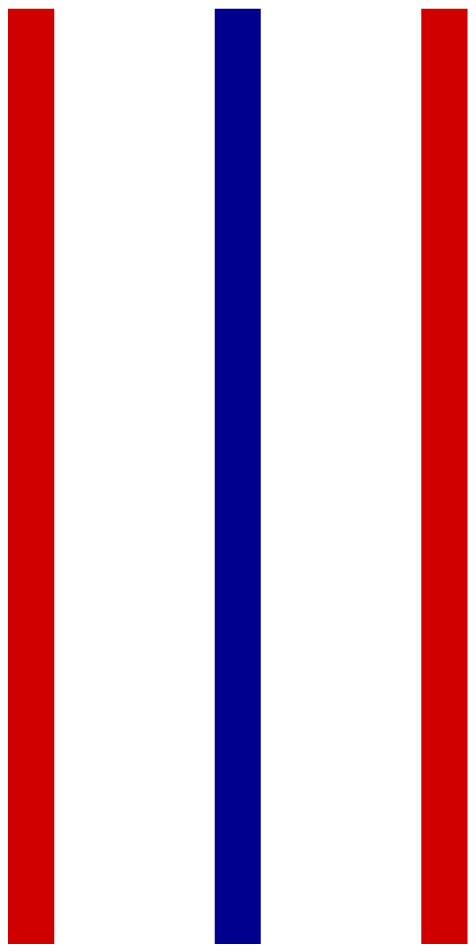}\bigr)$ is the l.h.s. of the
~\eqref{eq:sqr2} relation.
The first term on the r.h.s. of~\eqref{eq:sqr2} is isotopic to 
$-\Fmv\bigl(\figins{-6}{0.22}{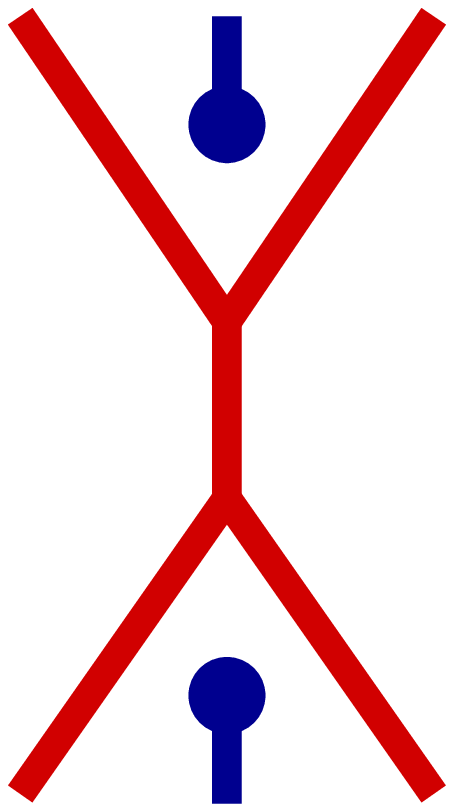}\bigr)$. 
For the second term on the r.h.s. of~\eqref{eq:sqr2}
we notice that
$\Fmv\bigl(\figins{-6}{0.22}{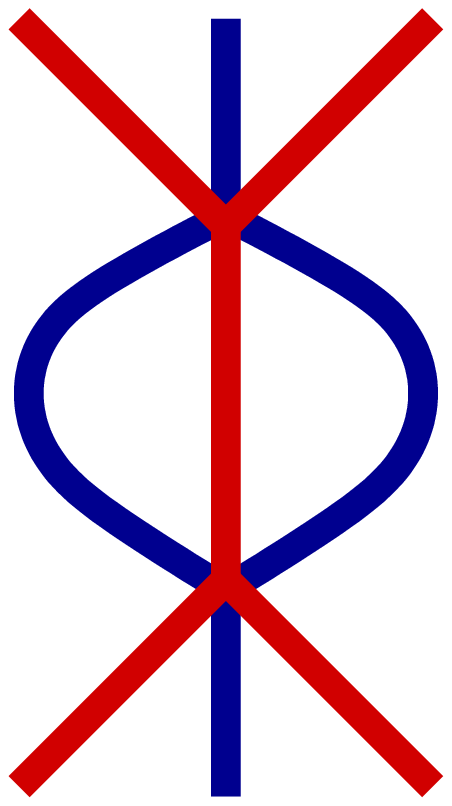}\bigr)$ 
contains 
\begin{equation}
\label{eq:sbub}
\figins{-34}{0.8}{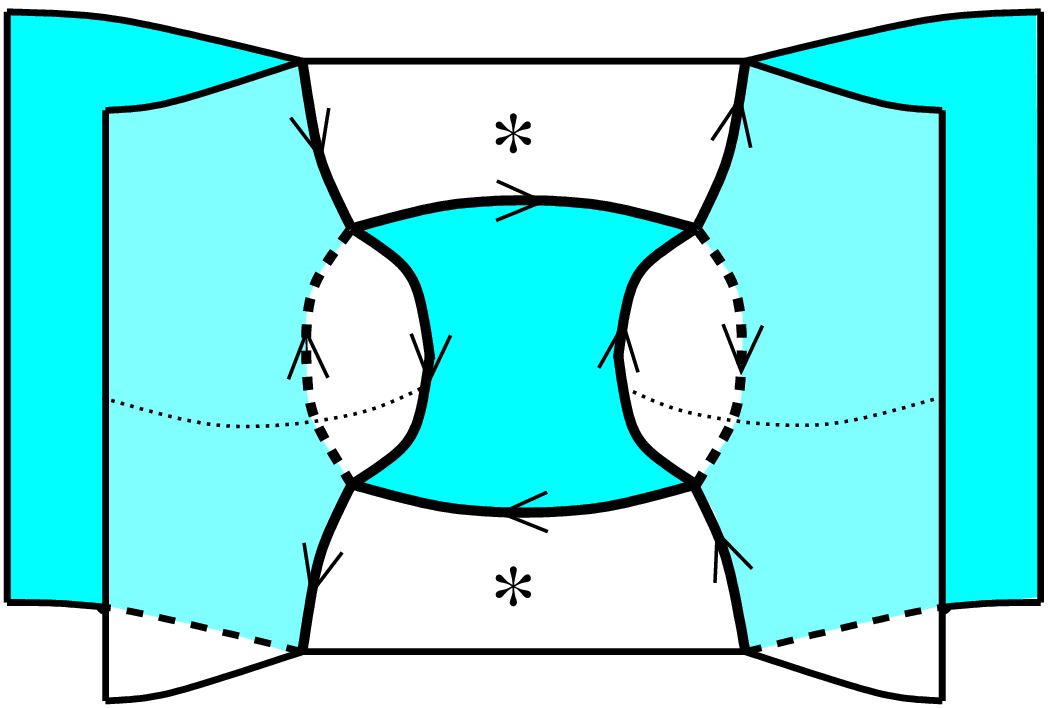}\ 
\vspace*{2ex}
\end{equation}
Applying~\eqref{eq:dr31} followed by~\eqref{eq:mp} to remove the singular vertices creating
simple-simple bubbles on the two double facets and using Equation~\eqref{eq:bubble-ij}
to remove these bubbles we get that 
$\Fmv\bigl(\figins{-6}{0.22}{reid3-short}\bigr)$ 
is the second term on the r.h.s. of~\eqref{eq:sqr2}.

We now prove relation~\eqref{eq:dumbsq} in the form
\begin{equation*}
\figins{-26}{0.75}{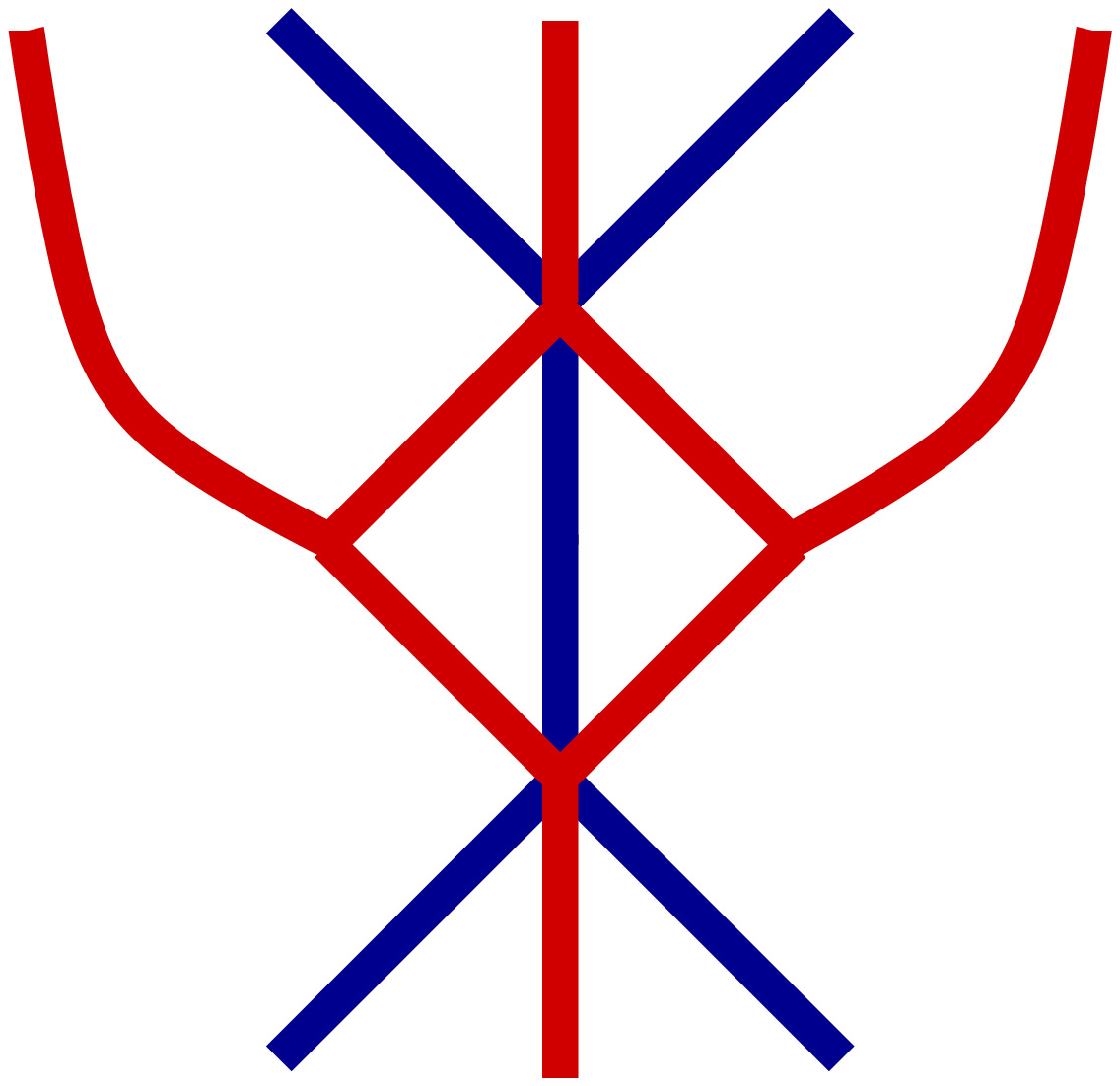}\
=\
\figins{-26}{0.75}{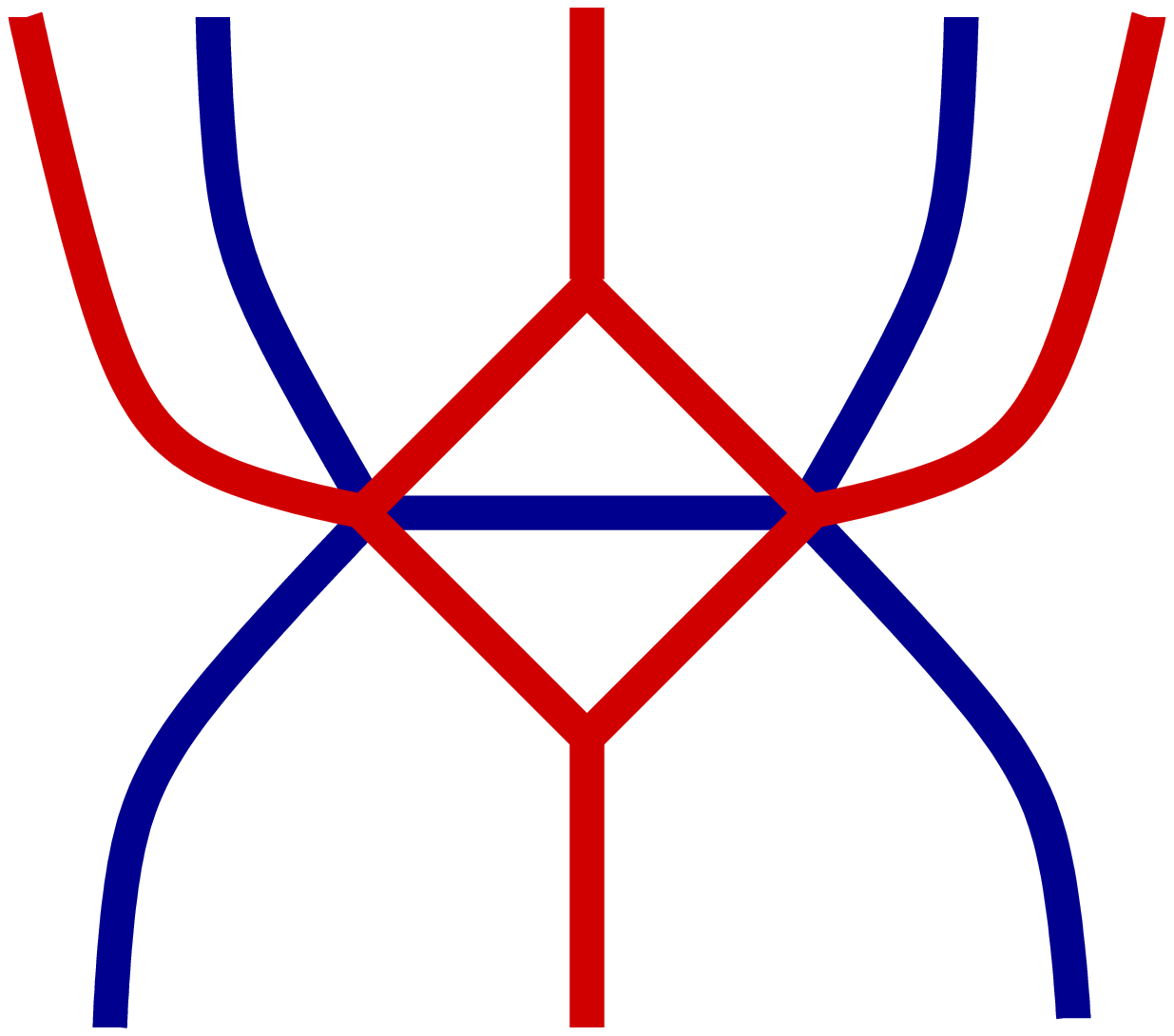}\ .
\end{equation*}
The image of the l.h.s. also contains a bit like the one in Equation~\eqref{eq:sbub}.
Simplifying it like we did in the proof of~\eqref{eq:reid3} we obtain the
$\Fmv$ reduces to
\begin{equation}\label{eq:lhs}
\labellist
\tiny\hair 2pt
\pinlabel $j$   at -55  60 
\pinlabel $j+1$ at -45  28  
\pinlabel $j+2$ at -30  -2
\endlabellist
-\ \
\figins{-34}{1.0}{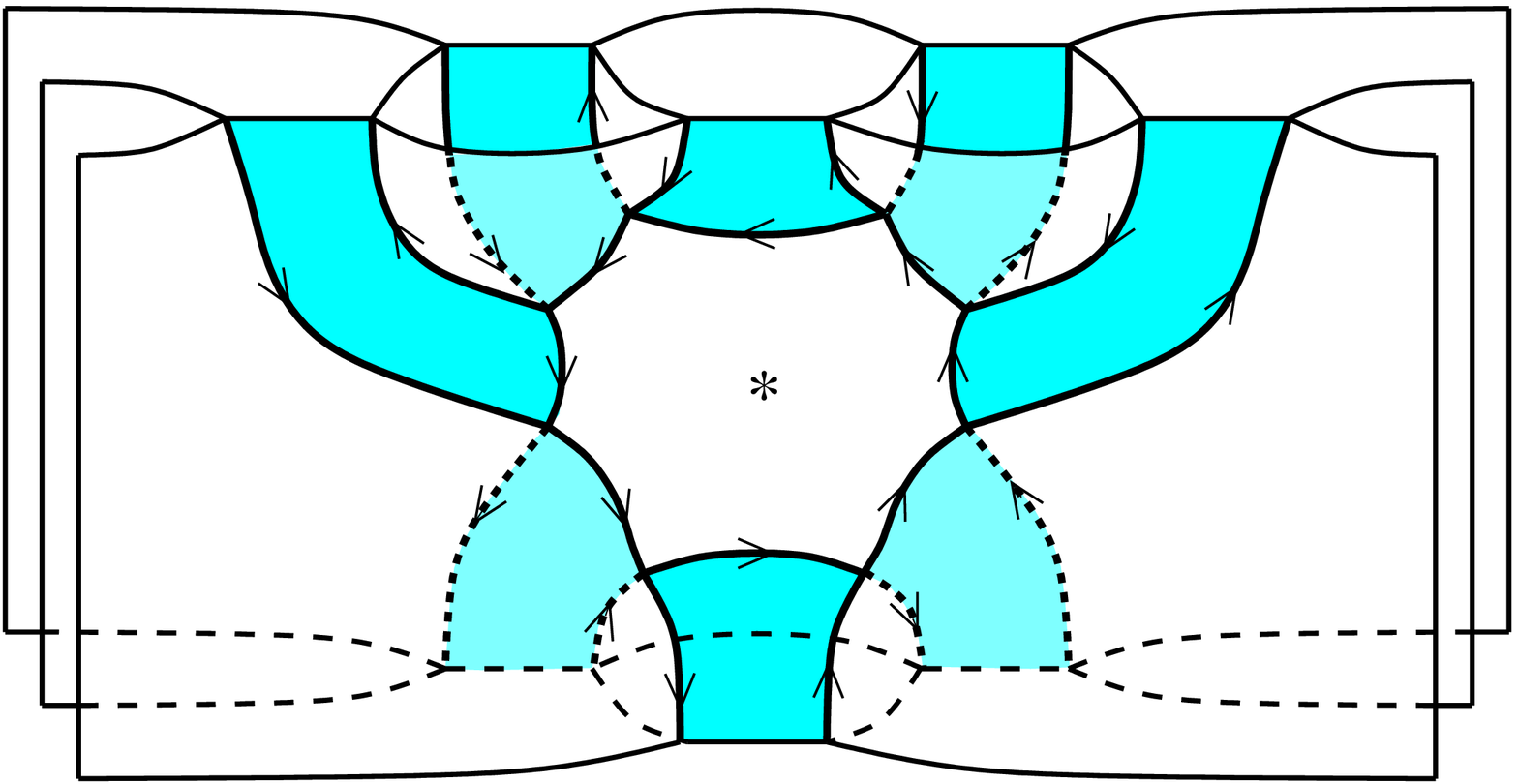}
\vspace*{2ex}
\end{equation}
For the r.h.s. we have
\begin{equation*}
\labellist
\tiny\hair 2pt
\pinlabel $j$   at -45  60 
\pinlabel $j+1$ at -35  28  
\pinlabel $j+2$ at -20  -2
\endlabellist
\figins{-34}{1.0}{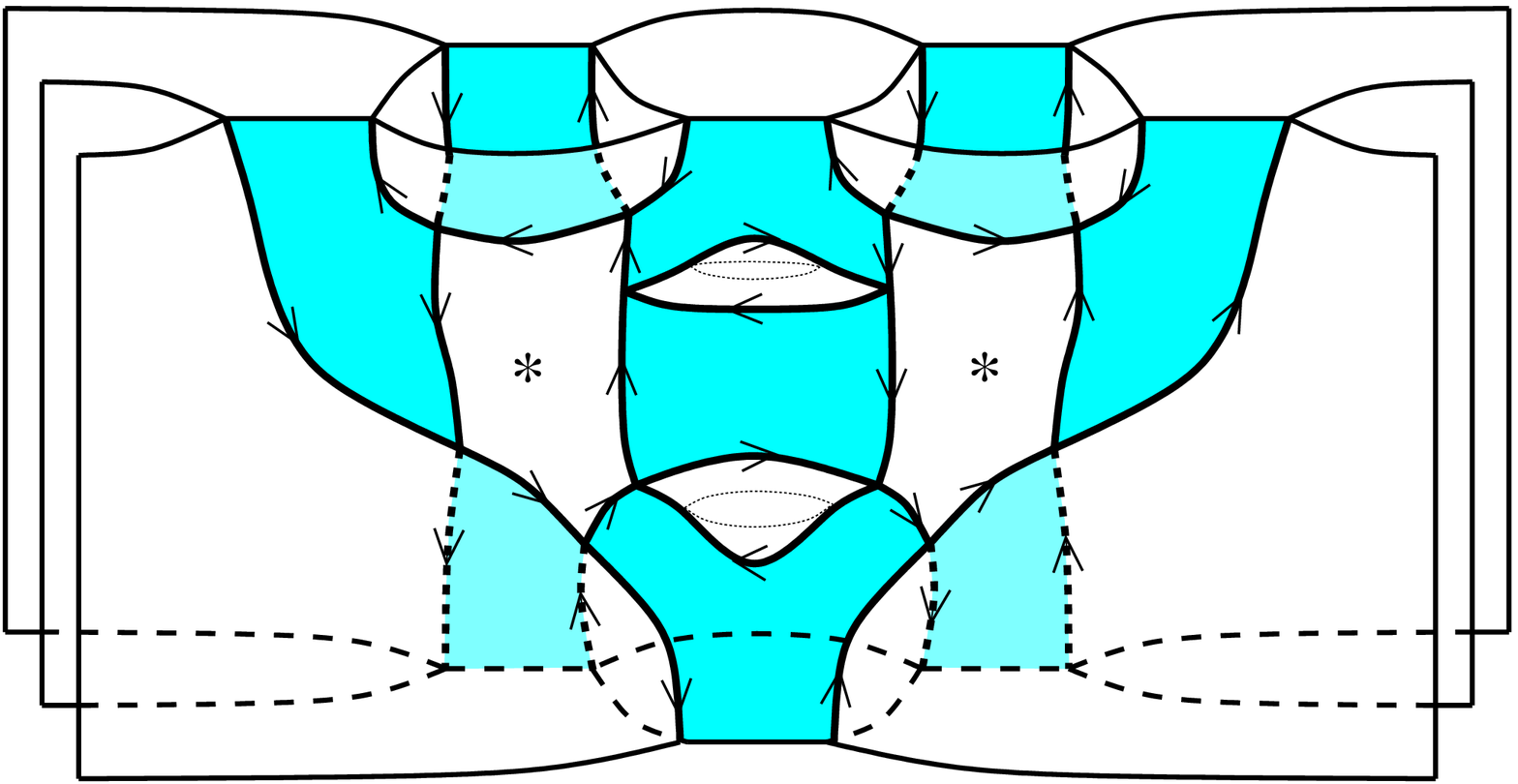}
\vspace*{2ex}
\end{equation*}
Using~\eqref{eq:dr31} on the vertical digon, followed by~\eqref{eq:mp} and 
the Bubble relation~\eqref{eq:bubble-ij}, we obtain~\eqref{eq:lhs}.

Relation~\eqref{eq:slidenext} follows from straightforward computation and is left to the reader.

\medskip
\n\emph{Relations involving three colors:} 
Relations~\eqref{eq:slide6v} and~\eqref{eq:slide4v} follow from isotopies of the foams involved. 
To show that $\Fmv$ respects relation~\eqref{eq:dumbdumbsquare} we use a 
different type of argument. First of all, we note that the images under 
$\Fmv$ of both sides of relation~\eqref{eq:dumbdumbsquare} are multiples of 
each other, because the graded vector space of morphisms in $\foam$ between 
the bottom and top webs has dimension one in degree zero. 
Verifying this only requires computing the 
coefficient of $q^{-(4N-4)}$ (this includes the necessary shift!) in the 
MOY polynomial associated to the 
web 
$$\figins{-64}{2}{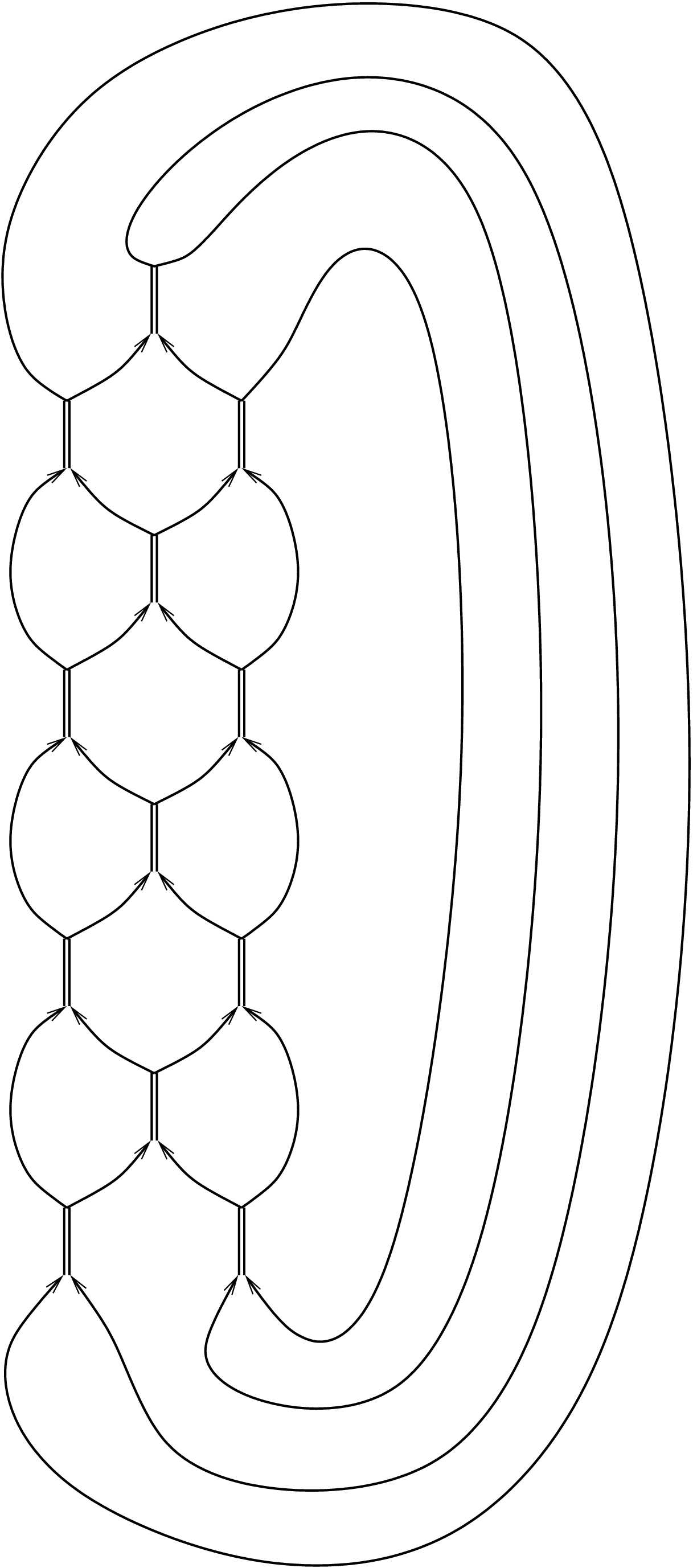}\ ,$$
which is a standard calculation left to the reader. To see that the 
multiplicity coefficient is equal to one, we close both 
sides of relation~\eqref{eq:dumbdumbsquare} simply by putting a dot on each 
open end. Using relations~\eqref{eq:slidedotdist} and~\eqref{eq:dot6v} to 
reduce these closed diagrams, we see that both sides give the same 
non-zero sum of disjoint unions of 
coloured \emph{StartDot-EndDot} diagrams. Note that we have already proved 
that $\Fmv$ respects relations~\eqref{eq:slidedotdist} and~\eqref{eq:dot6v}. 
By applying foam relation~\eqref{eq:Fdot} to the images of all non-zero terms 
in the sum, we obtain a non-zero sum of dotted sheets. This implies that both 
sides of~\eqref{eq:dumbdumbsquare} have the same image under $\Fmv$.
\end{proof}

We have now proved that $\Fmv$ is a monoidal functor for all $N\geq 4$. Our 
main result about the whole family of these functors, i.e. for all $N\geq 4$ 
together, is the proposition below. It implies that all the defining 
relations in Soergel's category can be obtained from the corresponding 
relations between $sl(N)$ foams, when all $N\geq 4$ are considered, and 
that there are no other independent relations in Soergel's category 
corresponding to relations between foams.  

\begin{prop}
\label{prop:faithfulness}
Let $\ii,\jj$ be two arbitrary objects in $\SSoer$ and let 
$f\in\Hom(\ii,\jj)$ be arbitrary. If $\Fmv(f)=0$ for all $N\geq 4$, then 
$f=0$. 
\end{prop}
\begin{proof}
Let us first suppose that $\ii=\jj=\emptyset$. Suppose also that $f$ has 
degree $2d$ and that $N\geq \mbox{max}\{4,d+1\}$. Recall that, as shown in Corollary 3 
in ~\cite{EKr}, we know that 
$\Hom(\emptyset,\emptyset)$ is the free commutative polynomial ring generated 
by the StartDot-EndDots of all possible colors. So $f$ is a polynomial in 
StartDot-EndDots, and therefore a sum of monomials. Let $m$ be one of these 
monomials, no matter which one, and let $m_j$ denote the power 
of the StartDot-EndDot with color $j$ in $m$. Close $\Fmv(f)$ by 
glueing disjoint discs to the boundaries of all open simple facets 
(i.e. the vertical ones with corners in the pictures). 
For each color $j$, put $N-1-m_j$ dots on the left simple open facet 
corresponding to $j$ and also put $N-1$ dots on the rightmost simple open 
facet. 
Note that, after applying~\eqref{eq:rd1}, we get a linear combination of 
dotted simple spheres. Only one term survives and is equal to $\pm 1$, 
because only in that term each sphere has exactly $N-1$ dots. This shows that 
$\Fmv(f)\ne 0$, because it admits a non-zero closure. 

Now let us suppose that $\ii=\emptyset$ and $\jj$ is arbitrary. 
By Corollaries 4.11 and 4.12 in~\cite{EKh}, we know that 
$\Hom(\emptyset,\jj)$ is the free $\Hom(\emptyset,\emptyset)$-module of rank 
one generated by the disjoint union of StartDots coloured by $\jj$. 
Closing off the StartDots by putting dots on all open ends gives an element 
of $\Hom(\emptyset,\emptyset)$, whose image under $\Fmv$ is non-zero for $N$ 
big enough by the above. This shows that the generator of $\Hom(\emptyset,\jj)$ 
has non-zero image under $\Fmv$ for $N$ big enough, because $\Fmv$ is a 
functor.

Finally, the general case, for $\ii$ and $\jj$ arbitrary, can be reduced 
to the previous case by Corollary 4.12 in~\cite{EKh}.   
\end{proof}


\vspace*{1cm}

\noindent {\bf Acknowledgements}
The authors were supported by the Funda\c {c}\~{a}o para a Ci\^{e}ncia e a Tecnologia (ISR/IST plurianual funding) 
through the programme ``Programa Operacional Ci\^{e}ncia, Tecnologia, Inova\-\c{c}\~{a}o'' (POCTI) and the 
POS Conhecimento programme, cofinanced by the European Community fund FEDER. 

Pedro Vaz was also financially supported by the Funda\c {c}\~{a}o para a Ci\^{e}ncia e a Tecnologia through the post-doctoral fellowship SFRH/BPD/46299/ 2008.


\end{document}

%% file: defs.tex
\newcommand{\Fmv}{\mathcal{F}_{N,n}}
\newcommand{\ii}{\underline{\textbf{\textit{i}}}}
\newcommand{\jj}{\underline{\textbf{\textit{j}}}}

\newcommand{\n}{\noindent}
\newcommand{\foam}{\mathbf{Foam}_{N}}

\newcommand{\PF}{\mathbf{Pfoam}}

\newcommand{\SSoer}{\mathcal{SC}_{1,n}}
\newcommand{\SSSoer}{\mathcal{SC}_{2,n}}

\DeclareMathOperator{\sk}{s_\gamma}

\newcommand{\figins}[3] 
{\raisebox{#1pt}{\includegraphics[height=#2 in]{figs/#3}}}

\newtheorem{thm}{Theorem}[section]
\newtheorem{lem}[thm]{Lemma}
\newtheorem{cor}[thm]{Corollary}
\newtheorem{prop}[thm]{Proposition}

\theoremstyle{definition}
\newtheorem{defn}[thm]{Definition}

\newcommand{\bQ}{\mathbb{Q}}

\newcommand{\bC}{\mathbb{C}}

\DeclareMathOperator{\Hom}{Hom}

\catcode`\@=11
\long\def\@makecaption#1#2{%
    \vskip 10pt
    \setbox\@tempboxa\hbox{%
\small{#1: }\ignorespaces #2}%
    \ifdim \wd\@tempboxa >\captionwidth {%
        \rightskip=\@captionmargin\leftskip=\@captionmargin
        \unhbox\@tempboxa\par}%
      \else
        \hbox to\hsize{\hfil\box\@tempboxa\hfil}%
    \fi}
\newdimen\@captionmargin\@captionmargin=2\parindent
\newdimen\captionwidth\captionwidth=\hsize
\catcode`\@=12